%% file: main.tex
\documentclass[12pt]{report}   
\usepackage{graphicx}  
\usepackage{mathrsfs}
\usepackage[letterpaper, margin=3cm]{geometry}
\usepackage{setspace}  
\usepackage{times}  
\usepackage[explicit]{titlesec}  
\usepackage[titles]{tocloft}  
\usepackage[backend=bibtex, sorting=none, bibstyle=alphabetic, citestyle=alphabetic, sorting=nyt]{biblatex}  
\usepackage[bookmarks=false, hidelinks]{hyperref}
\usepackage[page]{appendix}  
\usepackage{rotating}  
\usepackage[normalem]{ulem}  
\usepackage{mypackage} 

\usepackage{afterpage}

\bibliography{references}

\AtEveryBibitem{\clearfield{issn}}
\AtEveryBibitem{\clearlist{issn}}

\AtEveryBibitem{\clearfield{language}}
\AtEveryBibitem{\clearlist{language}}

\AtEveryBibitem{\clearfield{doi}}
\AtEveryBibitem{\clearlist{doi}}

\begin{document}

\onehalfspacing


\input{titlePage.tex}

\currentpdfbookmark{Title Page}{titlePage}  

\pagenumbering{roman}

\newpage\null\thispagestyle{empty}\newpage

\setcounter{page}{3} 

\input{Abstract_Ack.tex}

\newpage\null\thispagestyle{empty}\newpage

\addtocontents{toc}{\cftpagenumbersoff{chapter}} 

\addtocontents{toc}{\cftpagenumberson{chapter}}


\renewcommand{\cftchapdotsep}{\cftdotsep}  
\renewcommand{\cftchapfont}{\bfseries}  
\renewcommand{\cftchappagefont}{}  
\renewcommand{\cftchappresnum}{Chapter }
\renewcommand{\cftchapaftersnum}{:}
\renewcommand{\cftchapnumwidth}{5em}
\renewcommand{\cftchapafterpnum}{\vskip\baselineskip} 
\renewcommand{\cftsecafterpnum}{\vskip\baselineskip}  
\renewcommand{\cftsubsecafterpnum}{\vskip\baselineskip} 
\renewcommand{\cftsubsubsecafterpnum}{\vskip\baselineskip} 

\titleformat{\chapter}[display]
{\normalfont\bfseries\filcenter}{\chaptertitlename\ \thechapter}{0pt}{\MakeUppercase{#1}}

\renewcommand\contentsname{Table of Contents}

\setcounter{tocdepth}{1}

\begin{singlespace}
\tableofcontents
\end{singlespace}

\newpage\null\thispagestyle{empty}\newpage
\addtocounter{page}{-1}

\currentpdfbookmark{Table of Contents}{TOC}

\clearpage


\clearpage
\pagenumbering{arabic}
\setcounter{page}{1} 

\titleformat{\chapter}[display]
{\normalfont\bfseries\filcenter}{\MakeUppercase\chaptertitlename\ \thechapter}{0pt}{\MakeUppercase{#1}}  
\titlespacing*{\chapter}
  {0pt}{0pt}{30pt}	
  
\titleformat{\section}{\large\bfseries}{\thesection}{1em}{#1}

\titleformat{\subsection}{\normalfont\bfseries}{\thesubsection}{1em}{#1}

\titleformat{\subsubsection}{\normalfont\itshape}{\thesubsection}{1em}{#1}


\input{Introduction.tex}

\input{Chapter1.tex}
\input{Chapter2.tex}
\input{Chapter3.tex}
\input{Chapter4.tex}
\input{Chapter5.tex}
\input{Chapter6.tex}

\begin{appendices}
    \addtocontents{toc}{\protect\renewcommand{\protect\cftchappresnum}{\appendixname\space}}
    \addtocontents{toc}{\protect\renewcommand{\protect\cftchapnumwidth}{6em}}
    
    \input{AppendixA.tex}

    \input{AppendixB.tex}

    \clearpage
\end{appendices}

\input{Notation.tex}


\addcontentsline{toc}{chapter}{References}  
\begin{singlespace}  
	\setlength\bibitemsep{\baselineskip}  
	\printbibliography[title={References}]
\end{singlespace}

\end{document}

%% file: titlePage.tex

\newcommand{\thesisTitle}{\Large Defining Newforms in Characteristic $p$}
\newcommand{\yourName}{Daniel Johnston}
\newcommand{\yourSchool}{Mathematics and Statistics}
\newcommand{\yourMonth}{May}
\newcommand{\yourYear}{2020}

\newcommand{\HRule}[1]{\rule{\linewidth}{#1}}


\begin{titlepage}

\begin{center}

\begin{singlespacing}
\phantom{XX}
\vspace{4\baselineskip}

\textbf{\thesisTitle}\\
\vspace{3\baselineskip}
By\\
\vspace{3\baselineskip}
{\Large \yourName}\\
\vspace{3\baselineskip}
Supervised by \\
Dr. Alex Ghitza \\
\vspace{3\baselineskip}
Submitted In Partial Fulfillment\\
of the Requirements for the Degree\\
Master of Science in the\\
School of \yourSchool\\
\vspace{3\baselineskip}
{\Large University of Melbourne}\\
\vspace{\baselineskip}
\yourMonth{} \yourYear{}
\vfill

\end{singlespacing}

\end{center}

\end{titlepage}

%% file: Abstract_Ack.tex
\clearpage

\begin{centering}
\textbf{ABSTRACT}\\
\vspace{\baselineskip}
\end{centering}
The theory of newforms, due to Atkin and Lehner \cite{ao1970hecke}, provides a powerful method for decomposing spaces of modular forms. However, many problems occur when trying to generalise this theory to characteristic $p$. Recently, Deo and Medvedovsky \cite{deo2019newforms} have suggested a way around these problems by using purely algebraic notions to define newforms. In this thesis, we describe the methods of Deo and Medvedovsky in detail and generalise their results where possible.

\vspace{2cm}

\begin{centering}
\textbf{ACKNOWLEDGEMENTS}\\
\vspace{\baselineskip}
\end{centering}
First and foremost, I would like to thank my supervisor Alex Ghitza. Not only did Alex recommend this thesis topic to me but he provided constant encouragement and support throughout the research project that I am very thankful for.\\[8pt]
I also thank my parents, Sally and Robert, for always supporting my decision to study mathematics, and my high school teacher, Rachel Eyles, for without her influence I may have never ended up pursuing number theory.\\[8pt]
Finally, I would like to thank all of my friends in the maths department who made this degree a lot of fun. A special mention goes out to Dan, Fenella, Kate, Lukas, Madeleine, Tom and Yao who also assisted me in proofreading this thesis.

\clearpage

%% file: Introduction.tex
\chapter*{Introduction}
\addcontentsline{toc}{chapter}{Introduction}

Modular forms are an indispensable tool in number theory. In recent years, they have garnered widespread attention due to their role in Wiles' proof of Fermat's Last Theorem (\cite{wiles1995modular}, \cite{taylor1995ring}) and the subsequent proof of the modularity conjecture \cite{breuil2001modularity}. This thesis is concerned with defining an important class of modular forms known as \emph{newforms}.\\
\\
For our purposes, a \emph{modular form} is a function on the complex upper half plane satisfying certain holomorphy and transformation conditions. These conditions allow us to write each modular form $f$ as a Fourier series
\begin{equation*}
    f(z)=\sum_{n=0}^\infty a_nq^n,\qquad q=e^{2\pi i z}.
\end{equation*}
Of particular interest are the values of the coefficients $a_n$, which encode valuable arithmetic information. A special class of modular forms we will study are \emph{cusp forms}, which have $a_0=0$ in their Fourier expansions.\\
\\
Associated to each non-zero modular form is a weight $k$ and a level $N$. The vector space of all modular forms (resp. cusp forms) of weight $k$ and level $N$ is typically denoted $\M_k(N)$ (resp. $\mcS_k(N)$). In this thesis, we are interested in a decomposition of $\mcS_k(N)$ into "old" and "new" subspaces:
\begin{equation}\label{introdecomp}
    \mcS_k(N)=\mcS_k(N)^{\old}\oplus\mcS_k(N)^{\new}.
\end{equation}
The space of \emph{oldforms}, $\mcS_k(N)^{\old}$, consists of cusp forms that arise from lower levels strictly dividing $N$. On the other hand, the space of \emph{newforms}, $\mcS_k(N)^{\new}$, is the orthogonal complement of $\mcS_k(N)^{\old}$ with respect to a suitable inner product. The theory of oldforms and newforms was first developed by Atkin and Lehner \cite{ao1970hecke} and later extended by Li \cite{li1975newforms}.\\
\\
One can also consider modular forms with coefficients over an arbitrary commutative ring $B$. To do this, we let $\M_k(N,\Z)$ denote the space of modular forms with Fourier coefficients in $\Z$ and then define $\M_k(N,B):=\M_k(N,\Z)\otimes_\Z B$. We also define $\mcS_k(N,B)$ analogously for cusp forms. If $B$ has characteristic zero, the theory of newforms for $\mcS_k(N,B)$ naturally generalises the classical theory of Atkin and Lehner. However, when $B$ has characteristic $p$, congruences between oldforms and newforms result in many undesirable properties. For instance, we no longer have a decomposition as in \eqref{introdecomp}.\\
\\
Recently, Deo and Medvedovsky \cite{deo2019newforms} have provided suggestions on how to deal with this behaviour in characteristic $p$. In particular, they suggest redefining newforms with robust, purely algebraic definitions. In \cite{deo2019newforms}, two such algebraic definitions are explored. These definitions are related to \emph{Hecke operators} and are inspired by classical observations of Serre \cite{serre1973formes} and Atkin and Lehner \cite{ao1970hecke}.\\
\\
The goal of this thesis is to provide an extensive analysis of Deo and Medvedovsky's algebraic definitions of newforms. In particular, we fill in the details of proofs in \cite{deo2019newforms} and suggest ways to extend their theory where possible. An outline of the thesis is as follows.

\begin{itemize}
    \item In Chapter \ref{chap1} we discuss the basic theory of modular forms and establish the notation that will be used throughout the thesis.
    \item In Chapter \ref{chapclassic} we summarise the classical theory of newforms due to Atkin and Lehner.
    \item In Chapter \ref{chapsquarefree} we introduce the algebraic definitions of newforms due to Deo and Medvedovsky. Moreover, we show that they agree with Atkin and Lehner's definitions in the classical setting.
    \item In Chapter \ref{chap4} we define newforms over an arbitrary commutative ring $B$. We then describe the properties of these forms when $B$ has characteristic zero.
    \item In Chapter \ref{chap5} we explore the theory of newforms in characteristic $p$ and analyse Deo and Medvedovsky's definitions in this setting.
    \item In Chapter \ref{chap6} we introduce the notion of modular forms with character and suggest ways to generalise Deo and Medvedovsky's definitions to this setting. The results in this chapter constitute the most original contribution of this thesis.
\end{itemize}
We will assume that the reader is familiar with basic definitions and results from complex analysis and abstract algebra. The later chapters also use a large amount of commutative algebra. The relevant background material for these chapters can be found in Appendix \ref{appcommalg}.

%% file: Chapter1.tex
\chapter{Modular Forms}\label{chap1}

We begin by providing some preliminary definitions and examples of modular forms. In particular, we will define modular forms in terms of their "level", allowing the notion of newforms to arise naturally. For more details one can consult chapters 1.1 and 1.2 of Diamond and Shurman's book \cite{diamond2005first}.

\section{Modular Forms of Level 1}\label{secmodform1}
Let 
\begin{equation*}
    \mathbb{H}=\{z\in\C\::\:\Im(z)>0\}
\end{equation*}
denote the complex upper half plane and 
\begin{equation*}
    \SL_2(\Z)=\{\gamma\in M_{2\times 2}(\Z)\::\:\det(\gamma)=1\}
\end{equation*}
be the set of $2\times 2$ matrices with integer entries and determinant 1.\\
\\
When studying the geometry of $\mathbb{H}$ it is common to consider the following action of $\SL_2(\Z)$ on $\H$:
\begin{equation}\label{saction}
    \gamma\cdot z=\frac{az+b}{cz+d},\quad\text{with }\gamma=\begin{pmatrix}a&b\\c&d\end{pmatrix}\in\SL_2(\Z).
\end{equation}
One can check that this is a well-defined action in the sense that $\gamma\cdot z\in\H$, $\begin{smatrix}1&0\\0&1\end{smatrix}\cdot z=z$ and $(\gamma\gamma')\cdot z=\gamma\cdot(\gamma'\cdot z)$ for all $z\in\H$ and $\gamma,\gamma'\in\SL_2(\Z)$.\\
\\
Modular forms are  holomorphic functions on $\H$ that satisfy certain symmetries with respect to the action \eqref{saction}. Such functions exhibit interesting arithmetic properties and are defined as follows.

\begin{definition}\label{defmodforms}
    Let $k$ be a nonnegative integer. A (level 1) \emph{modular form} of weight $k$ is a function $f\colon\mathbb{H}\to\C$ satisfying the following three conditions:
    \begin{enumerate}[label=\arabic*.]
        \item  $f$ is holomorphic on $\H$,
        \item $f\left(\gamma\cdot z\right)=(cz+d)^kf(z)$ for all $\gamma=\begin{smatrix}a&b\\c&d\end{smatrix}\in\SL_2(\Z)$,
        \item $f$ is holomorphic at infinity. That is, the limit $\lim_{\Im(z)\to\infty}f(z)$ exists.
    \end{enumerate}
    The set of all modular forms of weight $k$ forms a vector space of functions, denoted $\mathcal{M}_k(\SL_2(\Z))$. 
\end{definition}

From here onwards, we shall refer to the second condition as the \emph{modularity condition}. To simplify notation let 
\begin{equation*}
    j(\gamma,z):=cz+d
\end{equation*}
be the \emph{factor of automorphy} for $\gamma=\begin{smatrix}a&b\\c&d\end{smatrix}\in\SL_2(\Z)$ and $z\in\H$. For any $\gamma\in\SL_2(\Z)$ we then define the \emph{slash operator $|_k\gamma$} of weight $k$ on functions $f:\mathbb{H}\to\C$ by
\begin{equation}\label{slash}
    (f|_k\gamma)(z):=j(\gamma,z)^{-k}f(\gamma\cdot z).
\end{equation}
The modularity condition then simply states that $f|_k\gamma=f$ for all $\gamma\in\SL_2(\Z)$. The slash operator is also an action in the sense that $f|_k\begin{smatrix}1&0\\0&1\end{smatrix}=f$ and $f|_k\gamma\gamma'=f|_k\gamma|_k\gamma'$ for all functions $f:\mathbb{H}\to\C$.\\
\\
Now, substituting $\gamma=\begin{smatrix}1&1\\0&1\end{smatrix}\in\SL_2(\Z)$ into the modularity condition tells us that modular forms are $\Z$-periodic, i.e.\ $f(z+1)=f(z)$. As a result, all modular forms $f$ have a Fourier series (or "$q$-expansion")
\begin{equation}\label{fourier}
    f(q)=\sum_{n\in\Z}a_nq^n,\qquad\text{where }q=e^{2\pi i z}.
\end{equation}
The fact that $\lim_{\Im(z)\rightarrow\infty}f(z)$ exists means that $a_n=0$ for all $n<0$.
\begin{definition}
    A \emph{cusp form} is a modular form $f$ such that $\lim_{\Im(z)\to\infty}f(z)=0$. Equivalently, a cusp form has $a_0=0$ in its $q$-expansion \eqref{fourier}.
\end{definition}

We now look at some examples.

\begin{example}\label{exmodforms}\
    \begin{itemize}
        \item All constant functions on $\H$ are modular forms of weight $0$. The function $f(z)=0$ is a modular form (and cusp form) for every weight $k\geq 0$. 
        \item The \emph{Eisenstein series}
        \begin{equation*}
            G_k(z)=\sum_{(c,d)\in\Z^2\setminus\{(0,0)\}}\frac{1}{(cz+d)^k}
        \end{equation*}
        is a modular form of weight $k$ for all even $k\geq 4$. Its $q$-expansion is given by
        \begin{equation*}
            G_k(q)=2\zeta(k)\left(1-\frac{2k}{B_k}\sum_{n=1}^\infty\sigma_{k-1}(n)q^n\right)
        \end{equation*}
        where $\zeta$ is the Riemann zeta function, $B_k$ is the $k^{\text{th}}$ Bernoulli number and\\
        $\sigma_{k-1}(n)=\sum_{d\mid n}d^{k-1}$. Scaling by a factor of $1/(2\zeta(k))$ gives the simpler series
        \begin{equation}\label{eisenstein}
            E_k(q)=1-\frac{2k}{B_k}\sum_{n=1}^\infty\sigma_{k-1}(n)q^n.
        \end{equation}
        Throughout the rest of this thesis, we will refer to any scaling of $G_k(q)$ as an Eisenstein series.
        \item The \emph{modular discriminant}
        \begin{equation*}
            \Delta(z)=\frac{1}{1728}\left(E_4(z)^3-E_6(z)^2\right)
        \end{equation*}
        is a cusp form of weight $12$ and has the $q$-expansion
        \begin{equation*}
            \Delta(q)=q-24q^2+252q^3-1472q^4+4830q^5-6048q^6-16744q^7+\cdots.
        \end{equation*}
        In the next chapter we shall see how the theory of newforms reveals nontrivial relationships between these Fourier coefficients. There are still many open problems regarding $\Delta(q)$. For instance, it is unknown whether the $n^{\text{th}}$ Fourier coefficient of $\Delta(q)$ is nonzero for all $n\geq 1$.
        \item The only modular form of odd weight is $f(z)=0$. To see this, suppose that $f$ is a modular form of odd weight $k$. Substituting $\gamma=\begin{smatrix}-1&0\\0&-1\end{smatrix}\in\SL_2(\Z)$ into the modularity condition for $f$ then gives $f(z)=(-1)^kf(z)=-f(z)$, which forces $f(z)=0$.
    \end{itemize}
\end{example}
    
\section{Modular Forms of Level $N$}
Studying the level 1 modular forms defined in the previous section eventually becomes quite restrictive. In fact, one can show that all modular forms of level 1 can be expressed as polynomials in the Eisenstein series $E_4(z)$ and $E_6(z)$ \cite[Proposition 3.6]{kilford2008modular}.\\
\\
By replacing $\SL_2(\Z)$ in the modularity condition with some subgroup $\Gamma\leqslant\SL_2(\Z)$, we are left with a larger set of functions to study. The subgroups that interest us the most are those based on systems of congruences.
\begin{definition}
    Let $N$ be a positive integer. The \emph{principal congruence subgroup of level $N$} is
    \begin{equation*}
        \Gamma(N)=\left\{\begin{pmatrix}a&b\\c&d\end{pmatrix}\in\SL_2(\Z)\::\:\begin{pmatrix}a&b\\c&d\end{pmatrix}\equiv\begin{pmatrix}1&0\\0&1\end{pmatrix}\pmod{N}\right\},
    \end{equation*}
    where the matrix congruence is interpreted entrywise.
\end{definition}
\begin{definition}\label{defcong}
    A \emph{congruence subgroup} is any subgroup $\Gamma\leqslant\SL_2(\Z)$ such that $\Gamma(N)\subseteq\Gamma$ for some positive integer $N$. The smallest such $N$ is called the \emph{level} of $\Gamma$.
\end{definition}
For a given level $N\in\Z_{>0}$, the most widely studied congruence subgroups are
\begin{align}\label{gamma0}
    \Gamma_0(N)&=\left\{\begin{pmatrix}a&b\\c&d\end{pmatrix}\in\SL_2(\Z)\::\:\begin{pmatrix}a&b\\c&d\end{pmatrix}\equiv\begin{pmatrix}*&*\\0&*\end{pmatrix}\pmod{N}\right\},\ \text{and}\\[8pt]
    \Gamma_1(N)&=\left\{\begin{pmatrix}a&b\\c&d\end{pmatrix}\in\SL_2(\Z)\::\:\begin{pmatrix}a&b\\c&d\end{pmatrix}\equiv\begin{pmatrix}1&*\\0&1\end{pmatrix}\pmod{N}\right\}
\end{align}
where $*$ means unspecified. Note that $\Gamma(N)\subseteq\Gamma_1(N)\subseteq\Gamma_0(N)\subseteq\SL_2(\Z)$.\\
\\
We now alter our original definition of modular forms (Definition \ref{defmodforms}) to the following.
\begin{definition}\label{defmodformscong}
    Let $\Gamma$ be a congruence subgroup of $\SL_2(\Z)$ and $k$ be an integer. A modular form of weight $k$ with respect to $\Gamma$ is a function $f\colon\mathbb{H}\rightarrow\C$ satisfying the following three conditions:
    \begin{enumerate}[label=\arabic*.]
        \item  $f$ is holomorphic on $\H$,
        \item $f|_k\gamma=f$ for all $\gamma=\begin{smatrix}a&b\\c&d\end{smatrix}\in\Gamma$,
        \item $f|_k\alpha$ is holomorphic at infinity for all $\alpha\in\SL_2(\Z)$.
    \end{enumerate}
\end{definition}
The first two conditions in Definition \ref{defmodformscong} are natural generalisations of the first two conditions in our original definition of modular forms (Definition \ref{defmodforms}). The third condition is a less obvious generalisation but is required for the space of modular forms to be finite dimensional. In the case where $\Gamma=\SL_2(\Z)$, note that the third condition agrees with Definition \ref{defmodforms} since we always have $f|_k\alpha=f$ by the modularity condition.\\
\\
If $f$ is a modular form with respect to a congruence subgroup $\Gamma$ of level $N$, we say that $f$ is a \emph{modular form of level $N$}. Since $\Gamma(1)=\SL_2(\Z)$, this justifies why the modular forms from Section \ref{secmodform1} were called level 1 modular forms.\\
\\
Another subtlety in defining modular forms with respect to a congruence subgroup $\Gamma$ is the transfer over to Fourier series. Since $\Gamma$ may not contain $\begin{smatrix}1&1\\0&1\end{smatrix}$, it does not necessarily follow that a modular form $f$ with respect to $\Gamma$ is $\Z$-periodic. However, we still have $\begin{smatrix}1&h\\0&1\end{smatrix}\in\Gamma$ for some minimal $h\in\Z_{>0}$ giving $f(z+h)=f(z)$ and thus a Fourier series of the form
\begin{equation}\label{fourier2}
    f(z)=\sum_{n\in\Z}a_nq_h^n,\qquad\text{where }q_h=e^{2\pi i z/h}.
\end{equation}
However, this is not something to worry about in the case of $\Gamma=\Gamma_0(N)$ or $\Gamma=\Gamma_1(N)$ since both of these congruence subgroups contain $\begin{smatrix}1&1\\0&1\end{smatrix}$.\\
\\
The definition of a cusp form also carries over from the level 1 case in a similar way.
\begin{definition}
    A modular form $f$ (of weight $k$ with respect to $\Gamma$) is a cusp form if $a(0)=0$ in the $q$-expansion \eqref{fourier2} of $f|_k\alpha$ for all $\alpha\in\SL_2(\Z)$. 
\end{definition}
As in the case of level 1, the set of modular forms of weight $k$ with respect to $\Gamma$ forms a vector space $\M_k(\Gamma)$. The set of cusp forms also forms a vector subspace denoted by $\mcS_k(\Gamma)$. These spaces are finite dimensional and their dimensions can be computed explicitly \cite[Chapter 3]{diamond2005first}.
\begin{definition}
    The \emph{algebra of modular forms} is given by $\M(\Gamma)=\sum_{k=0}^\infty \M_k(\Gamma)$. Similarly, the \emph{algebra of cusp forms} is $\mcS(\Gamma)=\sum_{k=0}^\infty\mcS_k(\Gamma)$. 
\end{definition}
In particular, any $f\in \M(\Gamma)$ is a finite linear combination $f=f_1+\cdots+f_n$ where each $f_i$ occurs in a fixed weight. Note that $\M(\Gamma)$ is closed under multiplication since if $f$ is a modular form of weight $k$ and $g$ is a modular form of weight $k'$, then $fg$ is a modular form of weight $k+k'$.\\
\\
We now finish with some examples of level $N$ modular forms.
\newpage
\begin{example}\label{levelnexamples}\
    \begin{itemize}
        \item By generalising the Eisenstein Fourier series in \eqref{eisenstein} to $k=2$ we obtain
        \begin{equation*}
            E_2(q)=1-\frac{4}{B_2}\sum_{n=1}^\infty\sigma(n)q^n.
        \end{equation*}
        Note that this is \textbf{not} a (level 1) modular form. In particular, the modularity condition is not satisfied. However the function
        \begin{equation*}
            E_2^N(z):=E_2(z)-NE_2(Nz)
        \end{equation*}
        is a modular form of weight 2 for $\Gamma_0(N)$.
        \item The \emph{Jacobi theta function} is given by
        \begin{equation*}
            \theta(q)=\sum_{n\in\Z}q^{n^2}.
        \end{equation*}
        This function plays a crucial role in analytic number theory and is used to find the number of ways to express an integer as a sum of squares. In terms of modular forms, if $k$ is an even integer then  $\theta^k$ is a modular form of weight $k/2$ for $\Gamma_1(4)$.
        \item For some congruence subgroup $\Gamma$ and odd weight $k$, if $\begin{smatrix}-1&0\\0&-1\end{smatrix}\notin\Gamma$ then it is possible to have $\mathcal{M}_k(\Gamma)~\neq~\{0\}$. For instance, the function $\theta^2$ from the previous example is in $\M_1(\Gamma_1(4))$.
        \item Let $\Gamma\leqslant\SL_2(\Z)$ be a congruence subgroup. Since a modular form for $\Gamma$ must satisfy fewer symmetries than a modular form for $\SL_2(\Z)$, we have $\M_k(\SL_2(\Z))\subseteq\M_k(\Gamma)$. That is, any modular form with respect to $\SL_2(\Z)$ is also a modular form with respect to $\Gamma$.
    \end{itemize}
\end{example}
To see where the idea of newforms arises, we consider $\Gamma_0(M)$ and $\Gamma_0(N)$ with $M|N$. From the definition of $\Gamma_0$ we have $\Gamma_0(N)\subseteq\Gamma_0(M)$ and thus  $\M_k(\Gamma_0(M))\subseteq\M_k(\Gamma_0(N))$. So for fixed $N$, we can take any divisor $M|N$ and embed $\M_k(\Gamma_0(M))$ in $\M_k(\Gamma_0(N))$. Later we shall see that there are also other ways to embed $\M_k(\Gamma_0(M))$ in $\M_k(\Gamma_0(N))$.\\
\\
The space of \emph{oldforms} consists of forms in $\M_k(N)$ arising from lower levels. The corresponding space of \emph{newforms} then consists of forms that do not arise from lower levels. These ideas will be made precise in the next chapter.

%% file: Chapter2.tex
\chapter{Classical Theory of Newforms}\label{chapclassic}
In this chapter we summarise the original theory of oldforms and newforms first established by Atkin and Lehner \cite{ao1970hecke}. Following Atkin and Lehner, we will primarily concern ourselves with the congruence subgroup $\Gamma_0(N)$ defined in the previous chapter. As a consequence, we simplify our notation, writing $\mathcal{M}_k(N)$ to mean $\mathcal{M}_k(\Gamma_0(N))$ (or $\mcS_k(N)$ for $\mcS_k(\Gamma_0(N))$ in the case of cusp forms).\\[8pt]
The following theory can also be extended to other classes of modular forms (see \cite{li1975newforms} and \cite[Chapter 5]{diamond2005first}). However to prevent any overcomplications, we will refrain from discussing such generalisations until Chapter \ref{chap6}.\\[8pt]
For the rest of this chapter, we assume that $k\in\Z_{\geq 0}$ and $N\in\Z_{>0}$ are fixed integers.

\section{Atkin-Lehner and Hecke Operators}\label{secHecke}
In order to study newforms, we need to define a collection of important linear operators on $\M_k(N)$. We begin by noting that the slash operator can be extended to matrices in 
\begin{equation*}
    \GL_2(\Q)^{+}=\{\gamma\in M_{2\times 2}(\Q)\::\:\det(\gamma)>0\}.
\end{equation*} 
Namely, for a function $f\colon\H\to\C$ and $\gamma=\begin{smatrix}a&b\\c&d\end{smatrix}\in\GL_2(\Q)^{+}$ we let
\begin{equation*}
    (f|_k\gamma)(z)=(\det\gamma)^{k/2}j(\gamma,z)^{-k}f(\gamma\cdot z),
\end{equation*}
where $j(\gamma,z)=cz+d$ and $\gamma\cdot z=\frac{az+b}{cz+d}$ as in the $\SL_2(\Z)$ case. The factor of $\det(\gamma)^{k/2}$ is a convention chosen so that scalar matrices act as the identity (provided $k$ is even). The Atkin-Lehner operator is then defined as follows.
\begin{definition}\label{atkinlehnerdef}
    Let $\ell$ be a prime that divides $N$ exactly once and  $\gamma_\ell\in\GL_2(\Q)^{+}$ be any matrix of the form $\begin{smatrix}\ell&a\\N&\ell b\end{smatrix}$ where $a$ and $b$ are integers chosen so that $\ell b-a(N/\ell)=1$. The \emph{Atkin-Lehner operator} $w_\ell$ is the map that sends $f\in \M_k(N)$ to $f|_k\gamma_\ell$.
\end{definition}
\newpage
\begin{remarks}\
    \begin{enumerate}[label=(\roman*)]
        \item The conditions on $a$ and $b$ imply that $\det(\gamma_\ell)=\ell$.
        \item $w_\ell$ does not depend on the choice of $\gamma_\ell$. That is, if we let $\gamma_\ell=\begin{smatrix}\ell&a\\N&\ell b\end{smatrix}$ and $\gamma_\ell'=\begin{smatrix}\ell&a'\\N&\ell b'\end{smatrix}$ then
        \begin{equation*}
            \gamma_\ell(\gamma_\ell')^{-1}=\begin{pmatrix}\ell&a\\N&\ell b\end{pmatrix}\frac{1}{\ell}\begin{pmatrix}\ell b'&-a'\\-N&\ell\end{pmatrix}=\begin{pmatrix}\ell b'-a(N/\ell)&-a'+a\\N b'-N b&-a'(N/\ell)+\ell b\end{pmatrix}\in\Gamma_0(N),
        \end{equation*}
        so that for $f\in\M_k(N)$, we have $w_\ell (f)=f|_k\gamma_\ell=f|_k\gamma_\ell'$ as required.
    \end{enumerate}
\end{remarks}
Next we define the Hecke operators. These operators can be explicitly described in terms of their Fourier coefficients. In what follows, we write $a_n(f)$ for the $n^{th}$ Fourier coefficient of a modular form $f$. That is,
\begin{equation*}
    f(q)=\sum_{n=0}^\infty a_n(f)q^n.
\end{equation*}
\begin{definition}\label{hecke}
    Let $p$ be a prime number. For a modular form $f$, the \emph{Hecke operator} $T_p$ acting on $f$ is given by the $q$-expansion with $n^{\text{th}}$ coefficient
    \begin{alignat*}{3}
        a_n(T_p f)&=a_{np}(f),\qquad&&\text{if $p\mid N$,}\\
        a_n(T_p f)&=a_{pn}(f)+p^{k-1}a_{n/p}(f),\qquad&&\text{if $p\nmid N$,}
    \end{alignat*}
    where $a_{n/p}(f)=0$ if $p$ does not divide $n$. For a prime power $p^s$ with $s\geq 2$ we let $T_1$ act as the identity and then inductively define 
    \begin{equation*}
        T_{p^s}=T_{p}T_{p^{s-1}}-p^{k-1}T_{p^{s-2}}.
    \end{equation*}
    Finally, to define $T_m$ for any positive integer $m$ we let $T_{mm'}~=~T_{m}T_{m'}$ whenever $\gcd(m,m')=~1$.
\end{definition}
For $m\geq 0$, we will also write $U_m$ for the formal $\C$-linear operator $U_m\colon\C[[q]]\to\C[[q]]$ given by $a_n(U_mf)=a_{mn}(f)$. If $m\mid N$ then $U_m$ coincides with $T_m$ and we will generally write $U_m$ in place of $T_m$ in such situations.\\
\\
The following proposition lists some algebraic properties of these operators. Proofs of the following results can be found in sections 2 and 3 of \cite{ao1970hecke}.
\begin{proposition}\label{heckeproperties}
    Let $p$ and $p'$ be distinct primes not dividing $N$ and let $\ell$ be a prime dividing $N$ exactly $\alpha$ times. Let $f$ be a modular form of weight $k$ for $\Gamma_0(N)$. Then
    \begin{enumerate}[label=(\alph*)]
        \item $T_p(f)$ and $U_\ell(f)$ are in $\M_k(N)$.
        \item If $\alpha=1$ then $w_\ell(f)\in\M_k(N)$. Moreover, $w_\ell$ is an involution, i.e. $w_\ell^2(f)=f$.
        \item $T_pT_{p'}(f)=T_{p'}T_p(f)$,\\
        $T_pU_\ell(f)=U_\ell T_p(f)$ and if $\alpha=1$,\\
        $T_pw_\ell(f)=w_\ell T_p(f)$.
        \item If $\alpha=1$ then $U_\ell(f)+\ell^{\frac{k}{2}-1}w_\ell(f)\in\M_k(N/\ell)$.
        \item If $\alpha>1$, then $U_\ell(f)\in\M_k(N/\ell)$.
    \end{enumerate}
\end{proposition}
\begin{remark}
    The above proposition also holds if we restrict to cusp forms.
\end{remark}

\section{Definition of Newforms}\label{sec22}
Let $M$ be a positive integer dividing $N$ and $e\mid(N/M)$. We consider the map
\begin{align}\label{defembedding}
    i_e\colon\M_k(M)&\to\M_k(N)\\
    f&\mapsto(z\mapsto f(ez))\notag.
\end{align}
\begin{proposition}
    The map $i_e$ is well defined.
\end{proposition}
\begin{proof}
    We need to show that $g(z):=f(ez)$ is in $\M_k(N)$. The holomorphicity of $f$ carries over to $g$ so it suffices to check that $g$ satisfies the modularity condition $g|_k\gamma=g$ for any $\gamma\in\Gamma_0(N)$. So, let $\gamma=\begin{smatrix}a&b\\c&d\end{smatrix}$ be an arbitrary matrix in $\Gamma_0(N)$. We have
    \begin{align*}
        (g|_k\gamma)(z)&=(cz+d)^{-k}g(\gamma\cdot z)\\[8pt]
        &=(cz+d)^{-k}g\left(\frac{az+b}{cz+d}\right)\\[8pt]
        &=(cz+d)^{-k}f\left(e\frac{az+b}{cz+d}\right)\\[8pt]
        &=(cz+d)^{-k}f\left(\frac{a(ez)+be}{\frac{c}{e}(ez)+d}\right)\\[8pt]
        &=(cz+d)^{-k}\left(\frac{c}{e}(ez)+d\right)^kf(ez)\qquad(*)\\[8pt]
        &=g(z),
    \end{align*}
    as required. Note that $(*)$ follows since $\begin{smatrix}a&be\\c/e&d\end{smatrix}\in\Gamma_0(M)$. In particular, $N\mid c$ and $e\mid(N/M)$ implies that $(c/e)\mid M$.
\end{proof}
\begin{remark}
    The above arguments also hold if we restrict $i_e$ to a map over cusp forms.
\end{remark}
The map $i_e$ thus gives an embedding of $\M_k(M)$ into $\M_k(N)$. Using this embedding, we make the following definition.
\begin{definition}
    The space of \emph{oldforms} $\M_k(N)^{\old}$ is the $\C$-linear subspace of $\M_k(N)$ spanned by the maps $i_e\colon\M_k(M)\to\M_k(N)$ for all $M$ strictly dividing $N$ and $e|(N/M)$. More explicitly,
    \begin{equation*}
        \M_k(N)^{\old}=\sum_{\substack{M\mid N\\M\neq N\\e\mid (N/M)}}i_e(\M_k(M)).
    \end{equation*}
\end{definition}

Now that we have a formal notion of "old", we can describe what it means to be "new". Following Atkin and Lehner \cite{ao1970hecke} we restrict our attention to cusp forms. In this direction let
\begin{equation*}
    \mcS_k(N)^{\old}=\sum_{\substack{M\mid N\\M\neq N\\e\mid (N/M)}}i_e(\mcS_k(M)).
\end{equation*}
We now consider the \emph{Petersson inner product} $\langle\cdot,\cdot\rangle_N\colon\mcS_k(N)\times\mcS_k(N)\to\C$, which is given by
\begin{equation}\label{petersson}
    \langle f,g\rangle_N:=\int_{z\in D_N}\!f(z)\overline{g(z)}y^k\frac{\mathrm{d}x\mathrm{d}y}{y^2},
\end{equation}
where $z=x+iy$ and $D_N$ is a fundamental domain for the action of $\Gamma_0(N)$ on the upper half plane $\H$. One can check that this is a well-defined inner product on $\mcS_k(N)$ and the integral \eqref{petersson} does not depend on the fundamental domain chosen \cite[Chapter 5.4]{diamond2005first}. Moreover, for a prime $p\nmid N$ the Hecke operators $T_p$ are Hermitian with respect to the Petersson inner product \cite[Lemma 13]{ao1970hecke}. In other words,
\begin{equation}\label{hermitian}
    \langle T_pf,g\rangle_N=\langle f,T_pg\rangle_N.
\end{equation}
\begin{definition}
    The space of \emph{newforms} in $S_k(N)$ is the orthogonal complement of $S_k(N)^{\old}$ with respect to the Petersson inner product. That is,
    \begin{equation*}
        S_k(N)^{\new}=\{f\in S_k(N)\::\:\langle f,g\rangle_N=0\text{ for all }g\in S_k(N)_{\old}\}.
    \end{equation*}
\end{definition}
Since $\mcS_k(N)$ is finite dimensional, it then follows that 
\begin{equation*}
    \mcS_k(N)=\mcS_k(N)^{\old}\oplus\mcS_k(N)^{\new}.
\end{equation*}
In particular, $\mcS_k(N)^{\old}\cap\mcS_k(N)^{\new}=\{0\}$. To help identify oldforms and newforms we will make use of the following theorem (adapted from \cite{ao1970hecke}).
\begin{theorem}\label{maintheorem}
    Let $f(q)=\sum_{n=1}^\infty a_nq^n\in\mcS_k(N)$ and suppose that $a_n=0$ whenever $\gcd(n,N)=~1$. Then $f$ is an oldform.
\end{theorem}
\begin{proof}
    See \cite[Theorem 1]{ao1970hecke} for the original proof or \cite{carlton1999result} for a more recent proof involving representation theory.
\end{proof}

\section{Properties of Newforms}\label{propsec}
We now prove some statements regarding oldforms and newforms. Namely, we are interested in how newforms interact with the operators defined in Section \ref{secHecke}. For the most part we will follow Section 4 of \cite{ao1970hecke} whilst adopting more modern notation.\\
\\
We begin by utilising the Hermitian property of $T_p$ for primes $p\nmid N$ (see \eqref{hermitian}). In particular, we have the following result.
\begin{lemma}
    The space of oldforms $\mcS_k(N)^{\old}$ and newforms $\mcS_k(N)^{\new}$ are preserved under the action of the Hecke operators $T_p$ for $p\nmid N$.
\end{lemma}
\begin{proof}
    For oldforms, the statement follows since $T_p$ commutes with the $i_e$ embedding maps. That is,  if $f=\sum_{n=1}^\infty a_n(f)q^n$ is a cusp form then
    \begin{equation*}
        T_p(i_ef)=i_e(T_pf)=\sum_{n=1}^\infty[a(pn)+p^{2k-1}a(n/p)]q^{ne}.
    \end{equation*}
    For newforms, consider an arbitrary $f\in\mcS_k(N)^{\new}$. Then for any $p\nmid N$ and $g~\in~\mcS_k(N)^{\old}$ we have
    \begin{equation*}
        \langle T_p(f),g\rangle_N=\langle f,T_p(g)\rangle_N=0,
    \end{equation*}
    since $T_p(g)\in S_k(N)^{\old}$. As a consequence, $T_p(f)\in\mcS_k(N)^{\new}$ as required.
\end{proof}
In light of this result, we now analyse how the Hecke operators behave on the space of newforms.  From here onwards, we say $f\in\mcS_k(N)$ is an \emph{eigenform} if $f$ is an eigenfunction of all the Hecke operators $T_p$ with $p\nmid N$. This means that an eigenform is also an eigenfunction for any of the Hecke operators $T_m$ with $\gcd(m,N)=1$.
\begin{lemma}\label{basislemma}
    There exists a basis of eigenforms for $\mcS_k(N)^{\new}$.
\end{lemma}
\begin{proof}
    Since the $T_p$ operators are Hermitian (for $p\nmid N$) and commute with one another (Proposition \ref{heckeproperties}), the result follows by a version of the spectral theorem (see, for example \cite[Theorem 11 of Chapter IX in Volume 1]{gantmakher1959theory}).
\end{proof}
\begin{remark} 
    Lemma \ref{basislemma} also holds for $\mcS_k(N)$ or $\mcS_k(N)^{\old}$ by the same reasoning.
\end{remark}
\begin{lemma}
    For any eigenform $f(q)=\sum_{n=1}^\infty a_n(f)q^n\in\mcS_k(N)^{\new}$, we have $a_1(f)\neq 0$.
\end{lemma}
\begin{proof}
    Suppose for a contradiction that $a_1(f)=0$. Let $m$ be a positive integer such that $\gcd(m,N)=1$. Since $f$ is an eigenform there exists $\lambda_m\in\C$ such that $T_mf=\lambda_mf$. Hence,
    \begin{equation*}
        a_m(f)=a_1(T_m(f))=\lambda_ma_1(f)=0.
    \end{equation*}
    Applying Theorem \ref{maintheorem} then gives $f\in\mcS_k(N)^{\old}$ and thus $f=0$ since $\mcS_k(N)^{\new}$ and $\mcS_k(N)^{\old}$ intersect trivially. This contradicts $f$ being an eigenform so we must have $a_1(f)\neq 0$ as required.
\end{proof}
The previous lemma implies that we can normalise the eigenforms $f\in\mcS_k(N)^{\new}$ to have $a_1(f)=1$.
\begin{definition}
    A \emph{primitive form} (of weight $k$ for $\Gamma_0(N)$) is an eigenform $f\in\mcS_k(N)^{\new}$ with $a_1(f)=1$. 
\end{definition}
Let $f$ be a primitive form. Writing $\lambda_p(f)$ for the $T_p$ eigenvalue of $f$ (so that $T_pf~=~\lambda_p(f)f$) gives
\begin{equation}\label{heckevalueeqn}
    a_{pn}(f)+p^{k-1}a_{n/p}(f)=\lambda_p(f)a_n(f).
\end{equation}
Thus, by setting $n=1$,
\begin{equation}\label{tpevalue}
    \lambda_p(f)=a_p(f).
\end{equation}
This demonstrates how elegantly newforms behave under the action of the Hecke operators. In fact, substituting \eqref{tpevalue} into \eqref{heckevalueeqn} gives
\begin{equation}\label{inductivehecke}
    a_{np}(f)-a_{n}(f)a_{p}(f)+p^{k-1}a_{n/p}(f)=0
\end{equation}
which we can use to inductively compute $a_{m}$ when $\gcd(m,N)=1$, provided that we know the Hecke eigenvalues $\lambda_p(f)=a_p(f)$ for each prime $p\nmid N$. We now define an equivalence relation to further analyse this behaviour.
\begin{definition}\label{eqdef}
    Let $f$ and $g$ be eigenforms in $\mcS_k(N)$. We write $f\sim g$ if $f$ and $g$ have the same $T_p$ eigenvalues for all $p\nmid N$.
\end{definition}
\begin{lemma}\label{primitiveproof}
    If $f_1$ and $f_2$ are primitive forms in $\mcS_k(N)^{\new}$ with $f_1\sim f_2$ then $f_1=f_2$. 
\end{lemma}
\begin{proof}
    Since $f_1$ is equivalent to $f_2$, \eqref{inductivehecke} tells us that $a_m(f_1)=a_m(f_2)$ whenever $\gcd(m,N)=~1$. Now, $f_1-f_2\in\mcS_k(N)^{\new}$ and applying Theorem \ref{maintheorem} gives $f_1-f_2\in\mcS_k(N)^{\old}$. As $\mcS_k(N)^{\new}\cap\mcS_k(N)^{\old}=\{0\}$ it follows that $f_1=f_2$ as required.
\end{proof}

\begin{lemma}\label{oldisnew}
    We have
    \begin{equation*}
        \mcS_k(N)=\sum_{M\mid N}\sum_{e\mid(N/M)}i_e(\mcS_k(M)^{\new}),
    \end{equation*}
    so that any eigenform $f\in\mcS_k(N)$ is equivalent (in the sense of Definition \ref{eqdef}) to a new eigenform $h\in S_k(M)^{\new}$ where $M$ is a divisor of $N$.
\end{lemma}
\begin{proof}
    We proceed by induction on the factors of $N$. In level 1 there are no lower levels so $\mcS_k(1)=\mcS_k(1)^{\new}$. Now assume that the result holds for all proper divisors of $N$. Let $f\in\mcS_k(N)$, and write $f=f_{\new}+f_{\old}$ where $f_{\new}\in\mcS_k(N)^{\new}$ and $f_{\old}\in\mcS_k(N)^{\old}$. By definition $f_{\old}$ is a linear combination of forms $i_e(g)$ with $g\in\mcS_k(M)$ for some $M$ strictly dividing $N$ and $e\mid(N/M)$. But then, by the induction hypothesis, any such $g$ can be written as the sum of forms $i_d(h)$ where $h\in\mcS_k(M')^{\new}$ with $M'\mid M\mid N$ and $d\mid (M/M')$ .\\[8pt]
    For the statement about equivalent forms first recall that each space $\mcS_k(M)^{\new}$ has a basis of eigenforms. So, noting that the $T_p$ and $i_e$ operators commute (for $p\nmid N$) we have that any eigenform in $\mcS_k(N)$ must be equivalent to an eigenform in $\mcS_k(M)^{\new}$ for some $M\mid N$.
\end{proof}

\begin{lemma}\label{distincteigen}\
    Let $f\in\mcS_k(N)^{\old}$ and $g\in\mcS_k(N)^{\new}$ be eigenforms. Then $f$ and $g$ belong to different equivalence classes.
\end{lemma}
\begin{proof}
    Suppose for a contradiction that $f\sim g$. By Lemma \ref{oldisnew}, $f\sim h$ (and thus $g\sim h$) for some new eigenform $h\in\mcS_k(M)^{\new}$ where $M\mid N$. Since multiplication of $g$ or $h$ by a scalar does not affect their $T_p$ eigenvalues we may assume $a_1(g)=a_1(h)=1$. Applying the argument in Lemma \ref{primitiveproof} we then have $g-h\in\mcS_k(N)^{\old}$. Now $h=i_1(h)\in\mcS_k(N)^{\old}$ so it follows that $g\in\mcS_k(N)^{\old}$. But as $g$ is a newform this implies that $g=0$, a contradiction.
\end{proof}
We now move onto the main theorem (adapted from \cite[Theorem 3]{ao1970hecke}), which completely describes newforms in terms of the Atkin-Lehner and Hecke operators. To prove the theorem, we will make repeated use of the previous lemma along with the results listed in Proposition \ref{heckeproperties}. 
\begin{theorem}\label{hecketheorem}
    Let $f\in\mcS_k(N)^{\new}$ be a primitive form, $p$ be any prime with $p\nmid N$, and $\ell$ be any prime dividing $N$ exactly $\alpha$ times. If $f(q)=q+\sum_{n=2}^\infty a_nq^n$ we then have
    \begin{enumerate}[label=(\alph*)]
        \item $T_p(f)=a_pf$,
        \item $U_\ell(f)=a_\ell f$ and
        \item if $\alpha=1$ then $w_\ell(f)=\varepsilon_\ell f$, where $\varepsilon_\ell=\pm 1$.
    \end{enumerate}
    Further, if $\alpha\geq 2$, then $a_\ell=0$. While if $\alpha=1$, we have $a_\ell=-\ell^{\frac{k}{2}-1}\varepsilon_\ell$. In terms of Fourier coefficients, (a) and (b) become
    \begin{enumerate}[label=(\alph*)']
        \item $a_{np}-a_na_p+p^{k-1}a_{n/p}=0$,
        \item $a_{n\ell}-a_na_\ell=0$.
    \end{enumerate}
\end{theorem}
\begin{proof}
    First we recall that the $U_\ell$ and $T_p$ operators commute so that $U_\ell(f)\sim f$. Now, if $\alpha\geq 2$ we have $U_\ell (f)\in S_k(N/\ell)$ by Proposition \ref{heckeproperties} and thus $U_\ell(f)\in\mcS_k(N)^{\old}$. So by Lemma \ref{distincteigen}, $U_\ell f\equiv 0$. Moreover, $w_\ell$ commutes with $T_p$ so that $w_\ell(f)\sim f$. However, if $w_\ell(f)\in\mcS_k(N)^{\old}$ then $w_\ell(f)=0$ meaning that $f=w_\ell^2(f)=0$, a contradiction. So we must have $w_\ell(f)~\in~\mcS_k(N)^{\new}$ by Lemma \ref{distincteigen}. Because $w_\ell(f)\sim f$, Lemma \ref{primitiveproof} then implies that $w_\ell f$ is a scalar multiple of $f$, i.e.\ $w_\ell f=\varepsilon_\ell f$. We know that $\varepsilon_\ell=\pm 1$ since $w_\ell^2(f)=f$.\\[8pt]
    Now if instead $\alpha=1$, Proposition \ref{heckeproperties} gives $f\sim U_\ell(f)+\ell^{\frac{k}{2}-1}w_\ell(f)\in\mcS_k(N)^{\old}$. Hence by Lemma \ref{distincteigen}, $U_\ell(f)+\ell^{\frac{k}{2}-1}w_\ell(f)=0$. That is, $U_\ell f=-\ell^{\frac{k}{2}-1}\varepsilon_\ell f$ as required. The translation to Fourier series follows directly from the definition of the Hecke operators.
\end{proof}
\begin{example}
    We consider the modular discriminant 
    \begin{equation*}
        \Delta(q)=q-24q^2+252q^3-1472q^4+4830q^5-6048q^6-16744q^7+\cdots
    \end{equation*}
    from Section \ref{secmodform1}. Recall that $\Delta(q)$ is a cusp form of weight 12 and level 1. Since $\Delta(q)$ is a level 1 cusp form it cannot arise from any lower levels and is thus a newform. As a consequence, its Fourier coefficients satisfy the relations described in Theorem \ref{hecketheorem}. In particular, if $\tau(n)$ denotes the $n^{\text{th}}$ Fourier coefficient of $\Delta(q)$ then:
    \begin{align*}
        \tau(mn)&=\tau(m)\tau(n)\text{ if }\gcd(m,n)=1,\ \text{and}\\
        \tau(p^{r+1})&=\tau(p)\tau(p^r)-p^{11}\tau(p^{r-1})\text{ for $p$ prime and $r>0$}.
    \end{align*}
    These relations were first conjectured by Ramanujan \cite{ramanujan1916certain}.
\end{example}
\begin{example}
    The space $\mcS_2(45)^{\new}$ is spanned by the primitive form\footnote{This was computed using the software system \href{https://www.sagemath.org/}{SageMath}.}
    \begin{equation*}
        f=q+q^2-q^4-q^5-3q^8+O(q^{10}).
    \end{equation*}
    In this example, $N=45$, and 3 divides 45 twice whereas 5 divides $45$ once. Hence by Theorem \ref{hecketheorem}, we expect $U_3(f)=0$ and $U_5(f)=\pm 1$, which certainly agrees with the $q$-expansion above.
\end{example}
These examples demonstrate how studying newforms can reveal nontrivial data about the coefficients of modular forms. There are many other useful reasons to study newforms, such as their use in classifying cusp forms and their relation to Dirichlet $L$-functions \cite[\S 5.9]{diamond2005first}. How to generalise this theory to other types of modular forms (besides cusp forms on $\Gamma_0(N)$) is thus a very natural and widely studied question. 

%% file: Chapter3.tex
\chapter{Newforms in Squarefree Level}\label{chapsquarefree}

In this chapter we look at the equivalent ways of defining newforms suggested by Deo and Medvedovsky \cite{deo2019newforms}. To do this, we will restrict our attention to the space of ``$\ell\text{-newforms}$" and focus on the case where $N$ is a squarefree integer.\\[8pt]
The definitions in this chapter favour more algebraic notions, avoiding the use of the Petersson inner product. Following \cite{deo2019newforms} we will explore two such algebraic notions. The first involves classifying newforms by their $U_\ell$-eigenvalue. The second notion relates newforms to the kernel of a "trace" operator on $\M_k(N)$, inspired by an observation of Serre \cite[\S 3.1]{serre1973formes}. By defining newforms in such a way, further generalisations become easier to define and work with.

\section{The Space of $\ell$-newforms}\label{sec31}
Let $N>1$ be a fixed integer. Recall that the space of weight $k$-oldforms for $\Gamma_0(N)$ is defined as
\begin{equation*}
    \M_k(N)^{\old}=\sum_{\substack{M\mid N\\M\neq N\\e\mid (N/M)}}i_e(\M_k(M)).
\end{equation*}
Now, each $M$ strictly dividing $N$ is also a divisor of $N/\ell$ for some prime $\ell\mid N$. Hence we only need to consider forms arising from the levels $N/\ell$ for each $\ell\mid N$. In other words,
\begin{equation}\label{ellold}
    \M_k(N)^{\old}=\sum_{\substack{\ell\text{ prime}\\\ell\mid N}}\M_k(N)^{\ell-\old},
\end{equation}
where
\begin{equation*}
    \M_k(N)^{\ell-\old}=\M_k(N/\ell)+i_\ell(\M_k(N/\ell)).
\end{equation*}
Here we slightly abuse notation, writing $\M_k(N/\ell)$ when we actually mean $i_1(\M_k(N/\ell))$. For a detailed proof of the equality in \eqref{ellold}, see Proposition \ref{ellolddecomp} in Appendix \ref{appassort}.\\
\\
Now, similarly let $\mcS_k(N)^{\ell-\old}=\mcS_k(N/\ell)+i_\ell(\mcS_k(N/\ell))$ and define $\mcS_k(N)^{\ell-\new}$ to be the orthogonal complement of $\mcS_k(N)^{\ell-\old}$. We then have $\mcS_k(N)=\mcS_k(N)^{\ell-\old}\oplus\mcS_k(N)^{\ell-\new}$ and
\begin{equation*}
    \mcS_k(N)^{\new}=\bigcap_{\substack{\ell\text{ prime}\\\ell\mid N}}\mcS_k(N)^{\ell-\new}.
\end{equation*}
It therefore suffices to study $\mcS_k(N)^{\ell-\old}$ and $\mcS_k(N)^{\ell-\new}$ for each prime $\ell\mid N$ in order to determine $\mcS_k(N)^{\old}$ and $\mcS_k(N)^{\new}$ respectively. As with $\mcS_k(N)^{\old}$ and $\mcS_k(N)^{\new}$, these $\ell$-old and $\ell$-new subspaces have bases of eigenforms. The following result of Atkin-Lehner can then be used to give us further information about $\mcS_k(N)^{\ell-\old}$ and $\mcS_k(N)^{\ell-\new}$.

\begin{theorem}\label{infinitetheorem}
    Let $f\in\mcS_k(M_1)^{\new}$ and $g\in\mcS_k(M_2)^{\new}$ be primitive forms with $a_p(f)=a_p(g)$ for infinitely many primes $p$. Then $M_1=M_2$ and $f=g$.
\end{theorem}
\begin{proof}
    See \cite[Theorem 4]{ao1970hecke}.
\end{proof}

\begin{corollary}\label{newdecomp}
    Let $\ell$ be a prime dividing $N$ exactly $\alpha$ times. Then,
    \begin{equation*}
        \mcS_k(N)^{\ell-\new}=\sum_{\substack{M\\ \ell^{\alpha}\mid M\mid N}}\sum_{e\mid(N/M)}i_e(\mcS_k(M)^{\new}).
    \end{equation*}
\end{corollary}
\begin{proof}
    Recall from Lemma \ref{oldisnew} that
    \begin{equation}\label{cuspdecomp}
        \mcS_k(N)=\sum_{M\mid N}\sum_{e\mid(N/M)}i_e(\mcS_k(M)^{\new}).
    \end{equation}
    Then, since $\mcS_k(N)^{\ell-\old}=\mcS_k(N/\ell)+i_\ell(\mcS_k(N/\ell))$, any $\ell$-old form arises from a newform at a level $M$ with $\ell^{\alpha}\nmid M$. In other words,
    \begin{equation}\label{oldcuspdecomp}
        \mcS_k(N)^{\ell-\old}=\sum_{\substack{M\\ \ell^{\alpha}\nmid M\mid N}}\sum_{e\mid(N/M)}i_e(\mcS_k(M)^{\new}).
    \end{equation}
    Comparing this with \eqref{cuspdecomp} gives
    \begin{equation*}
        \mcS_k(N)^{\ell-\new}\subseteq\sum_{\substack{M\\ \ell^{\alpha}\mid M\mid N}}\sum_{e\mid(N/M)}i_e(\mcS_k(M)^{\new}).
    \end{equation*}
    For the reverse inclusion, suppose that $M$ is a divisor of $N$ such that $\ell\mid M$. By Theorem \ref{infinitetheorem} and \eqref{oldcuspdecomp} any eigenform in $i_e(\mcS_k(M)^{\new})$ cannot be contained in $\mcS_k(N)^{\ell-\old}$ and must therefore be in $\mcS_k(N)^{\ell-\new}$ as required.
\end{proof}

\begin{corollary}\label{lolddistincteigen}\
    Let $f\in\mcS_k(N)^{\ell-\old}$ and $g\in\mcS_k(N)^{\ell-\new}$ be eigenforms. Then $f$ and $g$ belong to different equivalence classes. That is, there exists some prime $p\nmid N$ such that $T_p(f)\neq T_p(g)$.
\end{corollary}
\begin{proof}
    As shown in the proof of the previous corollary, $\ell$-new and $\ell$-old forms arise from newforms at different levels. Hence the eigenforms $f$ and $g$ must belong to different equivalence classes by Theorem \ref{infinitetheorem}.
\end{proof}

\begin{corollary}\label{coruleigen}
    Let $f\in\mcS_k(N)^{\ell-\new}$ be an eigenform where $\ell$ is a prime dividing $N$ exactly $\alpha$ times. If $\alpha=1$ then $w_\ell(f)=\varepsilon_\ell f$ and $U_\ell(f)=-\ell^{\frac{k}{2}-1}\varepsilon_\ell f$ for some $\varepsilon_\ell=\pm 1$. Otherwise if $\alpha\geq 2$ then $U_\ell(f)=0$
\end{corollary}
\begin{proof}
    Using Corollary \ref{lolddistincteigen} repeat the argument from the proof of Theorem \ref{hecketheorem}.
\end{proof}

\section{The Atkin-Lehner Operator Revisited}
For the rest of this chapter, we assume that $\ell$ is a fixed prime dividing $N$ exactly once. Note that if $N$ is squarefree then this is always true for every prime dividing $N$. Now, since $\ell$ divides $N$ exactly once we can always define the Atkin-Lehner operator $w_\ell$. Following Deo and Medvedovsky \cite[Section 3.4]{deo2019newforms} we will scale $w_\ell$ by defining $W_\ell=\ell^{k/2}w_\ell$.  This scaling of $w_\ell$ will be particularly important when we generalise to characteristic $p$ in Chapter \ref{chap5}. Note that since $w_\ell$ is an involution $W_\ell^2=\ell^k$.\\
\\
We now show how to conveniently define $\ell$-oldforms in terms of $W_\ell$.
\begin{proposition}\label{wlscaling}
    If $f\in\M_k(N/\ell)\subseteq\M_k(N)$ then $W_\ell f=\ell^ki_\ell f$.
\end{proposition}
\begin{proof}
    We have
    \begin{align*}
        (W_\ell f)(z)&=\ell^{k/2}\left(f|_k\begin{pmatrix}\ell&a\\N&\ell b\end{pmatrix}\right)(z)\\[10pt]
        &=\frac{\ell^k}{(N z+\ell b)^k}f\left(\frac{\ell z+a}{N z+\ell b}\right)\\[10pt]
        &=\frac{\ell^k}{(N z+\ell b)^k}f\left(\frac{(\ell z)+a}{\frac{N}{\ell}(\ell z)+\ell b}\right)\\[10pt]
        &=\frac{\ell^k}{(N z+\ell b)^k}(N z+\ell b)^kf(\ell z)\qquad\text{(Since $\begin{smatrix}1&a\\N/\ell&\ell b\end{smatrix}\in\Gamma_0(N/\ell)$)}\\[10pt]
        &=\ell^k(i_\ell f)(z),
    \end{align*}
    as required.
\end{proof}
Hence $W_\ell$ is just a rescaling of $i_\ell$ on $\M_k(N/\ell)$. As a result
\begin{equation*}
    \M_k(N)^{\ell-\old}=\M_k(N/\ell)+W_\ell(\M_k(N/\ell)).
\end{equation*}
\begin{proposition}\label{constint}
    If $f\in \M_k(N/\ell)\cap W_\ell( \M_k(N/\ell))$, then $f$ is constant.
\end{proposition}
\begin{proof}
    Let $f\in\M_k(N/\ell)\cap W_\ell( \M_k(N/\ell))$ so that $f=W_\ell(g)$ for some $g\in \M_k(N/\ell)$. We have
    \begin{equation*}
        \ell^kg=W_\ell^2(g)=W_\ell(f)=\ell^k i_\ell( f)
    \end{equation*}
    and thus
    \begin{equation}\label{fislkf}
        f=W_\ell(g)=\ell^ki_\ell(g)=\ell^k i_\ell^2(f).
    \end{equation}
    If $f=\sum_{n=0}^\infty a_nq^n$ then \eqref{fislkf} becomes
    \begin{equation}\label{qisql}
        \sum_{n=0}^\infty a_nq^n=\ell^k\sum_{n=0}^\infty a_nq^{n\ell^2}.
    \end{equation} 
    Now, suppose for a contradiction that $f$ is not constant. Let $n>0$ be the least integer with $a_n\neq 0$. Then the left-hand side of \eqref{qisql} has a $q^n$ term whereas the right-hand side of \eqref{qisql} does not. This is impossible so $f$ must be constant as required. 
\end{proof}

In Appendix \ref{appassort} (Proposition \ref{directoverc}) we see that the algebra $\M(N)=\sum_{k=0}^\infty\M_k(N)$ is in fact a direct sum $\M(N)=\bigoplus_{k=0}^\infty\M_k(N)$. In particular, constant-valued modular forms only occur in weight $0$. Combining this with Proposition \ref{constint} gives the following result.
\begin{corollary}\label{linindepwl}
    Let $k>0$. Then $\M_k(N/\ell)\cap W_\ell(\M_k(N/\ell))=\{0\}$ and as a result, 
    \begin{equation*}
        \M_k(N)^{\ell-\old}=\M_k(N/\ell)\oplus W_\ell(\M_k(N/\ell)).
    \end{equation*}
\end{corollary}

\section{First Algebraic Notion: $U_\ell$ Eigenvalue}\label{ulsect}
We have shown (Corollary \ref{coruleigen}) that $\mcS_k(N)^{\ell-\new}$ has a basis of eigenforms $\{f_1,f_2,\dots,f_n\}$ such that   
\begin{equation*}
    U_\ell(f_i)=\pm\ell^{\frac{k}{2}-1}f_i.
\end{equation*}
In light of this property, we define the following operator.
\begin{definition}
    The $\D_\ell$ operator is the linear map $\D_\ell\colon\M_k(N)~\to~\M_k(N)$ defined by
    \begin{equation*}
        \D_\ell=\ell^2U_\ell^2-\ell^k,
    \end{equation*}
    so that
    \begin{equation*}
        \ker(\D_\ell)=\{f\in\M_k(N)\::\:U_\ell^2(f)=\ell^{k-2}f\}.
    \end{equation*}
\end{definition}
We claim that when $\D_\ell$ is restricted to cusp forms, this kernel is precisely the $\ell$-newforms.
\begin{theorem}\label{dlcusptheorem}
    We have
    \begin{equation*}
        \mcS_k(N)^{\ell-\new}=\ker(\D_\ell|_{\mcS_k(N)}).
    \end{equation*}
\end{theorem}
To prove this theorem, we will need some preliminary lemmas.
\begin{lemma}\label{weillemma}
    If $g\in\mcS_k(N/\ell)$ is an eigenform of weight $k$ then $|a_{\ell}(g)|~<~(\ell~+~1)\ell^{\frac{k-2}{2}}$, where $a_\ell(g)$ is the $T_\ell$ eigenvalue of $g$.
\end{lemma}
\begin{proof}
    This relies on the Weil bound \cite[Theorem 8.2]{deligne1974conjecture} which states that $|a_{\ell}(g)|\leq 2\ell^{\frac{k-1}{2}}$. It thus suffices to show that $2\ell^{\frac{k-1}{2}}<(\ell+1)\ell^{\frac{k-2}{2}}$. Squaring both sides of this inequality and dividing by $\ell^{k-2}$, we see that this is equivalent to $4\ell<(\ell+1)^2$, which is always true for $\ell>1$.
\end{proof}
\begin{remark}
    The Weil bound is only stated for primitive forms in \cite{deligne1974conjecture}. However, any eigenform $g\in\mcS_k(N/\ell)$ has the same $T_\ell$-eigenvalue as some primitive form in a lower level (Lemma \ref{oldisnew}).  
\end{remark}
\begin{lemma}\label{charpolynomial}
    For $k>0$, let $g\in \M_k(N/\ell)$ be an eigenform and consider the two dimensional subspace, $V_{\ell,g}$, spanned by $g$ and $W_{\ell}g$. The characteristic polynomial of $U_{\ell}^2$ on $V_{\ell,g}$ is $P_{\ell,g}(X)=X^2-(a_{\ell}^2(g)-2\ell^{k-1})X+\ell^{2k-2}$, where $a_{\ell}(g)$ is the $T_{\ell}$-eigenvalue of $g$. 
\end{lemma}
\begin{proof}
    First note that $g$ and $W_\ell g$ are linearly independent by Corollary \ref{linindepwl}. We need to find out how $U_{\ell}$ acts on $g$ and $W_\ell(g)$. By the definition of the Hecke operators of $\M_k(N/\ell)$,
    \begin{equation}\label{heckerelations}
        T_{\ell}=U_{\ell}+\ell^{-1}W_\ell,
    \end{equation}
    noting that $(W_\ell f)(q)=\ell^kf(q^\ell)$ by Proposition \ref{wlscaling}. Evaluating \eqref{heckerelations} at $g$ and rearranging then gives
    \begin{equation*}
        U_{\ell}(g)=a_{\ell}(g)g-\ell^{-1}W_\ell(g).
    \end{equation*}
    On the other hand, $U_\ell(i_\ell g)=g$ and thus $U_{\ell}(W_\ell(g))=\ell^kg$. Putting everything together, the transformation matrix of $U_{\ell}$ with respect to the basis $\{g,W_\ell(g)\}$ is
    \begin{equation*}
        U_\ell=
        \begin{pmatrix}
            a_{\ell}(g)&\ell^k\\
            -\ell^{-1}&0
        \end{pmatrix}.
    \end{equation*}
    The transformation matrix for $U_\ell^2$ is then
    \begin{equation*}
        U_\ell^2=
        \begin{pmatrix}
            a_\ell^2(g)-\ell^{k-1}&\ell^ka_\ell(g)\\
            -\ell^{-1}a_\ell(g)&-\ell^{k-1}
        \end{pmatrix}
    \end{equation*}
    which has the desired characteristic polynomial. 
\end{proof}
We are now ready to prove Theorem \ref{dlcusptheorem}.
\begin{proof}[Proof of Theorem \ref{dlcusptheorem}]
    From Corollary \ref{coruleigen} it follows that any $f\in\mcS_k(N)^{\ell-\new}$ satisfies $U_\ell^2(f)=\ell^{k-2}f$ so that $\mcS_k(N)^{\ell-\new}\subseteq\ker(\D_\ell|_{\mcS_k(N)})$. It thus suffices to show that no $\ell$-old forms have a $U_\ell^2$ eigenvalue of $\ell^{k-2}$.\\
    \\
    Note that there are no nonzero cusp forms of weight 0 \cite[Theorem 3.5.1]{diamond2005first} so we may assume $k>0$. Now, $\mcS_k(N/\ell)$ has a basis of eigenforms so that $\mcS_k(N)^{\ell-\old}$ is the direct sum of the two-dimensional subspaces $V_{\ell,g}$ described in Lemma \ref{charpolynomial}. These spaces are $U_{\ell}$-invariant so we deal with them separately. By Lemma \ref{charpolynomial}, the characteristic polynomial of $U_\ell^2$ on $V_{\ell,g}$ is 
    \begin{equation*}
        P_{\ell,g}(X)=X^2-(a_{\ell}^2(g)-2\ell^{k-1})X+\ell^{2k-2}.
    \end{equation*}
    Therefore, if $\lambda_1$ and $\lambda_2$ denote the $U_\ell^2$ eigenvalues on $V_{\ell,g}$,
    \begin{equation*}
        \lambda_1+\lambda_2=a_\ell^2(g)-2\ell^{k-1}
        \ \text{and}\ \lambda_1\lambda_2=\ell^{2k-2}.
    \end{equation*} 
    Suppose for a contradiction that one of these eigenvalues is $\ell^{k-2}$. Then the other eigenvalue must be $\ell^k$ and hence \begin{equation}\label{eqal2}
        a_{\ell}^2(g)=\ell^{k-2}+2\ell^{k-1}+\ell^k=\ell^{k-2}(\ell+1)^2,
    \end{equation}
    which contradicts the Weil bound (Lemma \ref{weillemma}).
\end{proof}
\section{Second Algebraic Notion: Kernel of Trace Operator}\label{tlsect}
For our next algebraic definition of newforms we look at how to project $\mcS_k(N)$ onto $\mcS_k(N/\ell)$. Fortunately there is a natural way to do this using the following operator.
\begin{definition}\label{tracedefinition}
    For any weight $k\geq 0$, the \emph{trace operator} from level $N$ to level $N/\ell$ is given by
    \begin{align*}
        \Tr_\ell\colon\M_k(N)&\to\M_k(N/\ell)\\[8pt]
        f&\mapsto\sum_{\gamma\in\Gamma_0(N)\backslash\Gamma_0(N/\ell)}f|_k\gamma.
    \end{align*}
See Proposition \ref{trwelldefined} in the appendix as for why this operator is well defined.
\end{definition}
By explicitly finding coset representatives for $\Gamma_0(N)\backslash\Gamma_0(N/\ell)$, one obtains the following formula for $\Tr_\ell$ \cite[Lemma 2.2]{mcgraw2003modular}:
\begin{equation*}
    \Tr_{\ell}(f)=f+\ell^{1-k/2}U_\ell\omega_{\ell}f.
\end{equation*}
In terms of the scaled $W_\ell$ operator; this is:
\begin{equation}\label{explicittrace}
    \Tr_\ell(f)=f+\ell^{1-k} U_\ell W_\ell f.
\end{equation}
If $f\in\M_k(N/\ell)$ we can simplify this expression further.
\begin{lemma}\label{lemmatraceprop}
    Let $f\in\M_k(N/\ell)$. We have
    \begin{enumerate}[label=(\alph*)]
        \item $\Tr_\ell(f)=(\ell+1)f$,
        \item $\Tr_\ell(W_\ell f)=\ell T_\ell(f)$.
    \end{enumerate}
\end{lemma}
\begin{proof}
    For part (a), we note that $U_\ell(i_\ell(f))=f$ and thus
    \begin{equation*}
        \Tr_{\ell}(f)=f+\ell^{1-k} U_{\ell}W_{\ell}f=f+\ell^{1-k}U_{\ell}\ell^ki_\ell f=f+\ell f=(\ell+1)f.
    \end{equation*}
    For part (b), we first recall that $W_\ell^2=\ell^k$. As a consequence,
    \begin{equation*}
        \Tr_\ell(W_\ell f)=W_\ell f+\ell^{1-k}U_\ell W_\ell^2 f=W_\ell f+\ell U_\ell f.
    \end{equation*}
    The result then follows since $\ell T_\ell=\ell U_\ell+\ell^ki_\ell=\ell U_\ell+W_\ell$.
\end{proof}
Using these results we can now define newforms in terms of the trace operator. This characterisation was first noted by Serre \cite[\S 3.1(c) remarque (3)]{serre1973formes} in the specific case when $N=\ell$.   

\begin{theorem}\label{tracetheorem}
    We have
    \begin{equation*}
        \mcS_k(N)^{\ell-\new}=\ker(\Tr_\ell|_{\mcS_k(N)})\cap\ker(\Tr_\ell W_\ell|_{\mcS_k(N)}).
    \end{equation*}
\end{theorem}

\begin{proof}
    Using \eqref{explicittrace} and Corollary \ref{coruleigen} we see that $\Tr_\ell(f)=\Tr_\ell(W_\ell f)=0$ for any $\ell$-new form $f\in\mcS_k(N)^{\ell-\new}$. Hence $\mcS_k(N)^{\ell-\new}\subseteq\ker(\Tr_\ell)\cap\ker(\Tr_\ell W_\ell)$.\\[8pt]
    We now need to show that for any eigenform $f\in\mcS_k(N)^{\old}$, $f\notin \ker(\Tr_\ell)\cap\ker(\Tr_\ell W_\ell)$. As in our proof of Theorem \ref{dlcusptheorem}, it will suffice to consider $k>0$ and $f$ contained in the space $V_{\ell,g}$ spanned by $g$ and $W_{\ell}(g)$ for some eigenform $g\in \mcS_k(N/\ell)$. By Lemma \ref{lemmatraceprop} we can represent $\Tr_{\ell}$ and $\Tr_{\ell}W_{\ell}$ as matrices acting on the ordered basis $\{g,W_{\ell}(g)\}$:
    \begin{equation*}
        \Tr_{\ell}=\begin{pmatrix}\ell+1&\ell a_{\ell}(g)\\0&0\end{pmatrix},\qquad \Tr_{\ell}W_{\ell}=\begin{pmatrix}\ell a_{\ell}(g)&(\ell+1)\ell^k\\0&0\end{pmatrix},
    \end{equation*}
    where $a_\ell(g)$ is the $T_\ell$ eigenvalue of $g$. The kernels of these matrices only intersect nontrivially when $a_{\ell}(g)^2=(\ell+1)^2\ell^{k-2}$. Thus $f\notin{\ker(\Tr_{\ell})\cap\ker(\Tr_{\ell}W_\ell)}$ otherwise we would contradict the Weil bound (Lemma \ref{weillemma}).
\end{proof}
\section{Extension to all Modular Forms}\label{secallmod}
We now have two equivalent definitions of newforms on $\mcS_k(N)$ in terms of the operators 
\begin{equation*}
    \D_\ell(f)=\ell^2U_\ell^2f-\ell^kf\ \text{and}\ \Tr_\ell(f)=f+\ell^{1-k}U_\ell W_\ell f.
\end{equation*}
These operators are well-defined on the entire space of modular forms $\M_k(N)$. Hence we can use them to define notions of newforms on spaces besides $\mcS_k(N)$.
\begin{definition}\label{ulandtrldef}
    Let $\mathcal{C}$ be a Hecke-invariant subspace of $\M_k(N)$. Define
    \begin{align*}
        \mathcal{C}^{U_\ell-\new}&:=\ker(\D_\ell|_\mathcal{C}),\ \text{and}\\[8pt] \mathcal{C}^{\Tr_\ell-\new}&:=\ker(\Tr_\ell|_C)\cap\ker(\Tr_\ell W_\ell|_C).
    \end{align*}
\end{definition}
Here, Hecke-invariant means that $\mathcal{C}$ is closed under the action of the Hecke operators. Our previous results in this chapter tell us that 
\begin{equation*}
    \mcS_k(N)^{\ell-\new}=\mcS_k(N)^{U_\ell-\new}=\mcS_k(N)^{\Tr_\ell-\new}
\end{equation*}
and we now consider whether a similar statement holds more generally for $\M_k(N)$.\\
\\
First we note that $\M_k(N)$ decomposes as
\begin{equation*}
    \M_k(N)=\mcS_k(N)\oplus\E_k(N)
\end{equation*}
where $\E_k(N)$ is a Hecke-invariant subspace of $\M_k(N)$ called the \emph{Eisenstein space} \cite[Chapter 4 and 5.11]{diamond2005first}. As our theory of newforms on $\mcS_k(N)$ is already established, we restrict our attention to $\E_k(N)$. \\
\\
Recall that in $\mcS_k(N)$, the $\ell$-new subspace consists of eigenforms that are not equivalent to any $\ell$-old eigenforms (Corollary \ref{lolddistincteigen}). Since the Hecke operators are well-defined on the Eisenstein space, the notions of eigenforms and equivalent eigenforms also make sense in $\E_k(N)$. We thus make the following definition.
\begin{definition}
    The space of \emph{$\ell$-new Eisenstein forms} $\E_k(N)^{\ell-\new}$ is the span of all eigenforms in $\E_k(N)$ that are not equivalent to any eigenforms in $\E_k(N)^{\ell-\old}=\E_k(N/\ell)~+~W_\ell(\E_k(N/\ell))$.
\end{definition}
In support of this definition, we note that $\E_k(N)$ has a basis of eigenforms and every newform in $\E_k(N)^{\new}=\bigcap_{\ell\mid N}\E_k(N)^{\ell-\new}$ is a simultaneous eigenfunction for all the Hecke operators $T_m$ with $m\geq 1$ \cite[Proposition 5.2.3]{diamond2005first}.\\
\\
We now see how $\E_k(N)^{\ell-\new}$ is related to $\E_k(N)^{U_\ell-\new}$ and $\E_k(N)^{\Tr_\ell-\new}$.
\begin{theorem}\label{eisentheorem}
    Suppose $k\neq 2$. We then have
    \begin{equation*}
        \E_k(N)^{\ell-\new}=\E_k(N)^{U_\ell-\new}=\E_k(N)^{\Tr_\ell-\new}=\{0\}.
    \end{equation*}
\end{theorem}
\begin{proof}
     First suppose that $k=0$. Then the only modular forms are constant functions \cite[Theorem 3.5.1]{diamond2005first}. Hence $\E_0(N)=\E_0(1)$ so that there are no nonzero $\ell$-new forms in $\E_0(N)$. Now let $f(z)=c$ be some nonzero, constant-valued function on $\H$. We then have
     \begin{align*}
         \mathcal{D}_\ell(f(z))&=\ell^2U_\ell^2 c-\ell^0 c=(\ell^2-1)c\neq 0,\ \text{and}\\[8pt]
         \Tr_\ell(W_\ell(f(z)))&=\Tr_\ell(f(z))=c+\ell^{1-0}U_\ell W_\ell c=c+\ell U_\ell \ell^0 i_\ell c=(1+\ell)c\neq 0.
     \end{align*}
     Hence $\E_0(N)^{\ell-\new}=\E_0(N)^{U_\ell-\new}=\E_0(N)^{\Tr_\ell-\new}=\{0\}$.\\[8pt]
     Now assume $k\neq 0,2$. In Appendix \ref{appassort} (Proposition \ref{eisisold}) we show that $\E_k(N)=\E_k(N)^{\ell-\old}$ and thus $\E_k(N)^{\ell-\new}=\{0\}$. As in the proof of Theorems \ref{dlcusptheorem} and \ref{tracetheorem}, any $\ell$-old form in $\M_k(N)^{U_\ell-\new}$ or $\M_k(N)^{\Tr_\ell-\new}$ arises from an eigenform $g\in\M_k(N/\ell)$ with its $T_\ell$-eigenvalue satisfying $a_\ell^2(g)=(\ell+1)^2\ell^{k-2}$. However, no eigenform in $\E_k(N/\ell)$ has such an eigenvalue \cite[Proposition 5.2.3]{diamond2005first}. Hence
     \begin{equation*}
        \E_k(N)^{\ell-\new}=\E_k(N)^{U_\ell-\new}=\E_k(N)^{\Tr_\ell-\new}=\{0\},    
     \end{equation*} 
     as required.
\end{proof}

When $k=2$ however, these notions of "newness" no longer coincide. This is illustrated in the following example.

\begin{example}
    Suppose that $N$ is not prime and let $g=E_2^{N/\ell}(z)$ be as defined in Example \ref{levelnexamples}. Note that $g\in\E_2(N/\ell)$ with $T_\ell(g)=(\ell+1)g$ \cite[Propostion 5.2.3]{diamond2005first}. Now let $h=\ell g- W_\ell(g)$, noting that $h\notin\E_2(N)^{\ell-\new}$. However, using Lemma \ref{lemmatraceprop} we have
    \begin{align*}
        \Tr_\ell(h)&=\ell\Tr_\ell(g)-\Tr_\ell(W_\ell(g))=(\ell+1)g-\ell(\ell+1)g=0,\ \text{and}\\[8pt]
        \Tr_\ell(W_\ell(h))&=\ell\Tr_\ell(W_\ell(g))-\ell^2\Tr_\ell(g)=\ell^2(\ell+1)g-\ell^2(\ell+1)g=0.
    \end{align*}
    Hence $h\in\E_2(N)^{\Tr_\ell-\new}$ and one can similarly show that $h\in\E_2(N)^{U_\ell-\new}$. 
\end{example}
\begin{question}
    Although the above example shows that $\E_2(N)^{\Tr_\ell-\new}\neq\E_2(N)^{\ell-\new}$ (whenever $N$ is not prime) is it still true that $\E_2(N)^{\Tr_\ell-\new}=\E_2(N)^{U_\ell-\new}$?
\end{question}
This problematic behaviour when $k=2$ is overlooked by Deo and Medvedovsky \cite[Proposition 6.1 and Proposition 6.3]{deo2019newforms}. However, it does not affect their subsequent results. \\
\\
For a more detailed analysis of Eisenstein newforms and their relation to the trace operator see \cite[Chapter 3]{weisinger1977some}.

%% file: Chapter4.tex
\chapter{Newforms in Characteristic Zero}\label{chap4}
We now consider a generalisation of modular forms, allowing $q$-expansions to have coefficients over some commutative ring $B$. The aim of this chapter is to show that when $B$ has characteristic zero, the theory of newforms mimics that discussed in Chapters \ref{chapclassic} and \ref{chapsquarefree}.\\[8pt]
The definitions in this chapter are adapted from \cite[Sections 2,5 and 6]{deo2019newforms}. From here onwards, we will assume the reader is familiar with the basic theory of flat modules. A summary can be found in Appendix \ref{flatappend}.

\section{Modular Forms over a Commutative Ring}\label{sec41}
Let $N>1$ be a fixed positive integer. Let $\M_k(N,\Z)$ (resp. $\mcS_k(N,\Z)$) denote the set of modular forms (resp. cusp forms) of weight $k$ for $\Gamma_0(N)$ with integral Fourier coefficients. For any commutative ring $B$ we then define
\begin{align*}
    \M_k(N,B)&:=\M_k(N,\Z)\otimes_\Z B,\ \text{and}\\
    \mcS_k(N,B)&:=\mcS_k(N,\Z)\otimes_\Z B,
\end{align*}
where $\M_k(N,\Z)$, $\mcS_k(N,\Z)$ and $B$ are viewed as $\Z$-modules\footnote{There is also a geometric definition of modular forms with coefficients in $B$ that was first introduced by Katz \cite{katz1973p}. However for most of the cases that we are interested in, Katz's definition agrees with our more elementary definitions (see \cite[Theorem 12.3.7]{diamond1995modular}). Certainly, $\M_k(N,B)\subseteq\M_k(N,B)^{\text{Katz}}$, where $\M_k(N,B)^{\text{Katz}}$ is the space of Katz's geometric modular forms.}.
Using the isomorphism\\
$\Z\otimes_\Z B\cong B$, we can treat $\M_k(N,B)$ (and similarly $\mcS_k(N,B))$ as a subset of $B[[q]]$. Namely,
\begin{equation*}
    \M_k(N,B)=\text{span}_B\left\{\sum_{n=0}^\infty \tilde{a}_nq^n\in B[[q]]\::\:\sum_{n=0}^\infty a_nq^n\in\M_k(N,\Z)\right\},
\end{equation*}
where $\tilde{a}_n$ is the image of $a_n$ under the unique ring homomorphism $\Z\to B$. Note that $\M_k(N,\C)$ agrees with the classical definition of modular forms since $\M_k(N)$ has a basis in $\M_k(N,\Z)$ \cite[Corollary 12.3.12]{diamond1995modular}. The same is true for $\mcS_k(N,\C)$.\\
\\
Now, let $\M(N,B)=\sum_{k=0}^\infty\M_k(N,B)$ denote the algebra of modular forms (for $\Gamma_0(N)$ over $B$). Similarly let $\mcS(N,B)=\sum_{k=0}^\infty\mcS_k(N,B)$ be the algebra of cusp forms. These algebras are much easier to work with when $B$ is a domain of characteristic zero.
\begin{proposition}\label{charzerodirect}
    Let $B$ be a domain of characteristic zero. Then
    \begin{equation*}
        \M(N,B)=\bigoplus_{k=0}^\infty\M_k(N,B).
    \end{equation*}
    An analogous result holds for $\mcS(N,B)$.
\end{proposition}
\begin{proof}
    For $B=\C$, and thus $B=\Z$, this is proven in Appendix \ref{appassort} (Proposition \ref{directoverc}). Another way of saying this is that the map
    \begin{equation}\label{injectiveoplus}
        \bigoplus_{k=0}^\infty\M_k(N,\Z)\to\M(N,\Z)
    \end{equation}
    is injective. Now suppose that $B$ is an arbitrary domain of characteristic zero. Then $B$ is flat as a $\Z$-module (Corollary \ref{flatdedekind}). Tensoring \eqref{injectiveoplus} by $B$ therefore preserves injectivity, giving the desired result.
\end{proof}

Next we define the Atkin-Lehner and Hecke operators over $B$. Recall that for $p$ prime and $f=\sum_{n=0}^\infty a_n(f)q^n\in\M_k(N,\C)$, the Hecke operators are defined by
\begin{alignat*}{3}
    a_n(T_p f)&=a_n(U_p f)=a_{np}(f),\qquad&&\text{if $p\mid N$,}\\
    a_n(T_p f)&=a_{pn}(f)+p^{k-1}a_{n/p}(f),\qquad&&\text{if $p\nmid N$}.
\end{alignat*}
Suppose now that $f\in\M_k(N,\Z)$ and $k>0$. Then by the above definitions we also have $T_p(f)\in\M_k(N,\Z)$. As a result, the Hecke operators are defined on $\M_k(N,B)$ as follows.
\begin{definition}\label{fullheckedefn}
    Let $B$ be a commutative ring, $p$ be a prime and $k>0$. The Hecke operator $T_p$ is defined on simple tensors in $\M_k(N,B)$ by
    \begin{equation*}
        T_p(f\otimes_{\Z} x):=T_p(f)\otimes_{\Z} x,
    \end{equation*}
    where $f\in\M_k(N,\Z)$ and $x\in B$.
\end{definition}
\begin{remark} 
    If $p$ is invertible in $B$ or $p\mid N$ then the above definition naturally extends to $k=0$. In addition, suppose that $B$ is a domain of characteristic zero. Then $T_p$ is well-defined on the algebra $\M(N,B)$ by Proposition \ref{charzerodirect}.
\end{remark}
Now let $\ell$ be a prime dividing $N$ exactly once. A result due to B. Conrad \cite[Theorem A.1]{prasanna2009arithmetic} tells us that the Atkin-Lehner operator $W_\ell$ is $\Z[1/\ell]$-integral. In other words, if $f$ is in $\M_k(N,\Z[1/\ell])$ then so is $W_\ell(f)$. Thus, if $\ell$ is invertible in $B$ we can naturally define $W_\ell$ on $\M_k(N,B)$ as follows.
\begin{definition}\label{wlgeneraldefn}
    Let $B$ be a $\Z[1/\ell]$-algebra (so that $\ell$ is invertible in $B$). The Atkin-Lehner operator $W_\ell$ is defined on $\M_k(N,B)$ by
    \begin{equation*}
        W_\ell(f\otimes_{\Z[1/\ell]}x):=W_\ell(f)\otimes_{\Z[1/\ell]}x
    \end{equation*}
    where $f\in\M_k(N,\Z[1/\ell])$ and $x\in B$. If $B$ is a $\Z[1/\ell]$-domain of characteristic zero then this definition extends to the algebra $\M(N,B)$.
\end{definition}
\begin{remark}
    The above definition makes sense since
    \begin{equation*}
        \M_k(N,B)=\M_k(N,\Z)\otimes_{\Z} B=\M_k(N,\Z)\otimes_{\Z}\Z[1/\ell]\otimes_{\Z[1/\ell]}B=\M_k(N,\Z[1/\ell])\otimes_{\Z[1/\ell]} B,
    \end{equation*}
    by the associativity of the tensor product.
\end{remark}
These definitions indicate that a theory of oldforms and newforms on $\M_k(N,B)$ will significantly depend on what conditions we impose on $B$. So in this chapter, we will frequently assume that $\ell$ is invertible in $B$ and/or that $B$ is a domain of characteristic zero. Further results for when $B$ has characteristic $p$ are then explored in Chapter \ref{chap5}.

\section{Oldforms over a Commutative Ring}\label{sec42}
Let $\ell$ be a prime dividing $N$. Recall that the space of $\ell$-old forms (of weight $k$) is given by
\begin{equation*}
    \M_k(N)^{\ell-\old}=\M_k(N/\ell)+i_\ell(\M_k(N/\ell)),
\end{equation*}
or if $\ell$ divides $N$ exactly once then
\begin{equation*}
    \M_k(N)^{\ell-\old}=\M_k(N/\ell)+W_\ell(\M_k(N/\ell)).
\end{equation*}
For $\ell$-old forms over any commutative ring $B$, we let $\M_k(N,\Z)^{\ell-\old}:=\M_k(N)^{\ell-\old}\cap\Z[[q]]$ and then define
\begin{equation*}
    \M_k(N,B)^{\ell-\old}:=\M_k(N,\Z)\otimes_\Z B.
\end{equation*}
Similarly, we define $\mcS_k(N,\Z)^{\ell-\old}=\mcS_k(N)^{\ell-\old}\cap\Z[[q]]$ and $\mcS_k(N,B)^{\ell-\old}=\mcS_k(N,\Z)^{\ell-\old}\otimes_\Z~B$. Note that $\M_k(N/\ell)$ and $\mcS_k(N/\ell)$ have bases in $\Z[[q]]$ and $i_\ell$ is $\Z$-integral. Hence the above definitions imply that $\M_k(N,\C)^{\ell-\old}=\M_k(N)^{\ell-\old}$ and $\mcS_k(N,\C)^{\ell-\old}=\mcS_k(N)^{\ell-\old}$.\\
\\
From these definitions it follows that $\M_k(N/\ell,B)+i_\ell(\M_k(N/\ell,B)\subseteq\M_k(N,B)^{\ell-\old}$. On the other hand, it is not always true that $\M_k(N,B)^{\ell-\old}\subseteq\M_k(N/\ell,B)+i_\ell(\M_k(N/\ell,B))$. This is illustrated in the following example.
\begin{example}\label{badoldform}
    Let $p\geq 5$ and $E'_{p-1}(z):=-\frac{B_{p-1}}{2(p-1)}E_{p-1}(z)$ be the Eisenstein series of weight $p-1$ scaled so that $a_1=1$. In particular,
    \begin{equation}\label{normeisenstein}
        E_{p-1}'(q)=-\frac{B_{p-1}}{2(p-1)}+\sum_{n=1}^\infty\sigma_{p-2}(n)q^n\in \M_{p-1}(1,\C).
    \end{equation}
    Now let
    \begin{equation}\label{epcrit}
        f(z)=E'_{p-1}(q)-E'_{p-1}(q^{\ell})
    \end{equation}
    noting that \eqref{epcrit} is the only way to express $f$ as the sum of a form in $\M_{p-1}(1,\C)$ and a form in $i_\ell(\M_{p-1}(1,\C))$ (see Corollary \ref{linindepwl}). Since the constant terms cancel in \eqref{epcrit} we have \\
    $f\in\M_k(N,\Z)^{\ell-\old}$ and thus $\overline{f}\in\M_k(N,\F_p)^{\ell-\old}$ where $\overline{f}$ is the image of $f$ in $\F_p[[q]]$. However, $\overline{f}$ is not in $\M_k(N/\ell,\F_p)+i_\ell(\M_k(N/\ell,\F_p))$. To see this, we note that the denominator of $B_{p-1}$ in \eqref{normeisenstein} is divisible by $p$ (Proposition \ref{propbernoulli}), but $p$ is not invertible in $\F_p$.
\end{example}

If $B$ is a $\Z[1/\ell]$-algebra then this problem no longer occurs.

\begin{proposition}\label{goodloldprop}
    If $B$ is a $\Z[1/\ell]$-algebra, then
    \begin{equation*}
        \mcS_k(N,B)^{\ell-\old}=\mcS_k(N/\ell,B)\oplus W_{\ell}(\mcS_k(N/\ell,B)).
    \end{equation*}
\end{proposition}
To prove this result, we first need the following lemma.
\begin{lemma}\label{constlemmacharzero}
    Let $B$ be a $\Z[1/\ell]$-algebra. If $f\in \M_k(N/\ell,B)\cap W_\ell \M_k(N/\ell,B)$, then $f$ is constant.
\end{lemma}
\begin{proof}
    Generalise the argument from Proposition \ref{constint}.
\end{proof}
\begin{proof}[Proof of Proposition \ref{goodloldprop}]
    By the Lemma \ref{constlemmacharzero} we have $S_k(N/\ell,B)\cap W_\ell (S_k(N/\ell,B))=\{0\}$ since nonzero cusp forms cannot be constant. Hence it suffices to show that
    \begin{equation}\label{oldspaneq}
        \mcS_k(N,B)^{\ell-\old}=\mcS_k(N/\ell,B)+W_{\ell}(\mcS_k(N/\ell,B)).
    \end{equation}
    Now, note that \eqref{oldspaneq} is true when $B=\C$. Then, since $\mcS_k(N,\C)$ has a basis in $\mcS_k(N,\Q)$, \eqref{oldspaneq} also holds when $B=\Q$ by a dimension counting argument.\\[8pt]
    Next we consider the case when $B=\Z[1/\ell]$. Directly from the definition of oldforms over $B$ (see Section \ref{sec42}) we see that $\mcS_k(N/\ell,\Z[1/\ell])+W_\ell( \mcS_k(N/\ell,\Z[1/\ell]))\subseteq \mcS_k(N,\Z[1/\ell])^{\ell-\old}$. The reverse inclusion is not as clear.\\[8pt]
    Let $f\in \mcS_k(N,\Z[1/\ell])^{\ell-\old}$. Since $\Z[1/\ell]\subseteq\Q$ we have $\mcS_k(N,\Q)=\mcS_k(N,\Z[1/\ell])\otimes_{\Z[1/\ell]}\Q$ and thus
    \begin{equation*}
        f=\frac{1}{M}(g+W_{\ell}h),
    \end{equation*}
    where $M\in\Z$ and $g,h\in \mcS_k(N/\ell,\Z[1/\ell])$. Because $f$ has coefficients in $\Z[1/\ell]$ we must then have $g\equiv-W_\ell(h)\pmod{M\Z[1/\ell]}$\footnote{That is, $a_n(g)\equiv a_n(-W_\ell(h))$ for all $n\geq 0$.}. Now, $\Z[1/\ell]/M\Z[1/\ell]$ is also a $\Z[1/\ell]$-algebra so by Lemma \ref{constlemmacharzero}
    \begin{equation*}
        W_{\ell}(h)\equiv g\equiv 0\pmod{ M\Z[1/\ell]}.
    \end{equation*}
    Hence both $\frac{1}{M}g$ and $\frac{1}{M} W_{\ell}(h)$ have coefficients in $\Z[1/\ell]$ and thus $f$ is an element of \\
    $\mcS_k(N/\ell,\Z[1/\ell])+W_\ell( \mcS_k(N/\ell,\Z[1/\ell]))$ as required.\\[8pt]
    Finally, we note that if $B$ is any $\Z[1/\ell]$-algebra then 
    \begin{align*}
        \mcS_k(N,B)^{\ell-\old}
        &\cong \mcS_k(N/\ell,\Z[1/\ell])^{\ell-\old}\otimes_{\Z[1/\ell]}B\\[8pt]
        &\cong (\mcS_k(N/\ell,\Z[1/\ell])\oplus W_\ell( \mcS_k(N/\ell,\Z[1/\ell])))\otimes_{\Z[1/\ell]}B\\[8pt]
        &=\mcS_k(N/\ell,B)\oplus W_\ell( \mcS_k(N/\ell,B)),
    \end{align*}
    as required.
\end{proof}

\section{Newforms over a Commutative Ring}\label{sec43}
For any commutative ring $B$, we define
\begin{align*}
    \mcS_k(N,\Z)^{\ell-\new}&:=\mcS_k(N)^{\ell-\new}\cap\Z[[q]],\ \text{and}\\[8pt]
    \mcS_k(N,B)^{\ell-\new}&:=\mcS_k(N,\Z)^{\ell-\new}\otimes_{\Z}B.
\end{align*}
Also let
\begin{equation*}
    \mcS(N,B)^{\ell-\new}=\sum_{k=0}^\infty\mcS_k(N,B)^{\ell-\new}.
\end{equation*}
Then, to see that $\mcS_k(N,\C)^{\ell-\new}$ agrees with the classical definition, we prove that $\mcS_k(N)^{\ell-\new}$ has a basis in $\mcS_k(N,\Z)^{\ell-\new}$.
\begin{proposition}
    The space $\mcS_k(N)^{\ell-\new}$ has a basis in $\mcS_k(N,\Z)^{\ell-\new}$.
\end{proposition}
\begin{proof}
    First note that for all $M\geq 1$, $\mcS_k(M)^{\new}$ has a basis in $\Z[[q]]$. To see this, one can use the fact that Galois conjugates of newforms are new \cite[Corollary 12.4.5]{diamond1995modular} and then apply the argument in \cite[Corollary 6.5.6]{diamond2005first}.\\[8pt]
    Now, by Proposition \ref{newdecomp} we have the decomposition
    \begin{equation}\label{midpropdecomp}
        \mcS_k(N)^{\ell-\new}=\sum_{\substack{M\\ \ell^{\alpha}\mid M\mid N}}\sum_{e\mid(N/M)}i_e(\mcS_k(M)^{\new}),
    \end{equation}
    where $\alpha$ is the number of times that $\ell$ divides $N$. Since each $\mcS_k(M)^{\new}$ has a basis over $\Z$  and $i_e$ is $\Z$-integral, it thus follows that $\mcS_k(N)^{\ell-\new}$ has a basis in $\mcS_k(N,\Z)^{\ell-\new}$ as required.
\end{proof}
If $k\neq 2$, the above argument also applies to $\M_k(N)^{\ell-\new}=\mcS_k(N)^{\ell-\new}\oplus\E_k(N)^{\ell-\new}$ as defined in Section \ref{secallmod}. In particular, one can use the explicit basis of eigenforms described in \cite[Proposition 5.2.3]{diamond2005first} to produce a decomposition analogous to \eqref{midpropdecomp} for $\E_k(N)^{\ell-\new}$ and thus $\M_k(N)^{\ell-\new}$ \footnote{See also \cite[Proposition 21]{weisinger1977some}.}. In light of this, we also define 
\begin{align*}
    \M_k(N,\Z)^{\ell-\new}&:=\M_k(N)^{\ell-\new}\cap\Z[[q]],\ \text{and}\\[8pt]
    \M_k(N,B)^{\ell-\new}&:=\M_k(N,\Z)^{\ell-\new}\otimes_{\Z}B.
\end{align*}

\section{Span of Oldforms and Newforms in Characteristic Zero}\label{sec44}
Recall the classical decomposition of $\mcS_k(N,\C)$ into $\ell$-old and $\ell$-new forms:
\begin{equation}\label{classicaldecomp}
    \mcS_k(N,\C)^{\ell-\old}\oplus\mcS_k(N,\C)^{\ell-\new}=\mcS_k(N,\C).
\end{equation}
A natural question to ask is whether this decomposition holds if we replace $\C$ with an arbitrary commutative ring $B$. Unfortunately, this does not turn out to be the case in general. For instance, in characteristic $p$ there are congruences between $\ell$-old and $\ell$-new forms so that we cannot even use a direct sum as in \eqref{classicaldecomp}. However, at least when $B$ is a domain of characteristic zero, the space of $\ell$-old and $\ell$-new forms intersect trivially. 

\begin{proposition}\label{proplnewintersect}
    Let $B$ be a domain of characteristic zero. Then for any weight $k$,
    \begin{equation*}
        \mcS_k(N,B)^{\ell-\old}\cap\mcS_k(N,B)^{\ell-\new}=\{0\}.
    \end{equation*}
\end{proposition}
\begin{proof}
    First note that the proposition holds for $B=\C$ and thus $B=\Z$. Equivalently, the map
    \begin{equation}\label{lnewinjective}
        \mcS_k(N,\Z)^{\ell-\new}\oplus\mcS_k(N,\Z)^{\ell-\old}\to\mcS_k(N,\Z)
    \end{equation}
    is injective. Now suppose that $B$ is an arbitrary domain of characteristic zero. As $B$ is flat over $\Z$ (Corollary \ref{flatdedekind}), tensoring \eqref{lnewinjective} with $B$ preserves injectivity thus giving the desired result.
\end{proof}
In addition, if $B$ is a field then the $\ell$-new and $\ell$-old forms span $\mcS_k(N,B)$.
\begin{proposition}\label{fieldspanprop}
    Let $B$ be a field of characteristic zero. Then for any weight $k$,
    \begin{equation*}
        \mcS_k(N,B)^{\ell-\old}\oplus\mcS_k(N,B)^{\ell-\new}=\mcS_k(N,B).
    \end{equation*}
\end{proposition}
\begin{proof}
    Since $\mcS_k(N,B)^{\ell-\old}$ and $\mcS_k(N,B)^{\ell-\new}$ intersect trivially (Proposition \ref{proplnewintersect}), it suffices to show that 
    \begin{equation}\label{dimeq}
        \dim_B(\mcS_k(N,B)^{\ell-\old})+\dim_B(\mcS_k(N,B)^{\ell-\new})=\dim_B(\mcS_k(N,B)).
    \end{equation}
    First note that this holds when $B=\C$ since $\mcS_k(N,\C)^{\ell-\new}$ is the orthogonal complement of $\mcS_k(N,\C)^{\ell-\old}$ with respect to the Petersson inner product. Now, $\mcS_k(N,\C)^{\ell-\old}$, $\mcS_k(N,\C)^{\ell-\new}$ and $\mcS_k(N,\C)$ all have bases over $\Z$ and thus $\Q$. Hence \eqref{dimeq} holds for $B=\Q$ as well.\\[8pt]
    We now take $B$ to be an arbitrary field of characteristic $0$. Note that $\Q$ is a subfield of $B$. In particular, we have a unique injective ring homomorphism from $\Q$ to $B$, sending $\frac{a}{b}\in\Q$ to $\varphi(a)\varphi(b)^{-1}$ where $\varphi$ is the unique ring homomorphism from $\Z$ to $B$. As a result,
    \begin{equation*}
        \dim_B(\mcS_k(N,B))=\dim_B(\mcS_k(N,\Q)\otimes_\Q B)=\dim_\Q(\mcS_k(N,\Q))
    \end{equation*}
    and similarly for $\mcS_k(N,B)^{\ell-\old}$ and $\mcS_k(N,B)^{\ell-\new}$. Thus \eqref{dimeq} is true for all characteristic 0 fields as required.
\end{proof}

The following example demonstrates what can happen if we remove the assumption that $B$ is a field.
\begin{example}\label{oldnewcong}
    The space $\mcS_4(5)$ is spanned by a single cuspform
    \begin{equation*}
        f=q-4q^2+2q^3+8q^4-5q^5-8q^6+6q^7-23q^9+O(q^{10}).
    \end{equation*}
    Similarly, the space $\mcS_4(10)^{2-\new}$ is spanned by a single cuspform
    \begin{equation*}
        g=q+2q^2-8q^3+4q^4+5q^5-16q^6-4q^7+8q^8+37q^9+O(q^{10}).
    \end{equation*}
    Note that $f\in\mcS_4(10,\Z)^{2-\old}$ and $g\in\mcS_4(10,\Z)^{2-\new}$. Looking at the first few terms we also see that $a_n(f)\equiv a_n(g)\pmod{2}$ and by \cite[Corollary 9.20]{stein2007modular} this congruence holds for all $n\geq 0$. Hence, we also have 
    \begin{equation}\label{fminusg}
        h:=\frac{1}{2}(f-g)\in\mcS_4(10,\Z).
    \end{equation}
    But $h\notin\mcS_4(10,\Z)^{2-\old}\oplus\mcS_4(10,\Z)^{2-\new}$. To see this, note that $h=\frac{1}{2}f-\frac{1}{2}g$ is the only way to write $h$ as the sum of a forms in $\mcS_4(10,\C)^{2-\old}$ and $\mcS_4(10,\C)^{2-\new}$ by Proposition \ref{fieldspanprop}. However, $\frac{1}{2}f$ and $\frac{1}{2}g$ are not in $\mcS_4(10,\Z)$.
\end{example}
\section{$U_\ell$-new and $\Tr_\ell$-new in Characteristic Zero}\label{sec45}
We now look at generalising the notions of $U_\ell$-new and $\Tr_\ell$-new from Chapter \ref{chapsquarefree} to submodules of $\M(N,B)$, particularly when $B$ is a domain of characteristic zero. In what follows, we will always assume that $\ell$ is a fixed prime dividing $N$.

\begin{definition}
   Let $B$ be any commutative ring and $\mathcal{C}$ be a Hecke-invariant submodule of $\M_k(N,B)$. Define
    \begin{equation*}
        \mathcal{C}^{U_\ell-\new}:=\ker(\D_\ell|_{\mathcal{C}}).
    \end{equation*}
    If $B$ is a domain of characteristic zero then we can take $\mathcal{C}$ to be any Hecke-invariant submodule of $\M(N,B)=\bigoplus_{k=0}^\infty\M_k(N,B)$.
\end{definition}

\begin{theorem}\label{dlcharzerothm}
    Let $B$ be a domain of characteristic zero. We have
    \begin{equation*}
        \mcS(N,B)^{\ell-\new}=\mcS(N,B)^{U_\ell-\new}.
    \end{equation*}
\end{theorem}
\begin{proof}
    Since $B$ is a domain of characteristic zero, $\mcS(N,B)=\bigoplus_{k=0}^\infty\mcS_k(N,B)$. So, since $\D_\ell$ is weight preserving it suffices to prove
    \begin{equation}\label{singleweighteq}
        \mcS_k(N,B)^{\ell-\new}=\mcS_k(N,B)^{U_\ell-\new}
    \end{equation}
    in a single weight $k$. By Theorem \ref{dlcusptheorem}, \eqref{singleweighteq} holds for $B=\C$ and thus also for $B=\Z$. Then, if $B$ is any other domain of characteristic zero we have
    \begin{equation}\label{ellnewcalculations}
        \mcS_k(N,B)^{\ell-\new}=\mcS_k(N,\Z)^{\ell-\new}\otimes_\Z B=\mcS_k(N,\Z)^{U_\ell-\new}\otimes_\Z B=\ker(D_{\ell}|_{\mcS_k(N,\Z)})\otimes_\Z B.
    \end{equation}
    However, since $B$ is flat over $\Z$,
    \begin{equation*}
        \ker(\D_\ell|_{\mcS_k(N,\Z)})\otimes_\Z B=\ker(\D_\ell|_{\mcS_k(N,\Z)}\otimes_\Z 1_B)=\ker(\D_\ell|_{\mcS_k(N,B)})=\mcS_k(N,B)^{U_\ell-\new}
    \end{equation*}
    (see Proposition \ref{preservekernelprop} in the Appendix). Combining this with \eqref{ellnewcalculations} gives the desired result.
\end{proof}
We now make a similar generalisation for $\Tr_\ell$-new forms.
\begin{definition}\label{trlgeneraldef}
    Let $B$ be a $\Z[1/\ell]$-algebra and $\mathcal{C}$ be a Hecke-invariant submodule of $\M_k(N,B)$. Define
    \begin{equation*}
        \mathcal{C}^{\Tr_\ell-\new}:=\ker(\Tr_\ell|_C)\cap\ker(\Tr_\ell W_\ell|_C).
    \end{equation*}
    If $B$ is a $\Z[1/\ell]$-domain of characteristic zero then we can take $\mathcal{C}$ to be any Hecke-invariant submodule of $\M(N,B)=\bigoplus_{k=0}^\infty\M_k(N,B)$.
\end{definition}
\begin{remark}
    Here, we require that $B$ is a $\Z[1/\ell]$-algebra so that $W_\ell$ and $\Tr_\ell$ are well-defined. This follows from our general definition of $W_\ell$ (Definition \ref{wlgeneraldefn}) and the formula $\Tr_\ell(f)~=~f~+~\ell^{1-k}U_\ell W_\ell f$.
\end{remark}

\begin{theorem}\label{bigcharzerothm}
    Let $B$ be a $\Z[1/\ell]$-domain of characteristic zero. We have
    \begin{equation*}
        \mcS(N,B)^{\ell-\new}=\mcS(N,B)^{U_\ell-\new}=\mcS(N,B)^{\Tr_\ell-\new}
    \end{equation*}
    and if $k\neq 2$
    \begin{equation*}
        \M_k(N,B)^{\ell-\new}=\M_k(N,B)^{U_\ell-\new}=\M_k(N,B)^{\Tr_\ell-\new}.
    \end{equation*}
\end{theorem}
\begin{proof}
    In Chapter \ref{chapsquarefree} we showed that all of these statements were true over $\C$. To generalise to any characteristic zero $\Z[1/\ell]$-domain $B$, we repeat the arguments from the proof of Theorem \ref{dlcharzerothm} making use of Theorems \ref{tracetheorem} and \ref{eisentheorem}. 
\end{proof}
Considering all primes dividing $N$ gives the corresponding result for \\
$\mcS(N,B)^{\new}=\bigcap_{\ell\mid N}\mcS(N,B)^{\ell-\new}$ and $\M_k(N,B)^{\new}=\bigcap_{\ell\mid N}\M_k(N,B)^{\ell-\new}$.
\begin{corollary}\label{bigcharzerocorollary}
    Let $B$ be a $\Z[1/N]$-domain of characteristic zero. If $N$ is squarefree then
    \begin{equation*}
        \mcS(N,B)^{\new}=\bigcap_{\ell\mid N}\mcS(N,B)^{U_\ell-\new}=\bigcap_{\ell\mid N}\mcS(N,B)^{\Tr_\ell-\new}
    \end{equation*}
    and if $k\neq 2$,
    \begin{equation*}
        \M_k(N,B)^{\new}=\bigcap_{\ell\mid N}\M_k(N,B)^{U_\ell-\new}=\bigcap_{\ell\mid N}\M_k(N,B)^{\Tr_\ell-\new}.
    \end{equation*}
\end{corollary}
\begin{proof}
    Follows directly from Theorem \ref{bigcharzerothm}. Note that since $N$ is squarefree, each prime $\ell$ divides $N$ exactly once. Moreover, each prime $\ell\mid N$ is invertible in $B$ with $\ell^{-1}=\frac{(N/\ell)}{N}$.
\end{proof}
The theory of newforms in characteristic zero thus naturally generalises the classical theory over $\C$.

%% file: Chapter5.tex
\chapter{Newforms in Characteristic \lowercase{$p$}}\label{chap5}
In this chapter, we attempt to generalise our results of the previous chapter to characteristic $p$. However, due to congruences between modular forms, the theory of newforms in characteristic $p$ is more complicated than in characteristic zero.\\[8pt]
Again we assume that $N>1$ is a fixed positive integer. Moreover, for any two modular forms $f,g\in\M_k(N,\Z)$ we will always write $f\equiv g\pmod{p}$ to mean that $a_n(f)\equiv a_n(g)$ for all $n\geq 0$.\\[8pt]
The main results of this chapter are adapted from \cite[Section 6.2]{deo2019newforms}. 

\section{Modular Forms in Characteristic $p$}\label{sec51}
In Chapter \ref{chap4} we showed that if $B$ is a domain of characteristic zero then
\begin{equation*}
    \M(N,B)=\bigoplus_{k=0}^\infty\M_k(N,B).
\end{equation*}
This is no longer true in characteristic $p$ as a result of the following proposition.
\begin{proposition}
    Let $p\geq 5$ be a prime number. The Eisenstein series $E_{p-1}$ can be scaled so that $E_{p-1}\in\M_k(N,\Z)$. Moreover, we can choose this scaling to be such that
    \begin{equation*}
        E_{p-1}(q)\equiv 1\pmod{p}.
    \end{equation*}
    If $p=2$ or $p=3$ then $E_4$ and $E_6$ can similarly be scaled so that
    \begin{equation*}
        E_4(q)\equiv E_6(q)\equiv 1\pmod{p}.
    \end{equation*}
\end{proposition}
\begin{proof}
    Recall that $E_{p-1}$ is given by
    \begin{equation*}
        E_{p-1}(q)=1-\frac{2(p-1)}{B_{p-1}}\sum_{n=1}^\infty\sigma_{p-2}(n)q^n,
    \end{equation*}
    where $B_{p-1}$ is the $(p-1)^{\text{th}}$ Bernoulli number. Now, write $B_{p-1}=\frac{a}{b}$ where $a$ and $b$ are coprime integers. As a result
    \begin{equation*}
        aE_{p-1}(q)=a-2(p-1)b\sum_{n=1}^\infty\sigma_{p-2}(n)q^n\in\M_k(N,\Z).
    \end{equation*}
    However, $b\equiv 0\pmod{p}$ by Proposition \ref{propbernoulli} and thus
    \begin{equation*}
        aE_{p-1}(q)\equiv a\pmod{p}.
    \end{equation*}
    Next we note that $\gcd(a,p)=1$ since $\gcd(a,b)=1$ and $b\equiv 0\pmod{p}$. Therefore there exists $m\in\{1,\dots,p-1\}$ such that 
    \begin{equation*}
        maE_{p-1}(q)\equiv 1\pmod{p},
    \end{equation*}
    which is the desired result. Finally, if $p=2$ or $p=3$ then we can perform a similar argument for $E_4$ and $E_6$ noting that $B_4=-\frac{1}{30}$ and $B_6=\frac{1}{42}$.
\end{proof}
\begin{remark}
    Here, we are slightly abusing notation by writing $E_{p-1}(q)\equiv 1\pmod{p}$ since technically an integer multiple of $E_{p-1}$ (but not necessarily $E_{p-1}$) satisfies this congruence. However, from here onwards we will keep this notation and always assume that the Eisenstein series are scaled as such.
\end{remark}
    
\begin{corollary}
    Let $p$ be a prime number. If $p\geq 5$ then $\M_k(N,\F_p)\subseteq\M_{k+(p-1)}(N,\F_p)$. On the other hand, if $p=2$ or $p=3$ then $\M_k(N,\F_p)\subseteq\M_{k+4}(N,\F_p)$ and $\M_k(N,\F_p)\subseteq~\M_{k+6}(N,\F_p)$. In particular, the sum $\sum_{k=0}^\infty\M_k(N,\F_p)$ is never direct.
\end{corollary}
\begin{proof}
    If $p\geq 5$ then for any $f\in\M_k(N)$ we have $f\equiv E_{p-1}f\pmod{p}$. Since $E_{p-1}f$ is in $\M_{k+(p-1)}(N)$ we attain the desired result. For $p=2$ or $p=3$ we can argue similarly using $E_4$ and $E_6$.
\end{proof}
These are essentially the only congruences we will have to consider in characteristic $p$.
\begin{proposition}\label{gorenprop}
    Let $p\geq 5$ be a prime number and $B$ be a domain of characteristic $p$. The kernel of the map
    \begin{equation*}
        \bigoplus_{k=0}^\infty\M_k(N,B)\to B[[q]]
    \end{equation*}
    is generated by $\overline{E}_{p-1}-\overline{1}$. Here, $\overline{E}_{p-1}$ and $\overline{1}$ are the images of $E_{p-1}\in\M_{p-1}(N,\Z)$ and $1\in \M_0(N,\Z)$ in $B[[q]]$.
\end{proposition}
\begin{remark}
    Although $E_{p-1}\equiv 1\pmod{p}$ we view $\overline{E}_{p-1}$ and $\overline{1}$ as distinct modular forms in $\bigoplus_{k=0}^\infty\M_k(N,B)$ since they arise from forms in $\M(N,\Z)$ with different weights.
\end{remark}
\begin{proof}[Proof of Proposition 5.3]
    For $B=\F_p$ this is shown in \cite[Theorem 5.4 in Chapter 4]{goren2002lectures}. Now suppose that $B$ is an arbitrary domain of characteristic $p$. Then $B$ is a vector space over $\F_p$. Namely, this means that $B$ is a free $\F_p$-module and thus flat over $\F_p$. Since flat modules preserve kernels (Proposition \ref{preservekernelprop}) we obtain the desired result.
\end{proof}
In light of this observation, we define
\begin{equation*}
    \M(N,B)^i:=\sum_{k\:\equiv\:i\!\!\!\pmod{p-1}}\M_k(N,B),
\end{equation*}
so that when $B$ is a domain of characteristic $p$
\begin{equation}\label{eqmodpdecomp}
    \M(N,B)=\bigoplus_{i\in \Z/(p-1)\Z}\M(N,B)^i.
\end{equation}
In particular, when $p\geq 5$ the decomposition \eqref{eqmodpdecomp} follows from Proposition \ref{gorenprop}. On the other hand, if $p=2$ or $p=3$ then we have $\M(N,B)=\M(N,B)^0$ since there are no nonzero forms of odd weight in $\M(N,B)$.

\begin{proposition}\label{singleweightprop}
    Let $B$ be a domain of characteristic $p$ and $i\in\Z/(p-1)\Z$. Then each element of $\M(N,B)^i$ appears in a single weight.
\end{proposition}
\begin{proof}
    For any $f\in\M(N,B)^i$ we have $f=f_1+\dots+f_n$ where each $f_j$ is nonzero and of fixed weight $k_j$. If $p\geq 5$ then we can multiply each $f_j$ by a suitable power of $\overline{E}_{p-1}$ so that $f$ appears in a single weight. Now assume that $p=2$ or $p=3$. Since there are no nonzero forms in $\M(N,B)$ of odd weight, each of the $f_j$ differ in weight by a multiple of 2. So, we can similarly multiply each $f_j$ by suitable powers of $\overline{E}_4$ and $\overline{E}_6$ so that $f$ appears in a single weight.
\end{proof}
\begin{remark}
    The above proposition also holds if we restrict to $\mcS(N,B)^i:=\sum_{k\:\equiv\: i\pmod{p-1}}\mcS_k(N,B)$.
\end{remark}
\section{Atkin-Lehner, Hecke, $\D_\ell$ and $\Tr_\ell$ Operators in Characteristic $p$}
Let $\ell$ be a prime dividing $N$ exactly once. In Chapter \ref{chap4} we showed that if $B$ is a $\Z[1/\ell]$-algebra then the scaled Atkin-Lehner operator $W_\ell=\ell^{k/2}w_\ell$ is well-defined on $\M_k(N,B)$. In addition, $\M_k(N/\ell,B)\cap W_{\ell}(\M_k(N/\ell,B))=B$ (Lemma \ref{constlemmacharzero}). If $B$ is also a domain (of any characteristic) then these results extend to the algebra $\M(N,B)=\sum_{k=0}^\infty\M_k(N,B)$.

\newpage

\begin{proposition}\label{propauto}
    Let $B$ be a $\Z[1/\ell]$-domain. Then:
    \begin{enumerate}[label=(\alph*)]
        \item The operator $W_\ell$ is a well-defined automorphism on the algebra $\M(N,B)$.
        \item We have $\M(N/\ell,B)\cap W_\ell(M(N/\ell,B))=B$.
    \end{enumerate}
\end{proposition}
\begin{proof}\
    \begin{enumerate}[label=(\alph*)]
        \item In characteristic zero $\M(N,B)=\bigoplus_{k=0}^\infty\M_k(N,B)$ so we can simply define $W_\ell$ on $\M(N,B)$ by extending its definition on each $\M_k(N,B)$ by linearity. Note that $W_\ell$ is an automorphism since it has an inverse $W_\ell^{-1}=\ell^{-k}W_\ell$.\\[8pt]
        Now suppose that $B$ has characteristic $p$. For any $f\in\M(N,B)$ we have a unique decomposition  $f=\sum_{i}f_i$ where $f_i\in\M(N,B)^i$ as in \eqref{eqmodpdecomp}. Hence it suffices to show that $W_\ell$ is well-defined on each $\M(N,B)^i$. As discussed in Section \ref{sec51}, each form in $\M(N,B)^i$ appears in a single weight. We need to show that $W_\ell$ does not depend on this weight.\\[8pt]
        So, let $f_i,f_i'\in\M(N,\Z)^i$ be forms with the same mod $B$ reduction\footnote{That is, $\varphi(a_n(f_i))=\varphi(a_n(f_i'))$ where $\varphi$ is the unique ring homomorphism $\varphi\colon\Z\to B$.} and weights $k$ and $k'$ respectively. Without loss of generality, assume that $k'\geq k$ so that $k'-k=m(p-1)$ for some $m\geq 0$. If $p\geq 5$ then $E_{p-1}^mf_i$ has the same weight as $f_i'$ and also has the same $q$-expansion mod $B$. Since $W_\ell$ is well-defined on $\M_{k'}(N,B)$ (i.e. in fixed weight) we then have
        \begin{align*}
            W_\ell(f_i')\equiv W_\ell(E_{p-1}^mf_i)\equiv W_\ell(E_{p-1})^mW_\ell(f_i)\equiv(\ell^{p-1}E_{p-1}(\ell z))^mW_\ell(f_i)\equiv W_{\ell}(f_i),
        \end{align*}
        where the congruences are taken mod $B$. Here $\ell^{(p-1)}\equiv 1$ by Fermat's little theorem\footnote{Since $\ell$ is invertible in $B$ we must have $\gcd(\ell,p)=1$.}. If $p=2$ or $p=3$ then we can perform an analogous argument using powers of $E_4$ and $E_6$. Therefore, $W_\ell$ is well-defined on each $\M(N,B)^i$ and thus $\M(N,B)$ as required.
        
        \item Again, if $B$ has characteristic zero then the result follows since $\M(N,B)=\bigoplus_{k=0}^\infty\M_k(N,B)$. Now suppose that $B$ has characteristic $p$. Let $f$ and $g$ be forms in $\M(N,B)$ such that $f=W_\ell(g)$. We write $f=\sum_i f_i$, $g=\sum_i g_i$ where $f_i,g_i\in\M(N,B)^i$. By multiplying by suitable powers of Eisenstein series we may also assume that each $f_i$ and $g_i$ have the same fixed weight $k_i$. Now, since $f=W_\ell(g)$, we have $\sum_if_i=\sum_iW_{\ell}(g_i)$. However, $W_\ell$ is weight-preserving so linear independence implies that $f_i=W_\ell(g_i)$. As a result, each $f_i$ is constant (Lemma \ref{constlemmacharzero}) so that $f=\sum_if_i\in B$ as required.\qedhere 
    \end{enumerate}
\end{proof}
In light of these results, we define for any $\Z[1/\ell]$-domain $B$
\begin{equation*}
    \mcS(N,B)^{\ell-\old}=\mcS(N/\ell,B)\oplus W_\ell(\mcS(N/\ell,B)).
\end{equation*}
This agrees with the standard definition of $\ell$-old forms in fixed weight (cf. Proposition \ref{goodloldprop}).\\
\\
We also emphasise the importance of using the scaled Atkin-Lehner operator $W_\ell$. In particular, the following example shows that the operator $w_{\ell}=\ell^{-k/2}W_\ell$ is \textbf{not} always well-defined on $\M(N,B)$.
\begin{example}
    Let $B=\F_5$ so that $\overline{1}\in\M_0(N,B)$ and $\overline{E}_4\in\M_4(N,B)$ have the same $q$-expansion. We then have $w_3(1)=1$ and $w_3(E_4)=3^2E_4(\gamma_\ell\cdot z)\equiv -1\pmod{5}$, so that $w_3(\overline{1})\neq w_3(\overline{E}_4)$.
\end{example}

The main idea here is that $\ell^k\equiv\ell^{k'}\pmod{p}$ if $k\equiv k'\pmod{p-1}$ but we might have $\ell^{k/2}\not\equiv\ell^{k'/2}$. Taking this into consideration, we define the following operator.

\begin{definition}
    Let $B$ be a $\Z[1/\ell]$-algebra. The \emph{weight-separating operator}\\
    $S_\ell\colon\M_k(N,B)\to\M_k(N,B)$ is given by $S_\ell(f)=\ell^k f$ with inverse $S_\ell^{-1}(f)=\ell^{-k}f$. If $B$ is also a domain then we can define $S_\ell$ on $\M(N,B)$ by extending this definition by linearity.
\end{definition}

Now, let $B$ be a $\Z[1/\ell]$-domain and $f\in\M_k(N,B)$. In terms of the weight-separating operator, we have
\begin{equation*}
    \D_\ell(f)=\ell^2U_\ell^2 f-S_\ell f,\ \text{and}\ \Tr_\ell(f)=f+\ell S_\ell^{-1}U_\ell W_\ell f.
\end{equation*}
As a result, these operators can be extended to $\M(N,B)$ and we can consider notions of $U_\ell$-new and $\Tr_\ell$-new for Hecke invariant submodules of $\M(N,B)$. Similarly, the operator $T_\ell$ is well-defined on $\M(N/\ell,B)$, where it is given by
\begin{equation*}
    T_\ell=U_\ell+\ell S_\ell W_\ell f.
\end{equation*}
 
\section{Newforms for $\Z[1/\ell]$-domains in Characteristic $p$}
Again let $\ell$ be a prime dividing $N$ exactly once. We now consider how the notions of $\ell$-new, $U_\ell$-new and $\Tr_\ell$-new forms carry over into characteristic $p$ (cf. Section \ref{sec45}). For the rest of this section, we will always assume that $B$ is a $\Z[1/\ell]$-domain of characteristic $p$ unless otherwise stated.\\
\\
We begin by giving an example showing that in general, all three notions of newness disagree on $\M_k(N,B)$. This is in contrast to the case of characteristic zero in which all three notions coincide provided $k\neq 2$ (Theorem \ref{bigcharzerothm}).
\begin{example}\label{excharpmodforms}
    Let $p\geq 3$ and $f=\overline{1}\in\M_0(N,\F_p)$. Note that $f\notin\M_0(N,\F_p)^{\ell-\new}$ since $f$ is the mod $p$ reduction of $1\in\M_0(1)\subseteq\M_0(N)^{\ell-\old}$. However,
    \begin{align*}
        \D_\ell(f)&=\ell^2U_\ell^2(1)-\ell^0=\ell^2-1\ \text{and,}\\
        \Tr_\ell(f)&=1+\ell^{1-0}U_\ell W_{\ell}(f)=1+\ell.
    \end{align*}
    Hence if $\ell\equiv -1\pmod{p}$ then $f\in\M(N,\F_p)^{U_\ell-\new}$ and $f\in\M(N,\F_p)^{\Tr_\ell-\new}$. On the other hand, if $\ell\equiv 1\pmod{p}$ then $f\in\M(N,\F_p)^{U_\ell-\new}$ but $f\notin\M(N,\F_p)^{\Tr_\ell-\new}$. This example also generalises to higher weights by using the mod $p$ reduction of Eisenstein series. 
\end{example}

By restricting to cusp forms, we can overcome much of this problematic behaviour. In this direction, we prove a result (adapted from \cite[Proposition 6.4]{deo2019newforms}) regarding the $\ell$-old forms that appear in $\mcS(N,B)^{U_\ell-\new}$ and $\mcS(N,B)^{\Tr_\ell-\new}$.

\begin{proposition}\label{proploldcaplnew}
    Let $B$ be a $\Z[1/\ell]$-algebra and $f,g\in \mcS_k(N/\ell,B)$ for some fixed weight $k$. The following are equivalent
    \begin{enumerate}[label=(\arabic*)]
        \item $f+W_\ell(g)\in\ker(\D_\ell)$.
        \item $f+W_\ell(g)\in\ker(\Tr_\ell)\cap\ker(\Tr_\ell W_\ell)$.
        \item $\ell T_{\ell}f=-(\ell+1)S_\ell g$ and $\ell T_\ell g=-(\ell+1)f$.
    \end{enumerate}
    In other words,
    \begin{align*}
        (\mcS_k(N,B)^{\ell-\old})^{U_\ell-\new}&=(\mcS_k(N,B)^{\ell-\old})^{\Tr_\ell-\new}\\
        &=\{f+W_\ell(g)\::\:\ell T_{\ell}f=-(\ell+1)S_\ell g\ \text{and}\ \ell T_\ell g=-(\ell+1)f\}
    \end{align*}
    If $B$ is also a domain then the above results extend to the algebra $\mcS(N,B)$.
\end{proposition}
\begin{proof}
    First we show $(1)\Leftrightarrow(3)$. Recall that on $\M(N/\ell,B)$ we have $W_\ell=S_\ell i_\ell$, $U_\ell W_\ell=S_\ell$ and $\ell U_\ell=\ell T_\ell-W_\ell$. So, 
    \begin{align*}
        \D_{\ell}(f)&=\ell^2U_\ell^2f-S_\ell f\\
        &=\ell U_\ell(\ell T_\ell f-W_\ell f)-S_\ell f\\
        &=\ell T_{\ell}(\ell T_\ell f)-W_\ell(\ell T_\ell f)-\ell U_\ell W_\ell f-S_\ell f\\
        &=\ell^2T_{\ell}^2f-(\ell+1)S_\ell f-W_\ell\ell T_\ell f.
    \end{align*}
    Whereas,
    \begin{align*}
        \D_{\ell}(W_{\ell}g)&=\ell^2U_\ell^2W_\ell g-S_\ell W_\ell g\\
        &=\ell^2U_\ell S_\ell g-W_\ell S_\ell g\\
        &=\ell^2T_\ell S_\ell g-\ell W_\ell S_\ell g-W_\ell S_\ell g\\
        &=\ell^2T_\ell S_\ell g-(\ell+1)W_\ell S_\ell g.
    \end{align*}
    Hence, if $f+W_\ell(g)\in\ker(\D_\ell)$ then
    \begin{align*}
        0&=\D_\ell(f+W_\ell g)\\
        &=(\ell^2T_\ell^2f-(\ell+1)S_\ell f+\ell^2T_\ell S_\ell g)-W_{\ell}(\ell T_\ell f+(\ell+1)S_\ell g).
    \end{align*}
    But by Proposition \ref{goodloldprop} we have $\mcS_k(N/\ell,B)\cap W_\ell \mcS_k(N/\ell,B)=\{0\}$. Thus $f+W_\ell(g)\in~\ker(\D_\ell)$ if and only if
    \begin{equation*}
        \ell^2T_\ell^2f-(\ell+1)S_\ell f+\ell^2T_\ell S_\ell g=0\text{ and }\ell T_\ell f+(\ell+1)S_\ell g=0.
    \end{equation*}
    The second equation tells us that $\ell T_\ell f=-(\ell+1)S_\ell g$. Substituting this into the first equation and simplifying then gives
    \begin{equation*}
        \ell T_\ell g=-(\ell+1)f,
    \end{equation*}
    as required.\\
    \\
    For (2)$\Leftrightarrow$(3) we recall that $\Tr_\ell(f)=(\ell+1)f$ and $\Tr_\ell(W_\ell g)=\ell T_\ell(g)$ (see Lemma \ref{lemmatraceprop}). So,
    \begin{equation*}
        \Tr_{\ell}(f+W_\ell g)=0\Leftrightarrow\ (\ell+1)f+\ell T_\ell g=0\Leftrightarrow\ \ell T_{\ell}g=-(\ell+1)f. 
    \end{equation*}
    On the other hand,
    \begin{align*}
        &0=\Tr_\ell W_\ell(f+W_\ell g)=\Tr_\ell(W_\ell f+S_\ell g)\\
        \Leftrightarrow\ &\ell T_\ell f=-(\ell+1)S_\ell g, 
    \end{align*}
    as required. If $B$ is a domain then we can repeat the above argument for $f,g\in\mcS(N/\ell,B)$ by using Proposition \ref{propauto}.
\end{proof}
\begin{remark}
    In fixed weight $k$, this result can be expressed more symmetrically in terms of $w_\ell=\ell^{-k/2}W_\ell$. In particular, if $\lambda_k=-(\ell+1)\ell^{\frac{k-2}{2}}$ then 
    \begin{align*}
        (\mcS_k(N,B)^{\ell-\old})^{U_\ell-\new}&=(\mcS_k(N,B)^{\ell-\old})^{\Tr_\ell-\new}\\
        &=\{f+w_\ell(g)\::\:T_{\ell}f=\lambda_kg\ \text{and}\ T_\ell g=\lambda_kf\}.
    \end{align*}
    The constant $\lambda_k$ also appears in the work of Ribet \cite{ribet1983congruence} and Diamond \cite{diamond1991congruence} on congruences between $\ell$-old and $\ell$-new forms. The connection between these results is explored further in \cite[Section 7]{deo2019newforms}.
\end{remark}

We now move onto the main theorem, comparing the notions of $\ell$-new, $U_\ell$-new and $\Tr_\ell$-new cusp forms in characteristic $p$.

\begin{theorem}\label{mainmodptheorem}
    Let $B$ be a $\Z[1/\ell]$-domain of characteristic $p$. We then have
    \begin{equation*}
        \mcS(N,B)^{\ell-\new}\subseteq\mcS(N,B)^{U_\ell-\new}=\mcS(N,B)^{\Tr_\ell-\new}.
    \end{equation*}
\end{theorem}
\begin{proof}
    Since $B$ is a domain of characteristic $p$ we have $\mcS(N,B)=\bigoplus_{i\in 2\Z/(p-1)\Z}\mcS(N,B)^i$ and each element of $\mcS(N,B)^i$ appears in a single weight. As a result, it will suffice to prove the theorem for some fixed weight $k$. In what follows, we will write $X^B$ to indicate that an operator $X$ is acting on $\mcS_k(N,B)$.\\[8pt]
    First we show that $\mcS_k(N,B)^{\ell-\new}\subseteq\mcS_k(N,B)^{U_\ell-\new}$. So, let $f\otimes b$ be a simple tensor in $\mcS_k(N,B)^{\ell-\new}=\mcS_k(N,\Z)^{\ell-\new}\otimes_\Z B$. We then have
    \begin{equation*}
        \D_\ell^B(f\otimes b)=\D_\ell^\Z(f)\otimes b=0
    \end{equation*}
    noting that $f\in\mcS_k(N,\Z)^{\ell-\new}=\ker(\D_\ell^{\Z})$. Hence $f\otimes b\in\ker(\D_\ell^B)=\mcS_k(N,B)^{U_\ell-\new}$ so that in general, $\mcS_k(N,B)^{\ell-\new}\subseteq\mcS_k(N,B)^{U_\ell-\new}$ as required.\\[8pt]
    We now show that $\mcS_k(N,B)^{U_\ell-\new}=\mcS_k(N,B)^{\Tr_\ell-\new}$ by mimicking the argument in \cite[Theorem 1]{deo2019newforms}. In what follows, we assume the reader is familiar with the basic properties of $p$-adic numbers. A summary can be found in Appendix \ref{padicappend}.\\[8pt]
    First note that it suffices to prove $\mcS_k(N,\F_p)^{U_\ell-\new}=\mcS_k(N,\F_p)^{\Tr_\ell-\new}$. To see why this is sufficient, note that $B$ is an $\F_p$-vector space and thus flat over $\F_p$. Therefore, tensoring with $B$ will preserve the kernels of $\D_\ell$, $\Tr_\ell$ and $\Tr_\ell W_\ell$ (Proposition \ref{preservekernelprop}).\\[8pt]
    Now let $f\in\mcS_k(N,\F_p)^{U_\ell-\new}$. Noting that $\F_p=\Z_p/p\Z_p$, we can lift $f$ to some $\tilde{f}\in~\mcS_k(N,\Z_p)$ so that $f$ is the mod-$p$ reduction of $\tilde{f}$. We take such an $\tilde{f}$ and view it as a modular form over $\Q_p$. Since $\Q_p$ is a field of characteristic zero, we then have $\tilde{f}=\tilde{f}_0^{\old}+\tilde{f}_0^{\new}$ where $\tilde{f}_0^{\old}\in\mcS_k(N,\Q_p)^{\ell-\old}$ and $\tilde{f}_0^{\new}\in \mcS_k(N,\Q_p)^{\ell-\new}$ (see Proposition \ref{fieldspanprop}). Let $b\in\Z_{\geq 0}$ be sufficiently large so that $p^b\tilde{f}_0^{\old}$ and  $p^b\tilde{f}_0^{\new}$ are in $\mcS_k(N,\Z_p)$. That is, 
    \begin{equation*}
        \tilde{f}=p^{-b}(\tilde{f}^{\old}+\tilde{f}^{\new}),
    \end{equation*}
    where $\tilde{f}^{\old}\in\mcS_k(N,\Z_p)^{\ell-\old}$ and $\tilde{f}^{\new}\in\mcS_k(N,\Z_p)^{\ell-\new}$.\\[8pt]
    Since $f\in\ker(\D_{\ell}^{\mathbb{F}_p})$ it follows that $\D_{\ell}^{\mathbb{Z}_p}(\tilde{f})$ is in $p\Z_p[[q]]$. Next we note that $\D_\ell^{\Z_p}(\tilde{f}^{\new})=0$ since $\mcS_k(N,\Z_p)^{\ell-\new}=\mcS_k(N,\Z_p)^{U_\ell-\new}$ (Theorem \ref{dlcharzerothm}). As a result, $\D_{\ell}^{\Z_p}(\tilde{f})=p^{-b}\D_{\ell}^{\Z_p}(\tilde{f}^{\old})$. So, letting $f^{\old}$ be the reduction of $\tilde{f}^{\old}$ mod $p^{b+1}$ we have $f^{\old}\in\ker(\D_{\ell}^{\Z/p^{b+1}\Z})$. By Proposition \ref{proploldcaplnew} we then have $f^{\old}\in\ker(\Tr_{\ell})^{\Z/p^{b+1}\Z}\cap\ker(\Tr_{\ell}W_{\ell})^{\Z/p^{b+1}\Z}$. Lifting back up to characteristic zero we see that both $\Tr_{\ell}^{\Z_p}(\tilde{f}^{\old})$ and $(\Tr_\ell W_\ell)^{\Z_p}(\tilde{f}^{\old})$ are in $p^{b+1}\Z_p[[q]]$.\\[8pt]
    Now, $(\Tr_{\ell})^{\Z_p}(\tilde{f}^{\new})=0$ and hence both $\Tr_{\ell}^{\Z_p}(\tilde{f})$ and $(\Tr_{\ell}W_\ell)^{\Z_p}(\tilde{f})$ are in $p\Z_p[[q]]$. Therefore, $\Tr_{\ell}^{\F_p}(f)=0$ and $(\Tr_\ell W_\ell)^{\F_p}(f)=0$. That is, $f\in \ker(\Tr_{\ell})^{\F_p}\cap\ker(\Tr_{\ell}W)^{\F_p}$. Thus $\ker(\D_{\ell}^{\F_p})\subseteq\ker(\Tr_{\ell})^{\F_p}\cap\ker(\Tr_\ell W_\ell)^{\F_p}$. Reversing all of the steps above gives the reverse containment.
\end{proof}
\begin{remarks}\
\begin{enumerate}[label=(\roman*)]
    \item In \cite[Section 6.1]{deo2019newforms}, the space of $\ell$-new forms $\mcS_k(N,B)^{\ell-\new}:=\mcS_k(N,\Z)\otimes_{\Z}B$ is only defined when $B$ has characteristic zero. However, here we also considered the case where $B$ has characteristic $p$ for completeness.
    \item The relation $\mcS(N,B)^{\ell-\new}\subseteq\mcS(N,B)^{U_\ell-\new}$ holds when $B$ is any domain of characteristic $p$ (even if $\ell$ isn't invertible).
\end{enumerate}
\end{remarks}
\begin{question}
    Is the inclusion $\mcS(N,B)^{\ell-\new}\subseteq\mcS(N,B)^{U_\ell-\new}$ strict?
\end{question}
The author of this thesis admits to not knowing a comprehensive answer to this question. However, the following example shows a case where $\mcS_k(N,B)^{\ell-\new}\subsetneq\mcS_k(N,B)^{U_\ell-\new}$ in fixed weight $k$.
\begin{example}\label{charpcuspformex}
    The space $\mcS_2(11)$ is spanned by a single eigenform
    \begin{equation*}
        f=q-2q^2-q^3+2q^4+q^5+2q^6-2q^7-2q^9+O(q^{10}).
    \end{equation*}
    Similarly, the space $\mcS_2(33)^{3-\new}$ is spanned by a single eigenform
    \begin{equation*}
        g=q+q^2-q^3-q^4-2q^5-q^6+4q^7-3q^8+q^9+O(q^{10}).
    \end{equation*}
    Note that $f\in\mcS_2(33,\Z)^{3-\old}$ and $g\in\mcS_2(33,\Z)^{3-\new}$. Define $h:=f-w_3(f)=f-3i_3(f)$ so that
    \begin{equation*}
        h=q-2q^2-4q^3+2q^4+q^5+8q^6-2q^7+q^9+O(q^{10})\in\mcS_2(33,\Z)^{3-\old}.
    \end{equation*}
    Reducing mod 5 we then obtain
    \begin{align*}
        \overline{g}&=q+q^2+4q^3+4q^4+3q^5+4q^6+4q^7+2q^8+q^9+O(q^{10})\in\mcS_2(33,\F_5)^{3-\new},\ \text{and}\\[8pt] 
        \overline{h}&=q+3q^2+q^3+2q^4+q^5+3q^6+3q^7+q^9+O(q^{10})\in\mcS_2(33,\F_5)^{3-\old}.
    \end{align*}
    By the remark following Proposition \ref{proploldcaplnew} we see that $\overline{h}\in\mcS_2(33,\F_5)^{U_3-\new}$. However, $\overline{h}~\notin~\mcS_2(33,\F_5)^{3-\new}=\text{Span}_{\F_5}\{\overline{g}\}$.
\end{example} 

\section{Further Discussion}
Example \ref{charpcuspformex} from the previous section indicates that the notion of $\ell$-new may differ from the algebraic notions of $U_\ell$-new and $\Tr_\ell$-new in characteristic $p$. Because of this discrepancy, we now summarise the differences between each of these notions and their relative usefulness.\\
\\
First we consider the original motivation for defining newforms over $\C$. In particular, the theory of oldforms and newforms allows us to classify cusp forms in $\mcS_k(N)$ for some weight $k$ and level $N$. This classification is useful due to the desirable properties of newforms that we explored in Section \ref{propsec}. However, in characteristic $p$ this classification breaks down. Namely, it is often the case that
\begin{equation*}
    \mcS_k(N,B)^{\ell-\old}\cap\mcS_k(N,B)^{\ell-\new}\neq\{0\}
\end{equation*}
when $B$ has prime characteristic. See Example \ref{oldnewcong} for an instance of this behaviour in characteristic 2 or \cite[Example 1]{deo2019newforms} for an example in characteristic 7.\\
\\
An alternative way to classify oldforms and newforms over $\C$ is by their Hecke eigenvalues (Lemma \ref{distincteigen}). However, the Hecke operators are not always diagonalisable in characteristic $p$ making such a classification impossible (see \cite{jochnowitz1982congruences}). It is thus evident that $\ell$-old and $\ell$-new forms in characteristic $p$ cannot be used to decompose the space $\mcS_k(N,B)$. Despite this problem, we may still hope to retain algebraic properties of newforms in characteristic $p$.\\
\\
This is where the notions of $U_\ell$-new and $\Tr_\ell$-new are useful. These notions are explicitly defined in terms of the Atkin-Lehner and Hecke operators and thus consist of forms that are well-behaved with respect to these operators. For instance, a $U_\ell$-new form will always have a $U_\ell^2$ eigenvalue of $\ell^{k-2}$. Even in characteristic 0, defining newforms in this way is arguably more robust and elegant than using the Petersson inner product. Moreover, the intersection between $\ell$-old and $U_\ell$-new (or $\Tr_\ell$-new) forms can be described quite simply (Proposition \ref{proploldcaplnew}).\\
\\
An area of future research would be to study applications of $U_\ell$-new and $\Tr_\ell$-new forms to other aspects of modular form theory. Importantly, Deo and Medvedovsky \cite[Section 8]{deo2019newforms} show that these notions of ``newness" appear when looking at Monsky's Hecke-stable filtration \cite{monsky2015hecke}, \cite{monsky2016hecke}.\\
\\
Now, although the notions of $U_\ell$-new and $\Tr_\ell$-new agree on cusp forms (Theorem \ref{mainmodptheorem}), these notions do not agree on the space of all modular forms (Example \ref{excharpmodforms}). Moreover, we can only define $\Tr_\ell$-new forms over commutative rings where $\ell$ is invertible. So although similar, it is unclear which algebraic notion - $U_\ell$-new or $\Tr_\ell$-new, provides a "better" definition of newness. In Chapter \ref{chap6} we shall generalise these notions even further so that their different properties become more apparent.

%% file: Chapter6.tex
\chapter{Further Generalisations}\label{chap6}
In the last chapter we explored Deo and Medvedovsky's notions of $U_\ell$-new and $\Tr_\ell$-new forms in characteristic $p$. These notions were only defined for modular forms with respect to $\Gamma_0(N)$ but there are certainly other types of modular forms that we can consider. In this chapter we introduce the notion of modular forms with character and suggest definitions for $U_\ell$-new and $\Tr_\ell$-new forms in this setting. Moreover, we consider ways of removing the requirement that $N$ is squarefree (i.e. that each prime factor $\ell$ of $N$ divides $N$ exactly once).\\[8pt]
Note that most of the proofs in this chapter will be quite short and instead describe how to generalise proofs from previous chapters. We will also frequently make use of \emph{Dirichlet characters}. For information about these objects, see Appendix \ref{appdirichlet}.\\[8pt]
As per usual we assume that $N$ is a fixed positive integer.

\section{Modular Forms with Character}
We begin by defining modular forms with character over $\C$ and generalise to other commutative rings later.
\begin{definition}
    Let $k\geq 0$ and $\chi$ be a Dirichlet character mod $N$. A \emph{modular form with character} $\chi$ (of weight $k$ and level $N$) is an element of the vector space
    \begin{equation*}
        \M_k^\chi(N)=\left\{f\in\M_k(\Gamma_1(N))\::\:f|_k\begin{smatrix}a&b\\c&d\end{smatrix}=\chi(d)f\ \text{for all}\ \begin{smatrix}a&b\\c&d\end{smatrix}\in\Gamma_0(N)\right\}.
    \end{equation*}
    The space of cusp forms with character $\chi$, denoted $\mcS_k^{\chi}(N)$, is defined analogously.
\end{definition}
In particular, note that $\M_k^{\chi_0}(N)=\M_k(N)$, where $\chi_0$ is the trivial Dirichlet character mod $N$. We also have the following decomposition of $\M_k(\Gamma_1(N))$ \cite[Proposition 9.2]{stein2007modular}:
\begin{equation}\label{characterdecompeq}
    \M_k(\Gamma_1(N))=\bigoplus_{\chi}\M_k^{\chi}(N),
\end{equation}
where the direct sum is over all Dirichlet characters mod $N$. Therefore, instead of directly studying the space $\M_k(\Gamma_1(N))$, we can instead study $\M_k^{\chi}(N)$ for each Dirichlet character $\chi$ mod $N$.\\
\\
We can also define $\mcS_k^{\chi}(N)^{\new}$ in a similar way to how we defined $\mcS_k(N)^{\new}$ in Chapter \ref{chapclassic}. So, suppose that $\chi$ is a Dirichlet character mod $N$ induced by a Dirichlet character $\chi'$ mod $M$. For each $e\mid(N/M)$ we have an embedding map $i_e\colon\mcS_k^{\chi'}(M)\hookrightarrow\mcS_k^{\chi}(N)$ that sends $f\in\mcS_k^{\chi}(M)$ to the function $z\mapsto f(ez)$. We then use the images of these embedding maps (at each level from which $\chi$ is induced) to obtain $\mcS_k^{\chi}(N)^{\old}$, and define $\mcS_k^{\chi}(N)^{\new}$ to be the orthogonal complement of $\mcS_k^{\chi}(N)^{\old}$ with respect to the Petersson inner product.\\
\\
Similarly, for any prime $\ell$ dividing $N$, we define $\mcS_k^{\chi}(N)^{\ell-\old}$ and $\mcS_k^{\chi}(N)^{\ell-\new}$ as was done for $\mcS_k(N)$ in Section \ref{sec31}. Note that if $\chi$ is \textbf{not} induced by a character mod $N/\ell$ then there are no forms arising from level $N/\ell$ and thus $\mcS_k^{\chi}(N)^{\ell-\new}=\mcS_k^{\chi}(N)$. For this reason, we will frequently restrict to the case where $\chi$ is induced by a Dirichlet character mod $N/\ell$.\\
\\
The properties of newforms for $\mcS_k^{\chi}(N)$ naturally generalise those for $\mcS_k(N)$ given by Atkin and Lehner. In particular, Li \cite{li1975newforms} provides analogues of the results from sections \ref{propsec} and \ref{sec31} for forms with character.\\
\\
These definitions also generalise to modular forms for $\Gamma_1(N)$, and by \eqref{characterdecompeq} we have
\begin{equation*}
    \M_k(\Gamma_1(N))^{\ell-\new}=\bigoplus_\chi\M_k^{\chi}(N)^{\ell-\new}.
\end{equation*}
In fact, newforms for $\Gamma_1(N^2)$ can be used to describe newforms for the principal congruence subgroup $\Gamma(N)$. This is achieved by considering a group $\Gamma(N)^*$ that is conjugate to $\Gamma(N)$ and contains $\Gamma_1(N^2)$. See \cite[pp 20-21]{weisinger1977some} for more details.\\
\\
In Chapter 4 we discussed why $\M_k(N)$, $\mcS_k(N)$, $\mcS_k(N)^{\ell-\old}$ and $\mcS_k(N)^{\ell-\new}$ had bases in $\Z[[q]]$. These arguments also hold for modular forms for $\Gamma_1(N)$. Namely, the results we used from Diamond and Im's paper \cite[Corollary 12.3.2 and Corollary 12.4.5]{diamond1995modular} are stated for both $\Gamma_0(N)$ and $\Gamma_1(N)$. However, it is not always true that $\M_k^{\chi}(N)$ has a basis in $\Z[[q]]$. Instead, we have the following weaker result which will allow us to define $\M_k^{\chi}(N,B)$ for a restricted class of commutative rings $B$.

\begin{proposition}\label{chiintegrality}
    Let $\chi$ be a Dirichlet character mod $N$. The spaces $\M_k^{\chi}(N)$, $\mcS_k^{\chi}(N)$, $\mcS_k^{\chi}(N)^{\ell-\old}$ and $\mcS_k^{\chi}(N)^{\ell-\new}$ have bases consisting of forms with Fourier coefficients in $\Z[\chi]$. Here, $\Z[\chi]$ denotes the ring generated by the image of $\chi\colon\Z\to\C$.
\end{proposition}
To prove this result, we will need to introduce a new operator.
\begin{definition}
    Let $a\in(\Z/N\Z)^{\cross}$ and $\gamma=\begin{smatrix}*&*\\ * &\overline{a}\end{smatrix}$ be a matrix in $\Gamma_0(N)$ such that $\overline{a}\equiv a\pmod{N}$. The \emph{diamond operator} $\langle a\rangle$ is the map that sends $f\in\M_k(\Gamma_1(N))$ to $f|_k\gamma$.
\end{definition}
See \cite[Section 5.2]{diamond2005first} as for why this operator is well-defined, along with some other basic results. Importantly, we note that if $f\in\M_k^\chi(N)$ then $\langle a\rangle f=\chi(a)f$.
\begin{proof}[Proof of Proposition \ref{chiintegrality}]
    We only prove the result for $\M_k^{\chi}(N)$ but the proofs for $\mcS_k^{\chi}(N)$ and $\mcS_k^{\chi}(N)^{\ell-\new}$ are identical.\\[8pt]
    Consider the map
    \begin{align*}
        \pi_{\chi}\colon\M_k(\Gamma_1(N))&\to\M_k^{\chi}(N),\quad\pi_\chi(f)=\frac{1}{\varphi(N)}\sum_{a\in(\Z/N\Z)^{\cross}}\overline{\chi}(a)\langle a\rangle f,
    \end{align*}
    where $\varphi(N)=|(\Z/N\Z)^{\cross}|$. We show that $\pi_{\chi}$ is a well-defined projection. So, let $f\in~\M_k(\Gamma_1(N))$ and write $f=\sum_{\psi}f_{\psi}$ where the sum is over all Dirichlet characters $\psi$ mod $N$ and $f_{\psi}\in~\M_k^{\psi}(N)$ (refer to the decomposition in \eqref{characterdecompeq}). Then,
    \begin{align*}
        \pi_\chi(f)&=\frac{1}{\varphi(N)}\sum_{a\in(\Z/N\Z)^{\cross}}\left(\overline{\chi}(a)\langle a\rangle\sum_{\psi}f_\psi\right)\\[8pt]
        &=\frac{1}{\varphi(N)}\sum_{a\in(\Z/N\Z)^{\cross}}\left(\overline{\chi}(a)\sum_{\psi}\psi(a)f_\psi\right)\\[8pt]
        &=\frac{1}{\varphi(N)}\varphi(N)f_{\chi}\qquad(*)\\[8pt]
        &=f_{\chi},
    \end{align*}
    where $(*)$ follows from the orthogonality relations for Dirichlet characters (Proposition \ref{proporthogonality}).\\[8pt]
    Now, let $\mathcal{B}$ be a basis for $\M_k(\Gamma_1(N))$ such that every $f\in\mathcal{B}$ has Fourier coefficients in $\Z$. Since the diamond operator is $\Z$-integral (\cite[Proposition 12.3.11]{diamond1995modular}) it follows that $\pi_{\chi}(\mathcal{B})$ is a spanning set for $\M_k^{\chi}(N)$ consisting of forms with Fourier coefficients in $\Z[\chi]$. Taking a linearly independent subset of $\pi_{\chi}(\mathcal{B})$ then gives the desired result.
\end{proof}
In light of this result we now let $\M_k^{\chi}(N,\Z[\chi])$ be the set of all modular forms in $\M_k^{\chi}(N)$ whose Fourier coefficients are in $\Z[\chi]$. For any $\Z[\chi]$-algebra $B$ we then define
\begin{equation*}
    \M_k^\chi(N,B):=\M_k^{\chi}(N,\Z[\chi])\otimes_{\Z[\chi]}B,
\end{equation*}
and analogously for $\mcS_k^{\chi}(N,B)$, $\mcS_k^{\chi}(N,B)^{\ell-\old}$ and $\mcS_k^{\chi}(N,B)^{\ell-\new}$.\\
\\
Finally we consider the structure of $\M^{\chi}(N,B)=\sum_{k=0}^\infty\M_k^{\chi}(N,B)$. Note that $\M^{\chi}(N,B)$ isn't an algebra if $\chi$ is nontrivial. In particular, if $f$ and $g$ are modular forms with character $\chi$ then $fg$ is a modular form with character $\chi^2$. Regardless, $\M^{\chi}(N,B)$ is still an interesting module to study.

\begin{proposition}\label{charzerodecompcharacter}
    Let $\chi$ be a Dirichlet character mod $N$ and $B$ be a $\Z[\chi]$-domain of characteristic zero. Then $\M^{\chi}(N,B)=\bigoplus_{k=0}^\infty\M_k^\chi(N,B)$.
\end{proposition}
\begin{proof}
    First note that the statement holds for $B=\C$ (see the remark after Proposition \ref{directoverc} in Appendix $B$). Now, suppose that $B$ is a general $\Z[\chi]$-domain of characteristic zero. Since each element in the image of $\chi$ is a root of unity we have $\Z[\chi]=\Z[\zeta_n]$, where $\zeta_n$ is a primitive $n^{th}$ root of unity for some $n\in\Z_{>0}$. So, by Corollary \ref{flatdedekind} $B$ is flat over $\Z[\chi]$. We can then repeat the argument used for $\M(N,B)$ in Chapter 4 (Proposition \ref{charzerodirect}).
\end{proof}
\begin{proposition}\label{modpdecompwithchar}
    Let $B$ be a $\Z[\chi]$-domain of characteristic $p$. Then, 
    \begin{equation}\label{gamma1modpdecomp}
         \M^{\chi}(N,B)=\bigoplus_{i\in \Z/(p-1)\Z}\M^{\chi}(N,B)^i,
    \end{equation}
    where 
    \begin{equation*}
        \M^{\chi}(N,B)^i:=\sum_{k\:\equiv\:i\!\!\!\pmod{p-1}}\M_k^{\chi}(N,B).
    \end{equation*}
\end{proposition}
\begin{proof}
    As with the case of $\M(N,B)$ in Chapter \ref{chap5}, the proposition follows from a result of Goren \cite[Theorem 5.4 in Chapter 4]{goren2002lectures}. In \cite{goren2002lectures} this result is stated for modular forms with respect to $\Gamma_1(N)$ and thus also modular forms with character. However, unlike the case for $\M(N,B)$ we may now have modular forms of odd weight.
\end{proof}
It also turns out that each element of $\M^{\chi}(N,B)^i$ appears in a single weight. To see why this is true we need to define the notion of \emph{even} and \emph{odd} Dirichlet characters.
\begin{definition}
    Let $\chi$ be a Dirichlet character mod $N$. If $\chi(-1)=1$ we say that $\chi$ is \emph{even} and if $\chi(-1)=-1$ we say that $\chi$ is \emph{odd}.
\end{definition}
\begin{remark}
    Since $(\chi(-1))^2=\chi((-1)^2)=\chi(1)=1$ we always have $\chi(-1)=\pm 1$ as implied by the above definition.
\end{remark}
\begin{lemma}
    Let $\chi$ be a Dirichlet character mod $N$ and $f\in\M_k^{\chi}(N)$. If $k$ and $\chi$ have different parity then $f=0$.
\end{lemma}
\begin{proof}
    We only prove the lemma for $k$ even and $\chi$ odd. The argument for $k$ odd and $\chi$ is even is identical.\\[8pt]
    So, suppose that $k$ is even and $\chi(-1)=-1$. By the definition of the slash operator we have
    \begin{equation*}
        f|_k\begin{smatrix}-1&0\\0&-1\end{smatrix}=(-1)^kf=f.
    \end{equation*}
    However, since $f\in\M_k^{\chi}(N)$, we have
    \begin{equation*}
        f|_k\begin{smatrix}-1&0\\0&-1\end{smatrix}=\chi(-1)f=-f.
    \end{equation*}
    Hence $f=-f$ so that $f=0$, as required.
\end{proof}
\begin{proposition}
    Let $B$ be a domain of characteristic $p$ and $i\in\Z/(p-1)\Z$\,. Then each element of $\M^{\chi}(N,B)^i$ appears in a single weight. 
\end{proposition}
\begin{proof}
    We argue as in the case for trivial character in Chapter 5 (Proposition \ref{singleweightprop}). So, let $f\in\M^{\chi}(N,B)^i$ and write $f=f_1+\dots+f_n$ where each $f_j$ is nonzero and of fixed weight $k_j$. If $p\geq 5$ then we can multiply each $f_j$ by a suitable power of $\overline{E}_{p-1}$ so that $f$ appears in a single weight. Now assume that $p=2$ or $p=3$. By the previous lemma, any two of the $f_j$'s differ in weight by a multiple of 2. So, we can similarly multiply each $f_j$ by suitable powers of $\overline{E}_4$ and $\overline{E}_6$ so that $f$ appears in a single weight.
\end{proof}

\section{Atkin-Lehner and Hecke operators on $\M_k^{\chi}(N,B)$}\label{sec62}
We now generalise the Atkin-Lehner and Hecke operators to allow us to extend the notions of $U_\ell$-new and $\Tr_\ell$-new to modular forms with character. From here onwards, we always assume $N>1$ and that $\chi$ is a Dirichlet character mod $N$.\\
\\
On $\M_k^{\chi}(N)$, the Atkin-Lehner operator is the same as our original definition (Definition \ref{atkinlehnerdef}). That is, if $\ell$ is a prime dividing $N$ exactly once and $a,b\in\Z$ are such that $\ell b-a(N/\ell)=1$, then $w_\ell(f)=f|_k\begin{smatrix}\ell&a\\N&\ell b\end{smatrix}$ is well-defined for any $f\in\M_k^{\chi}(N)$.
\begin{proposition}
    Let $\ell$ be a prime dividing $N$ exactly once and assume that $\chi$ is induced by a Dirichlet character mod $N/\ell$. Then for any $f\in\M_k^{\chi}(N)$ we also have $w_\ell(f)\in\M_k^{\chi}(N)$.
\end{proposition}
\begin{proof}
    Let $a$ and $b$ be integers satisfying $\ell b-a(N/\ell)=1$ and let $\gamma_\ell=\begin{smatrix}\ell&a\\N&\ell b\end{smatrix}$. Then, let $\delta=\begin{smatrix}w&x\\ yN&z\end{smatrix}$ be an arbitrary matrix in $\Gamma_0(N)$. Define
    \begin{align*}
        \delta':=\gamma_\ell\delta\gamma_\ell^{-1}&=\begin{pmatrix}\ell&a\\N&\ell b\end{pmatrix}\begin{pmatrix}w&x\\yN&z\end{pmatrix}\frac{1}{\ell}\begin{pmatrix}\ell b&-a\\-N&\ell\end{pmatrix}\\[8pt]
        &=\begin{pmatrix}\ell bw+bayN-Nx-az(N/\ell)& -aw+a^2y(N/\ell)+\ell x+az\\ Nwb+\ell b^2yN-(N/\ell)Nx-Nbz& -a(N/\ell)w-abyN+\ell Nx+\ell bz\end{pmatrix}\in\Gamma_0(N).
    \end{align*}
    Since $\ell b=1+a(N/\ell)$ we have that the lower right entry of $\delta'$ is congruent to $z$ mod $N/\ell$. Therefore, for any $f\in\M_k^{\chi}(N)$ we have
    \begin{equation*}
        (w_\ell f)|_k\delta=f|_k\gamma\delta=f|_k\delta'\gamma=\chi(z)f|_k\gamma,
    \end{equation*}
    so that $w_\ell(f)\in\M_k^{\chi}(N)$ as required.
\end{proof}
Recall that when $\chi$ is trivial, the Atkin-Lehner operator is $\Z[1/\ell]$-integral. For general $\chi$ with conductor $d$, we instead have that $w_\ell$ is $\Z[1/N,\zeta_m]$-integral where $m=
\gcd(d,\ell)$ (see \cite{loeffler2020integrality})\footnote{Based on the case for trivial character, one would expect that $w_\ell$ is actually $\Z[1/\ell,\zeta_m]$-integral. However, the author of this thesis is unaware of any proof for such a result.}. Now, suppose that $\chi$ is induced by a Dirichlet character mod $N/\ell$ as above. Then by Proposition \ref{propinduced}, we have that $m=1$ and thus $w_\ell$ is $\Z[1/N]$-integral. So, if $B$ is a $\Z[1/N,\chi]$-algebra then $w_\ell$ extends to a linear operator on $\M_k^{\chi}(N,B)$. If $B$ is also a domain then the scaled Atkin-Lehner operator $W_\ell=\ell^{k/2}w_\ell$ is well-defined on $\M^{\chi}(N,B)=~\sum_{k=0}^\infty\M_k^{\chi}(N,B)$ (cf. Proposition \ref{propauto}).\\
\\
We now define the Hecke operators on $\M_k^{\chi}(N)$.
\begin{definition}
    Let $p$ be a prime number. The Hecke operators $U_p$ and $T_p$ are defined on $\M_k^{\chi}(N)$ by their effect on the $q$-expansion coefficients:
    \begin{alignat*}{3}
        a_n(T_p f)&=a_n(U_pf)=a_{np}(f),\qquad&&\text{if $p\mid N$,}\\
        a_n(T_p f)&=a_{pn}(f)+\chi(r)r^{k-1}a_{n/p}(f),\qquad&&\text{if $p\nmid N$.}
    \end{alignat*} 
\end{definition}
The reason for defining the Hecke operators as such is given in \cite[Chapter 5.2]{diamond2005first}. Importantly, for any prime $p$ and $f\in\M_k^{\chi}(N)$ we have $T_p(f)\in\M_k^{\chi}(N)$. Also note that if $ k>0$ or $p\mid N$ then $T_p$ is $\Z[\chi]$-integral and this definition extends to $\M_k^{\chi}(N,B)$ for any $\Z[\chi]$-algebra $B$. Otherwise, if $k=0$ and $p\nmid N$ then we also require that $p$ is invertible in $B$.\\
\\
Other properties of the Atkin-Lehner and Hecke operators also naturally generalise from the case of trivial character. In particular, if $f\in\M_k^{\chi}(N)$ and $\chi$ is induced by a character $\chi'$ mod $N/\ell$, then
\begin{enumerate}[label=(\roman*)]
    \item $T_p U_\ell(f)=U_\ell T_p(f)$,
    \item $T_pw_\ell(f)=w_\ell T_p(f)$,
    \item $w_\ell^2(f)=\overline{\chi}'(\ell)f$.
\end{enumerate}
Moreover, if $f\in\M_k^{\chi}(N/\ell)\subseteq\M_k^{\chi}(N)$, then
\begin{enumerate}[label=(\roman*')]
    \item $W_\ell(f)=\ell^k i_\ell(f)$,
    \item $U_\ell W_\ell(f)=\ell^k f$,
    \item $T_\ell(f)=U_\ell(f)+\chi(\ell)\ell^{-1}W_\ell(f)$.
\end{enumerate}
Here, we are assuming that $p$ and $\ell$ are primes such that $p\nmid N$ and $\ell$ divides $N$ exactly once. The proof of these facts are essentially identical to the case for trivial character.

\section{Generalising $U_\ell$-new}\label{sec63}
Assume $\ell$ is a prime dividing $N$ exactly once. Recall that for trivial character, $U_\ell$-new forms were defined by considering the $U_\ell$-eigenvalues of $\ell$-new forms. In particular, if $f\in\M_k(N)^{\ell-\new}$ then
\begin{equation*}
    U_\ell^2(f)=\ell^{k-2}f.
\end{equation*}
Now, suppose that $\chi$ is induced by a Dirichlet character $\chi'$ mod $N/\ell$. If $f\in\M_k^{\chi}(N)^{\ell-\new}$ then by \cite[Theorem 3]{li1975newforms}
\begin{equation*}
    U_\ell^2(f)=\chi'(\ell)\ell^{k-2}f.
\end{equation*}
In light of this result, we now generalise our definition of the $\D_\ell$ operator.
\begin{definition}
    Let $B$ be a $\Z[\chi]$-algebra. If $\chi$ is induced by a Dirichlet character $\chi'$ mod $N/\ell$ then the operator $\D_l^{\chi}\colon\M_k^{\chi}(N,B)\to\M_k^{\chi}(N,B)$ is given by 
    \begin{equation*}
        \D_\ell^{\chi}=\ell^2U_\ell^2-\chi'(\ell)\ell^k.
    \end{equation*}
\end{definition}
If $B$ is also a domain then $\D_\ell^{\chi}$ extends to an operator on $\M^{\chi}(N,B)$.
\begin{definition}
    Let $B$ be a $\Z[\chi]$-algebra. If $\chi$ is induced by a Dirichlet character mod $N/\ell$ then for any Hecke-invariant submodule $\mathcal{C}$ of $\M_k^{\chi}(N,B)$ we define
    \begin{equation*}
        \mathcal{C}^{U_\ell^{\chi}-\new}:=\ker(\D_\ell^{\chi}|_{\mathcal{C}}).
    \end{equation*}
    If $B$ is a $\Z[\chi]$-domain then we can take $\mathcal{C}$ to be any Hecke-invariant submodule of\\
    $\M^{\chi}(N,B)=\sum_{k=0}^\infty\M_k^{\chi}(N,B)$.
\end{definition}
As with the case for trivial character in Chapter \ref{chap4}, this definition of $U_\ell^{\chi}$-new agrees with the standard notion of $\ell$-new in characteristic zero.
\begin{theorem}\label{mainulcharacterthm}
    Let $B$ be a $\Z[\chi]$-domain of characteristic zero. If $\chi$ is induced by a Dirichlet character $\chi'$ mod $N/\ell$ then
    \begin{equation*}
        \mcS^\chi(N,B)^{\ell-\new}=\mcS^{\chi}(N,B)^{U_\ell^{\chi}-\new}.
    \end{equation*}
\end{theorem}
\begin{proof}
    Since $\M^{\chi}(N,B)=\bigoplus_{k=0}^\infty\M_k^\chi(N,B)$ and $\D_\ell^{\chi}$ is weight-preserving, it suffices to prove the theorem in fixed weight $k$. Moreover, since $B$ is flat over $\Z[\chi]$ it suffices to prove the theorem for $B=\Z[\chi]$ or more simply, $B=\C$.\\[8pt]
    One can then repeat the argument from the proof of Theorem \ref{dlcusptheorem} and its proceeding lemmas whilst keeping track of character. In particular, note that the Weil bound \cite[Theorem 8.2]{deligne1974conjecture} holds for modular forms with character. Moreover, for any eigenform $g\in\mcS_k^{\chi}(N/\ell)$, the characteristic polynomial of $U_\ell^2$ on $\text{span}\{g,i_\ell g\}$ is given by
    \begin{equation*}
        P_{\ell,g}(X)=X^2-(a_\ell^2(g)-2\ell^{k-1}\chi'(\ell))X+\ell^{2k-2}(\chi'(\ell))^2.\qedhere
    \end{equation*}
\end{proof}

Unfortunately, there does not appear to be a natural way to extend our definition of $U_\ell^{\chi}$-new to the case when $\ell$ divides $N$ more than once. To see this, we will return to modular forms for $\Gamma_0(N)$ (or equivalently, modular forms with trivial character).\\
\\
Let $\ell$ be a prime such that $\ell^2\mid N$. By Corollary \ref{coruleigen} we have $U_\ell(f)=0$ for all eigenforms $f\in\mcS_k(N)^{\ell-\new}$. In light of this, one may wish to define $\mcS_k(N)^{U_\ell-\new}$ as the kernel of $U_\ell$ (restricted to $\mcS_k(N)$). Certainly $\mcS_k(N)^{\ell-\new}\subseteq\ker(U_\ell)$ but the reverse inclusion may not hold.
\begin{example}
    The space $\mcS_6(4)^{\new}$ is 1-dimensional and spanned by
    \begin{equation*}
        f=q-12q^3+54q^5-88q^7-99q^9+O(q^{11}).
    \end{equation*}
    Since $f\in\mcS_6(4)^{\new}$, it follows that $f\in\ker(U_2)$ (Theorem \ref{hecketheorem}). However, we also have $f\in\mcS_6(8)^{2-\old}$ and thus $f\notin\mcS_6(8)^{2-\new}$. In other words,  $\mcS_6(8)^{2-\new}\neq\ker(U_2)$.
\end{example}

\section{Generalising $\Tr_\ell$-new}\label{sec64}
In order to generalise our notion of $\Tr_\ell$-new forms, we need to define a trace operator on $\M_k^{\chi}(N)$.

\begin{definition}
    Let $\ell$ be a prime dividing $N$ and suppose that $\chi$ is induced by a Dirichlet character $\chi'$ mod $N/\ell$. The trace operator on $\M_k^{\chi}(N)$ is given by:
    \begin{align*}
        \Tr_\ell^{\chi}\colon\M_k^{\chi}(N)&\to\M_k^{\chi'}(N/\ell)\\[8pt]
        f&\mapsto\sum_{\gamma\in\Gamma_0(N)\backslash\Gamma_0(N/\ell)}\overline{\chi'}(\gamma)f|_k\gamma,
    \end{align*}
    where $\overline{\chi'}(\begin{smatrix}a&b\\c&d\end{smatrix}):=\overline{\chi'}(d)$.
\end{definition}
To see why this operator is well-defined, one can refer to \cite[Corollary 10.2.11]{cohen2017modular} which generalises the argument in the proof of Proposition \ref{trwelldefined}.\\
\\
For the rest of this section we assume that $\ell$ is a prime dividing $N$ exactly once. In this case, $\Tr_\ell^{\chi}(f)$ is again explicitly given by
\begin{equation}\label{traceeqchar}
    \Tr_\ell^{\chi}(f)=f+\ell^{1-k}U_\ell W_\ell f,
\end{equation}
which is well-defined on $\M^{\chi}(N,B)$ for any $\Z[1/N,\chi]$-domain $B$. Moreover, if $\chi$ is induced by a character $\chi'$ mod $N/\ell$ then for any $f\in\M_k^{\chi'}(N/\ell)$ we have
\begin{enumerate}[label=(\roman*)]
    \item $\Tr_\ell^{\chi}(f)=(\ell+1)f$, and
    \item $\Tr_\ell^{\chi}(W_\ell f)=\overline{\chi}'(\ell)\ell T_\ell (f)$
\end{enumerate}
(cf. Lemma \ref{lemmatraceprop}). 

\begin{definition}\label{trlnewwithcharacter}
    Let $B$ be a $\Z[1/N,\chi]$-algebra. If $\chi$ is induced by a Dirichlet character $\chi'$ mod $N/\ell$ then for any Hecke-invariant submodule $\mathcal{C}$ of $\M_k^{\chi}(N,B)$ we define
    \begin{equation*}
        \mathcal{C}^{\Tr_\ell^{\chi}-\new}=\ker(\Tr_\ell^{\chi}|_{\mathcal{C}})\cap\ker(\Tr_\ell^{\chi} W_\ell|_{\mathcal{C}}).
    \end{equation*}
    If $B$ is a $\Z[1/N,\chi]$-domain then we can take $\mathcal{C}$ to be any Hecke-invariant submodule of $\M^{\chi}(N,B)=\sum_{k=0}^\infty \M_k^{\chi}(N,B)$.
\end{definition}

Again, we verify that this definition agrees with the standard notion of $\ell$-new in characteristic zero.

\begin{theorem}\label{bigtrlcharacterthm}
    Let $B$ be a $\Z[1/N,\chi]$-domain of characteristic zero. If $\chi$ is induced by a Dirichlet character $\chi'$ mod $N/\ell$ then
    \begin{equation*}
        \mcS^{\chi}(N,B)^{\ell-\new}=\mcS^{\chi}(N,B)^{\Tr_\ell^{\chi}-\new}.
    \end{equation*}
\end{theorem}
\begin{proof}
    Since $\M^{\chi}(N,B)=\bigoplus_{k=0}^\infty\M_k^{\chi}(N,B)$ and $\Tr_\ell^{\chi}$ is weight-preserving, it suffices to prove the theorem in fixed weight $k$. As in the proof of Proposition \ref{charzerodecompcharacter}, we note that $\Z[\chi]=~\Z[\zeta_n]$ for some positive integer $n$. It then follows that $\Z[1/N,\chi]=\Z[1/N,\zeta_n]$. So, $B$ is flat over $\Z[1/N,\chi]$ by Corollary \ref{flatdedekind}. As a result, it suffices to prove the theorem for $B=\Z[1/N,\chi]$ and thus $B=\C$.\\[8pt]
    First we show that $\mcS_k^{\chi}(N,\C)^{\ell-\new}\subseteq\mcS_k^{\chi}(N,\C)^{\Tr_\ell^{\chi}-\new}$. So, let $f$ be an eigenform in $\mcS_k^{\chi}(N,\C)^{\ell-\new}$. For any $p\nmid N$ we have that $T_p$ commutes with $U_\ell$ and $W_\ell$. Hence, $T_p$ also commutes with $\Tr_\ell^{\chi}$ and $\Tr_\ell^{\chi} W_\ell$. Suppose for a contradiction that $\Tr_\ell^{\chi}(f)\neq 0$ or $\Tr_\ell^{\chi} W_\ell(f)\neq~0$. Then $\Tr_\ell^{\chi}(f)$ (or $\Tr_\ell^{\chi}W_\ell(f)$) is an eigenform in $\mcS_k^{\chi'}(N/\ell)$ with the same Hecke eigenvalues away from $\ell$. This is impossible by \cite[Theorem 5]{li1975newforms}, so
    $f\in\ker(\Tr_\ell^{\chi})\cap\ker(\Tr_\ell^{\chi} W_\ell)$ as required. To show that $\mcS_k^{\chi}(N,\C)^{\Tr_\ell^{\chi}-\new}\subseteq\mcS_k^{\chi}(N,\C)^{\ell-\new}$ one can repeat the argument in the proof of Theorem \ref{tracetheorem} whilst keeping track of character.
\end{proof}
Combined with Theorem \ref{mainulcharacterthm}, the above result tells us that the notions of $U_\ell^{\chi}$-new and $\Tr_\ell^{\chi}$-new agree in characteristic zero. The same is true in characteristic $p$.
\begin{theorem}\label{maincharcharpthm}
    Let $B$ be a $\Z[1/N,\chi]$-domain of characteristic $p$. If $\chi$ is induced by a Dirichlet character $\chi'$ mod $N/\ell$ then 
    \begin{equation*}
        \mcS^{\chi}(N,B)^{U_\ell^{\chi}-\new}=\mcS^{\chi}(N,B)^{\Tr_\ell^{\chi}-\new}.
    \end{equation*}
\end{theorem}
\begin{proof}
    First we note that a slight modification of the proof of Proposition \ref{proploldcaplnew} gives
    \begin{align*}
        (\mcS_k^{\chi}(N,B)^{\ell-\old})^{U_\ell^{\chi}-\new}&=(\mcS_k^{\chi}(N,B)^{\ell-\old})^{\Tr_\ell^{\chi}-\new}\\
        &=\{f+W_\ell(g)\::\:\ell T_{\ell}f=-(\ell+1)S_\ell g\ \text{and}\ \ell T_\ell g=-(\ell+1)\chi'(\ell)f\}.
    \end{align*}
    Here, $k$ is some fixed weight, $S_\ell=\ell^k$, and we can take $B$ to be any $\Z[1/N,\chi]$-algebra\footnote{Note that if $f+W_\ell(g)\in (\mcS_k^{\chi}(N,B)^{\ell-\old})^{U_\ell^{\chi}-\new}=(\mcS_k^{\chi}(N,B)^{\ell-\old})^{\Tr_\ell^{\chi}-\new}$, then\\
    $T_\ell^2(f)=T_\ell^2(g)=\chi'(\ell)\ell^{k-2}(\ell+1)^2$. These $T_\ell^2$ eigenvalues are precisely those that appear in Diamond's paper on congruences between $\ell$-old and $\ell$-new forms \cite{diamond1991congruence}.}.\\[8pt] 
    To finish the proof, one can use the same argument as in Theorem \ref{mainmodptheorem}. The main difference here is that we have to work with the ring extensions $\F_p[\chi]$, $\Z_p[\chi]$ and $\Q_p[\chi]$.  
\end{proof}

These results indicate that the notions of $U_\ell^{\chi}$-new and $\Tr_\ell^{\chi}$-new are very natural generalisations of $U_\ell$-new and $\Tr_\ell$-new. \\
\\
Now, at first glance, the notion of $U_\ell^{\chi}$-new seems to be more useful than $\Tr_\ell^{\chi}$-new. In particular, the definition of $U_\ell^{\chi}$-new forms is arguably simpler and is defined over a larger class of commutative rings. However, as we saw in section \ref{sec63}, there does not seem to be a way to extend the idea of $U_\ell^{\chi}$-new forms to the case where $\ell$ divides $N$ more than once. On the other hand, it is actually possible to modify the definition of $\Tr_\ell^{\chi}$-new to such situations. This is the focus of the next section.

\section{Beyond Squarefree Levels}
In the previous section we defined a notion of $\Tr_\ell^{\chi}$-new for each Dirichlet character $\chi$ mod $N$. However, since this definition relied on the Atkin-Lehner operator, we needed to assume that $\ell$ divided $N$ exactly once. If one were to keep this assumption for all primes dividing $N$, then we are left with the requirement that $N$ is squarefree.\\
\\
To overcome this restriction, we suggest a modified definition of $\Tr_\ell^{\chi}$-new that no longer relies on the Atkin-Lehner operator. In particular, we take inspiration from \cite[Theorem 4]{li1975newforms} by making use of the "Fricke" operator.
\begin{definition}
    Let $\gamma_N=\begin{smatrix}0&-1\\N&0\end{smatrix}\in\GL_2(\Q)^{+}$. The \emph{Fricke operator} $h_N$ is the map that sends $f\in \M_k^{\chi}(N)$ to $f|_k\gamma_N$. We also write $H_N$ for the scaling of the Fricke operator given by $H_N=N^{k/2}h_N$.
\end{definition}
We now list some basic properties of the Fricke operator. 
\begin{enumerate}[label=(\roman*)]
    \item If $f\in\M_k^{\chi}(N)$ then $h_N(f)$ maps $\M_k^{\chi}(N)$ to $\M_k^{\overline{\chi}}(N)$.
    \item $h_N^2=(-1)^k$.
    \item If $f$ is a level 1 modular form, then $H_N(f(z))=N^kf(Nz)$.
    \item If $d$ is the conductor of $\chi$ then $h_N$ and $H_N$ extend to linear operators on $\M_k^{\chi}(N,B)$ for any $\Z[1/N,\zeta_d,\chi]$-algebra $B$. If $B$ is also a domain then $H_N$ is well-defined on $\M^{\chi}(N,B)$.
\end{enumerate}
The proof of the final property is discussed in \cite{loeffler2020integrality} where the Fricke operator is denoted $w_{N,k}$. The other properties are analogous to those of the Atkin-Lehner operator and can be proven in a similar fashion.\\
\\
Now, if $\ell$ divides $N$ exactly once, the trace operator is given explicity as in \eqref{traceeqchar}. However, if $\ell$ divides $N$ more than once (i.e. $\ell^2\mid N$), then for any $f\in\M_k^{\chi}(N)$,
\begin{equation*}
    \Tr_{\ell}^{\chi}(f)=\chi(-1)\ell N^{-k} H_{N/\ell}U_\ell H_Nf.
\end{equation*}
This formula (although with slightly different notation and scaling) is derived in Weisinger's thesis \cite[Proposition 16]{weisinger1977some}. So, if $\chi$ is a Dirichlet character with conductor $d$, then $\Tr_\ell^{\chi}$ also extends to $\M^{\chi}(N,B)$ for any $\Z[1/N,\zeta_d,\chi]$-domain $B$.\\
\\
We now define another notion of newness. From here onwards we will no longer assume that $\ell$ divides $N$ exactly once. That is, we allow $\ell$ to be \textbf{any} prime dividing $N$. 
\begin{definition}
    Let $d$ denote the conductor of $\chi$ and $B$ be a $\Z[1/N,\zeta_d,\chi]$-algebra. If $\chi$ is induced by a Dirichlet character $\chi'$ mod $N/\ell$ then for any Hecke-invariant submodule $\mathcal{C}$ of $\M_k^{\chi}(N,B)$ we define
    \begin{equation*}
        \mathcal{C}^{\Tr_\ell^{\chi}-\new'}=\ker(\Tr_\ell^{\chi}|_{\mathcal{C}})\cap\ker(\Tr_\ell^{\chi} H_N|_{\mathcal{C}}).
    \end{equation*}
    If $B$ is a $\Z[1/N,\zeta_d,\chi]$-domain then we can take $\mathcal{C}$ to be any Hecke-invariant submodule of $\M^{\chi}(N,B)=\sum_{k=0}^\infty \M_k^{\chi}(N,B)$.
\end{definition}
\begin{theorem}
    Let $d$ denote the conductor of $\chi$ and let $B$ be a $\Z[1/N,\zeta_d,\chi]$-domain of characteristic zero. If $\chi$ is induced by a Dirichlet character mod $N/\ell$ then
    \begin{equation*}
        \mcS^{\chi}(N,B)^{\ell-\new}=\mcS^{\chi}(N,B)^{\Tr_\ell^{\chi}-\new'},
    \end{equation*}
    and if $\ell$ divides $N$ exactly once,
    \begin{equation*}
        \mcS^{\chi}(N,B)^{\Tr_\ell^{\chi}-\new'}=\mcS^{\chi}(N,B)^{\Tr_\ell^{\chi}-\new}=\mcS^{\chi}(N,B)^{U_\ell-\new}.
    \end{equation*}
\end{theorem}
\begin{proof}
    Arguing as in the proof of Theorem \ref{bigtrlcharacterthm}, it suffices to prove the theorem in fixed weight $k$ and for $B=\C$. However, when $B=\C$, the theorem is well-known and a proof can be found in \cite[Theorem 4]{li1975newforms} or \cite[Theorem 2.2 of Chapter 8]{lang1976introduction}.\\[8pt]
    If $\ell$ divides $N$ exactly once, then the second statement in the theorem follows from Theorem \ref{mainulcharacterthm} and Theorem \ref{bigtrlcharacterthm}.
\end{proof}
\begin{remark}
    Weisinger \cite[Proposition 19]{weisinger1977some} gives a similar result for Eisenstein series in $\M_k^{\chi}(N,\C)$.
\end{remark}
We can also hope that our notions of newness agree when $B$ has characteristic $p$.
\begin{question}
    If $B$ is $\Z[1/N,\zeta_d,\chi]$-domain of characteristic $p$ do we also have
    \begin{equation*}
        \mcS^{\chi}(N,B)^{\Tr_\ell^{\chi}-\new'}\stackrel{\textbf{?}}{=}\mcS^{\chi}(N,B)^{\Tr_\ell^{\chi}-\new}=\mcS^{\chi}(N,B)^{U_\ell-\new}
    \end{equation*}
    whenever $\ell$ divides $N$ exactly once? Considering the proof of Theorem \ref{maincharcharpthm}, it would suffice to show that
    \begin{equation*}
        (\mcS_k^{\chi}(N,B)^{\ell-\old})^{\Tr_\ell^{\chi}-\new'}\\
        \stackrel{\textbf{?}}{=}\{f+W_\ell(g)\::\:\ell T_{\ell}f=-(\ell+1)S_\ell g\ \text{and}\ \ell T_\ell g=-(\ell+1)\chi'(\ell)f\}.
    \end{equation*}
\end{question}
A positive answer to this question would give further support for using the notion of $\Tr_\ell^{\chi}$-new$'$.\\
\\
As a final remark, we note that besides the Fricke operator, there are certainly other operators used in literature that can generalise or replace the Atkin-Lehner operator. For instance, there is Li's $V_q^M$ operator \cite[pp 287-288]{li1975newforms} or Weisinger's partial $W$-operators \cite[page 30]{weisinger1977some}. Each of these operators can be used to define notions of newness. Thus, a direction for further research would be to the study these other notions of newness and see how they compare to our use of the Fricke operator here.

%% file: AppendixA.tex
\chapter{Commutative Algebra}\label{appcommalg}

\section{Flat Modules}\label{flatappend}
Flat modules are a widely studied class of modules that are well-behaved with respect to the tensor product. They are defined as follows.

\begin{definition}
    Let $A$ be a commutative ring. An $A$-module $M$ is \textit{flat} if it satisfies any of the following equivalent conditions:
    \begin{enumerate}[label=(\arabic*)]
        \item If $f\colon N\rightarrow N'$ is an injective morphism of $A$-modules then $f\otimes 1\colon N\otimes M\to N'\otimes M$ is also injective.
        \item If $0\to N'\to N\to N''\to 0$ is a short exact sequence of $A$-modules then the sequence $0\to N'\otimes M\to N\otimes M\to N''\otimes M\to 0$ is also exact.
        \item If $\mathscr{S}$ is any exact sequence of $A$-modules then $\mathscr{S}\otimes M$ is also exact. 
    \end{enumerate}
\end{definition}
The equivalence of (1) and (2) follows from the right-exactness of the tensor product \cite[Proposition 2.18]{atiyahintroduction}. Then, (2) and (3) are equivalent since any exact sequence can be constructed from short exact sequences.

\begin{example}\
\begin{itemize} 
    \item A commutative ring $A$ is flat as a module over itself since $M\otimes A\cong M$ for any $A$-module $M$. More generally, we have that any free $A$-module is flat.
    \item The $\Z$-module $\Z/2\Z$ is \textbf{not} flat. To see this, consider the injective map
    \begin{equation*}
        \Z\xrightarrow{\times 2}\Z,
    \end{equation*}
    that sends an integer $x$ to $2x$. Tensoring with $\Z/2\Z$ we then obtain
    \begin{equation*}
        \Z\otimes\Z/2\Z\xrightarrow{\times 2\otimes 1}\Z\otimes\Z/2\Z
    \end{equation*}
    which maps everything to zero and is therefore not injective.
\end{itemize}
\end{example}
If $A$ is a Dedekind domain\footnote{A Dedekind domain is an integrally closed, Noetherian domain such that every nonzero prime ideal is maximal. Examples include principal ideal domains and rings of integers of algebraic number fields. For more information, see \cite[Chapter 9]{atiyahintroduction}.}, then a useful description of flat $A$-modules is well-known.
\begin{proposition}\label{flatpid}
    Let $A$ be a Dedekind domain. An $A$-module $M$ is flat if and only if $M$ is torsion-free.
\end{proposition}
\begin{proof}
    See \cite[VI Theorem 9.10]{fuchs2001modules} for a more general proof for Prüfer domains (which include Dedekind domains).
\end{proof}
For our purposes, we are interested in the Dedekind domains $\Z$, $\Z[1/m]$, $\Z[\zeta_n]$ and $\Z[1/m,\zeta_n]$ where $m$ and $n$ are positive integers and $\zeta_n$ is a primitive $n^{\text{th}}$ root of unity.

\begin{corollary}\label{flatdedekind}
    Let $A=\Z$, $\Z[1/m]$, $\Z[\zeta_m]$ or $\Z[1/m,\zeta_n]$. If $B$ is an $A$-domain of characteristic zero, then $B$ is flat (as an $A$-module).
\end{corollary}
\begin{proof}
    Apply the previous proposition noting that any such $B$ is torsion-free as an $A$-module.
\end{proof}

Finally, we show how flat modules "preserve" kernels.
\begin{proposition}\label{preservekernelprop}
    Let $A$ be a commutative ring and $f\colon N\rightarrow N'$ be an $A$-module homomorphism. If $M$ is a flat $A$-module then $\ker(f\otimes 1_M)\cong\ker(f)\otimes M$.
\end{proposition}
\begin{proof}
    Consider the exact sequence
    \begin{equation*}
        0\rightarrow\ker(f)\rightarrow N\xrightarrow{f}N'\rightarrow N'/\im(f)\rightarrow 0.
    \end{equation*}
    Tensoring with $M$ then gives the exact sequence 
    \begin{equation*}
        0\rightarrow \ker(f)\otimes M\rightarrow N\otimes M\xrightarrow{f\otimes 1_M}N'\otimes M\rightarrow N'/\im(f)\otimes M\rightarrow 0.
    \end{equation*}
    The exactness at $N\otimes M$ and injectivity of $\ker(f)\otimes M\rightarrow N\otimes M$ implies that\\
    $\ker(f\otimes 1_M)\cong\ker(f)\otimes M$ as required.
\end{proof}

\section{The $p$-adic Numbers}\label{padicappend}
The $p$-adic numbers are an increasingly important tool in modern mathematics that provide a link between the worlds of characteristic 0 and characteristic $p$. We begin by defining the $p$-adic integers; a subset of the $p$-adic numbers.
\begin{definition}
    Let $p$ be a prime number. A \textit{$p$-adic integer} $x$ is a formal power series of the form $x=\sum_{n=0}^\infty a_np^n$ where $a_n\in\{0,1,\dots,p-1\}$. The set of $p$-adic integers is denoted $\Z_p$. 
\end{definition}
Alternatively, we could more precisely define $\Z_p:=\Z[[X]]/(X-p)$. Using the ring structure of $\Z[[X]]$, the $p$-adic integers form an integral domain of characteristic $0$. Unlike the usual integers $\Z$, we are able to invert any element $x\in\Z_p$ provided $p$ does not divide $x$ (i.e. $a_0\neq 0$). For instance, if $p=3$ then 
\begin{equation*}
    \frac{1}{2}=2+2p+2p^2+2p^3+\cdots.
\end{equation*}
To see this we note that $2(2+2(3)+2(3)^2+2(3)^3+\cdots)=1+0p+0p^2+\cdots=1$.\\
\\
If we now allow $p$ to be invertible, we obtain the $p$-adic numbers.
\begin{definition}
    Let $p$ be a prime number. A \textit{$p$-adic number} $x$ is a formal Laurent series of the form $x=\sum_{n=-m}^\infty a_np^n$ for some $m\in\Z$. The set of $p$-adic numbers is denoted $\Q_p$.
\end{definition}
Since we can now divide by $p$, the $p$-adic numbers $\Q_p$ form a field. For our purposes, studying modular forms over $\Z_p$ or $\Q_p$ can provide insight into the theory of modular forms over $\F_p=\Z_p/p\Z_p$.

\section{Dirichlet Characters}\label{appdirichlet}
Dirichlet characters are arithmetic functions frequently studied in analytic number theory. They are used when classifying modular forms for $\Gamma_1(N)$.

\begin{definition}
    Let $N$ be a positive integer and $\chi\colon(\Z/N\Z)^{\cross}\to\C^{\cross}$ be a group homomorphism. We can extend $\chi$ to a function on all of $\Z$ by setting
    \begin{equation*}
        \chi(a)=\begin{cases}
            \chi(a\ \text{mod}\ N),&\text{if}\ \gcd(a,N)=1\\
            0,&\text{if}\ \gcd(a,N)>1.
        \end{cases}
    \end{equation*}
    Such a function $\chi\colon\Z\to\C$ is called a \emph{Dirichlet character} mod $N$.
\end{definition}
For any Dirichlet character $\chi$ we have the following properties:
\begin{itemize}
    \item $\chi(1)=1$ since $\chi\colon(\Z/N\Z)^{\cross}\to\C^{\cross}$ is a group homomorphism.
    \item $|\chi(a)|=1$ if $\gcd(a,N)=1$. To see this we let $\varphi(N)=|(\Z/N\Z)^{\cross}|$ and then note that $a^{\varphi(N)}\equiv 1$ by Euler's theorem. Hence,
    \begin{equation*}
        \chi(a)^{\varphi(N)}=\chi(a^{\varphi(N)})=\chi(1)=1,
    \end{equation*}
    so that $\chi(a)$ is a $\varphi(N)^{\text{th}}$ root of unity and $|\chi(a)|=1$ as required.
    \item There are exactly $\varphi(N)$ Dirichlet characters mod $N$. A simple proof of this result can be found in \cite[Proposition 17.5]{sutherland2015number}.
\end{itemize}

\begin{example}\label{dirichletexample}\
    \begin{enumerate}[label=(\arabic*)]
        \item For any positive integer $N$ we have the Dirichlet character
        \begin{equation*}
            \chi_0(a)=\begin{cases}
                1,&\text{if}\ \gcd(a,N)=1\\
                0,&\text{if}\ \gcd(a,N)>1.
            \end{cases}
        \end{equation*}
    called the \emph{trivial Dirichlet character} mod $N$.
    \item
    There are two Dirichlet characters mod 4. The first is the trivial character from the previous example. The second is given by
    \begin{equation*}
        \chi(a)=\begin{cases}
            0,&\text{if}\ a\equiv 0,2 \pmod{4}\\
            1,&\text{if}\ a\equiv 1 \pmod{4}\\
            -1,&\text{if}\ a\equiv 3 \pmod{4}\\
        \end{cases}
    \end{equation*}
    \item Let $M$ and $N$ be positive integers such that $M$ divides $N$. If $\chi$ is a Dirichlet character mod $M$ then we can define a Dirichlet character $\chi'$ mod $N$ by setting $\chi'(a)=\chi(a)$ whenever $\gcd(a,N)=1$. In this situation, we say that $\chi'$ is \emph{induced} by $\chi$. The smallest value of $M$ such that $\chi$ is induced by a character mod $M$ is called the \emph{conductor} of $\chi$.
    \end{enumerate}
\end{example}
We now prove a couple of important results regarding Dirichlet characters that are used in Chapter \ref{chap6} of this thesis.

\begin{proposition}\label{propinduced}
    Let $\chi$ be a Dirichlet character mod $N$ that is induced by a Dirichlet character mod $M$. Then, the conductor of $\chi$ divides $M$.
\end{proposition}
\begin{proof}
    We use an equivalent definition of induced characters shown in \cite[Lemma 17.20]{sutherland2015number}. Namely, for any $N'\mid N$, we have that $\chi$ is induced by a character mod $N'$ if and only if $\chi$ is constant on all residue classes in $(\Z/N\Z)^{\cross}$ that are congruent mod $N'$.\\[8pt]
    Let $d$ denote the conductor of $\chi$. Since $\chi$ is induced by Dirichlet characters mod $d$ and mod $M$, we have that $\chi$ is constant on all residue classes in $(\Z/N\Z)^{\cross}$ that are congruent mod $\gcd(d,M)$. That is, $\chi$ is induced by a character mod $\gcd(d,M)$. However, since $d$ is the conductor of $\chi$ we must have $d\leq\gcd(M,d)$. This is only possible if $d=\gcd(M,d)$ or equivalently $d\mid M$.  
\end{proof}

\begin{proposition}\label{proporthogonality}
    Let $\chi$ and $\psi$ be two Dirichlet characters mod $N$. Then
    \begin{equation*}
        \sum_{a\in(\Z/N\Z)^{\cross}}\overline{\chi}(a)\psi(a)=
        \begin{cases}
            \varphi(n),&\text{if}\ \chi=\psi\\
            0,&\text{otherwise}
        \end{cases}
    \end{equation*}
\end{proposition}
\begin{proof}
    First we show that
    \begin{equation}\label{orthoeq}
        \sum_{a\in(\Z/N\Z)^{\cross}}\chi(a)=
        \begin{cases}
            \varphi(n),&\text{if}\ \chi=\chi_0\\
            0,&\text{otherwise}
        \end{cases}
    \end{equation}
    If $\chi=\chi_0$ then \eqref{orthoeq} follows since $\chi_0(a)=1$ for all $a\in(\Z/N\Z)^{\cross}$. Now suppose that $\chi\neq\chi_0$. Then there exists $b\in(\Z/N\Z)^{\cross}$ such that $\chi(b)\neq 0$. Since $(\Z/N\Z)^{\cross}$ is a group, $ba$ runs over all the elements of $(\Z/N\Z)^{\cross}$ as $a$ does. So, setting $S_{\chi}=\sum_{a\in(\Z/N\Z)^{\cross}}\chi(a)$ we have
    \begin{equation*}
            \chi(b)S_{\chi}=\chi(b)\sum_{a\in(\Z/N\Z)^{\cross}}\chi(a)=\sum_{a\in(\Z/N\Z)^{\cross}}\chi(ba)=S_{\chi}.
    \end{equation*}
    Hence $S_{\chi}=0$ as required.\\[8pt]
    Replacing $\chi$ with $\overline{\chi}\psi$ in \eqref{orthoeq} gives
    \begin{equation*}
        \sum_{a\in(\Z/N\Z)^{\cross}}\overline{\chi}(a)\psi(a)=
        \begin{cases}
            \varphi(n),&\text{if}\ \overline{\chi}\psi=\chi_0,\\
            0,&\text{otherwise.}
        \end{cases}
    \end{equation*}
    Now, for any $a\in(\Z/N\Z)^{\cross}$ we have $|\chi(a)|=1$ and thus $\overline{\chi}(a)=\chi^{-1}(a)$. Hence $\overline{\chi}{\psi}=\chi_0$ is equivalent to $\chi=\psi$.
\end{proof}

%% file: AppendixB.tex
\chapter{Assorted Proofs}\label{appassort}

\begin{proposition}\label{ellolddecomp}
    Let $N>1$ and $k\geq 0$. Then,
    \begin{equation*}
        \M_k(N)^{\old}=\sum_{\substack{\ell\ \text{prime}\\\ell\mid N}}\M_k(N)^{\ell-\old}.
    \end{equation*}
\end{proposition}
\begin{proof}
    We note that one inclusion is immediate, namely
    \begin{equation*}
         \sum_{\substack{\ell\ \text{prime}\\\ell\mid N}}\M_k(N)^{\ell-\old}\subseteq\M_k(N)^{\old}.
    \end{equation*}
    For the other inclusion, we show that for each $M\mid N$ (with $M\neq N$) and $e\mid(N/M)$ there exists a prime $\ell$ such that $\ell\mid N$ and $i_e(\M_k(M))\subseteq \M_k(N)^{l-\old}$. So, let $M$ be a proper divisor of $N$ and $e\mid(N/M)$. We consider two cases.\\[8pt]
    \underline{Case 1:} $e=1$\\
    In this case, take $\ell$ to be a prime such that $M\mid(N/\ell)$. We can do this since $M$ is a proper divisor of $N$. Now, 
    \begin{equation*}
        i_1(\M_k(M))=\M_k(M)\subseteq\M_k(N/\ell)\subseteq\M_k(N/\ell)+i_\ell(\M_k(N/\ell))=\M_k(N)^{\ell-\old},
    \end{equation*}
    as required.\\[8pt]
    \underline{Case 2:} $e>1$\\
    In this case, we choose $\ell$ to be such that $\ell\mid e$. Since $e\mid(N/M)$ this means that in particular, $\ell\mid N$ and $M\mid(N/\ell)$. Now, let $g$ be an arbitrary element of $i_e(\M_k(M))$. In particular, $g(z):=f(ez)$ for some form $f\in\M_k(M)$. We wish to show that $g\in\M_k(N)^{\ell-\old}$.\\
    \\
    So, let $h(z):=f(\frac{e}{\ell}z)$. Note that $h\in\M_k(N/\ell)$ as it is the image of $f$ under the map $i_{e/\ell}:\M_k(M)\to\M_k(N/\ell)$. We then have $g=i_\ell(h)\in\M_k(N)^{\ell-\old}$ as required.
\end{proof}

\begin{proposition}\label{directoverc}
    For any nonnegative integer $N$, the algebra $\M(N)=\sum_{k=0}^\infty\M_k(N)$ is a direct sum, i.e. $\M(N)=\bigoplus_{k=0}^\infty \M_k(N)$.
\end{proposition}
\begin{proof}
    The following proof is adapted from \cite[Lemma 2.1.1]{miyake2006modular}. Let $n\in\Z_{\geq 0}$ and $\{f_0,\dots,f_n\}$ be a set of functions with $f_k\in\M_k(N)$. Now suppose that $\sum_{k=0}^nf_k=0$. We want to show that this implies $f_k=0$ for all $k\in\{0,\dots,n\}$. First we construct a sequence of matrices $\{\gamma_m\}_{m=0}^\infty$ in $\Gamma_0(N)$ given by 
    \begin{equation*}
        \gamma_m=\begin{pmatrix}mN+1&1\\mN&1\end{pmatrix}\in\Gamma_0(N),
    \end{equation*}
    so that $j(\gamma_m,z)=mNz+1$. This means that provided $z\neq 0$ we have $j(\gamma_m,z)\neq j(\gamma_{m'},z)$ for $m'\neq m$. If $z=0$ we could alternatively define $\gamma_m=\begin{smatrix}1&1\\ mN&mN+1\end{smatrix}$.
    As each $f_k$ is in $\M_k(N)$ we then have for any $m\in\Z_{\geq 0}$ and $z\in\mathbb{H}$
    \begin{equation}\label{b1identity}
        \sum_{k=0}^n j(\gamma_m,z)^kf_k(z)=\sum_{k=0}^n f_k(\gamma_m z)=0.
    \end{equation}
    Varying $m$ from $0$ to $n$ in \eqref{b1identity} leads to a system of linear equations that can be written in matrix form as
    \begin{equation}\label{vandermondeeq}
        \begin{pmatrix}1&j(\gamma_0,z)&j(\gamma_0,z)^2&\dots&j(\gamma_0,z)^n\\1&j(\gamma_1,z)&j(\gamma_1,z)^2&\dots&j(\gamma_1,z)^n\\
        \vdots&\vdots&\vdots&\ddots&\vdots\\1&j(\gamma_n,z)&j(\gamma_n,z)^2&\dots&j(\gamma_n,z)^n
        \end{pmatrix}\begin{pmatrix}f_0(z)\\f_1(z)\\\vdots\\f_n(z)\end{pmatrix}=\begin{pmatrix}0\\0\\\vdots\\0\end{pmatrix}
    \end{equation}
    If we let $M$ denote the $(n+1)\times (n+1)$ matrix in $\eqref{vandermondeeq}$ then
    \begin{equation*}
        \det(M)=\prod_{1\leq m'\leq m\leq n}(j(\gamma_{m'},z)-j(\gamma_m,z))\neq 0,
    \end{equation*}
    noting that $M$ is a Vandermonde matrix. Since this holds for all $z\in\mathbb{H}$ we have $f_k=0$ for all $k\in\{0,\dots,n\}$ as required.
\end{proof}
\begin{remark}
    More generally we have that $\M^{\chi}(N)=\bigoplus_{k=0}^\infty\M_k(N)$ for any Dirichlet character $\chi$ mod $N$. The proof of this fact is almost identical to the one above. In particular, note that the lower right entry of each $\gamma_m$ is congruent to 1 mod $N$ so that if $f_k\in\M_k^{\chi}(N)$ then $f_k(\gamma_m z)=\chi(1)j(\gamma_m,z)^kf_k(z)=j(\gamma_m,z)^kf_k(z)$.
\end{remark}

\begin{proposition}\label{trwelldefined}
    Let $M,N$ be integers such that $M|N$ and $R=\{R_i\}_{i=1}^n$ be a set of right coset representatives for $\Gamma_0(N)\backslash\Gamma_0(M)$. Let $f\in\M_k(N)$ and $F=\sum_{i=1}^nf|_kR_i$. Then $F\in\M_k(M)$.
\end{proposition}
\begin{proof}
    Let $\delta\in\Gamma_0(M)$. Note that right multiplication by $\delta$ permutes the coset representatives of $\Gamma_0(N)\backslash\Gamma_0(M)$. That is, $R_i\delta=\gamma_iR_{\sigma(i)}$ where $\sigma:\{1,\dots,n\}\to\{1,\dots,n\}$ is a bijection. We then have
    \begin{equation*}
        F|\delta=\sum_{i=1}^nf|_kR_i|_k\delta\\[8pt]=\sum_{i=1}^nf|_kR_i\delta\\[8pt]=\sum_{i=1}^nf|_k\gamma_iR_{\sigma(i)}=\sum_{i=1}^nf|_kR_{\sigma(i)}=F,
    \end{equation*}
    so that $F\in\M_k(M)$ as required.
\end{proof}
\begin{remark}
    The above proposition also holds if replace $\Gamma_0(M)$ and $\Gamma_0(N)$ with arbitrary congruence subgroups $\Gamma$ and $\Gamma'$ satisfying $\Gamma\subseteq\Gamma'$.
\end{remark}

\begin{proposition}\label{eisisold}
    Let $k>2$ and $N>1$. We then have
    \begin{equation*}
        \E_k(N)=\E_k(N)^{\ell-\old}=\E_k(N/\ell)\oplus W_\ell(\E_k(N/\ell))
    \end{equation*}
    whenever $\ell$ is a prime dividing $N$ exactly once.
\end{proposition}
\begin{proof}
    We proceed by a dimension counting argument. First note that $W_\ell$ is a $\C$-vector space isomorphism (with inverse $W_\ell^{-1}=\ell^{-k}W_\ell$) and thus $\dim_{\C}(\E_k(N/\ell))=\dim_{\C}(W_{\ell}(\E_k(N/\ell)))$. It therefore suffices to prove that $\dim_\C(\E_k(N))=2\dim_\C(\E_k(N/\ell))$.\\[8pt]
    Fortunately, there is an explicit formula for the dimension of $\E_k(N)$ \cite[Proposition 6.1]{stein2007modular}:
    \begin{equation*}
        \dim_\C(\E_k(N))=\sum_{d\mid N}\phi\left(\gcd\left(d,\frac{N}{d}\right)\right),
    \end{equation*}
    where $\phi$ is the Euler totient function. Hence,
    \begin{align*}
        \dim(\E_k(N))&=\sum_{d\mid N}\phi\left(\gcd\left(d,\frac{N}{d}\right)\right)\\[8pt]
        &=\sum_{d\mid(N/\ell)}\phi\left(\gcd\left(d,\frac{N}{d}\right)\right)+\sum_{\substack{d=\ell e\\ e\mid (N/\ell)}}\phi\left(\gcd\left(d,\frac{N}{d}\right)\right)\qquad\text{(Since $\ell\nmid N/\ell$.)}\\[8pt]
        &=\sum_{d\mid(N/\ell)}\phi\left(\gcd\left(d,\frac{(N/\ell)}{d}\right)\right)+\sum_{e\mid(N/\ell)}\phi\left(\gcd\left(e,\frac{(N/\ell)}{e}\right)\right)\\[8pt]
        &=2\dim(\E_k(N/\ell)),
    \end{align*}
    as required.
\end{proof}
\begin{remark}
    The above proposition can also be proven using explicit bases for $\E_k(N)$ and $\E_k(N/\ell)$ as described in \cite[Section 4.5]{diamond2005first}.
\end{remark}

\begin{proposition}\label{propbernoulli}
    Let $p$ be a prime number. Then $pB_{p-1}\equiv -1$ (mod $p$), where $B_{p-1}$ is the $(p-1)^{\text{th}}$ Bernoulli number. As a result, the denominator of $B_{p-1}$ is divisible by $p$.
\end{proposition}
\begin{proof}
    Since $B_1=-\frac{1}{2}$ the statement holds for $p=2$ so we will assume $p\geq 3$. Now, recall that the Bernoulli numbers are defined implicitly by
    \begin{equation*}
        \sum_{k=0}^\infty B_k\frac{t^k}{k!}=\frac{t}{e^t-1}.
    \end{equation*}
    Multiplying both sides by $(e^{pt}-1)/t$ gives
    \begin{equation}\label{bernoullisum}
        \frac{e^{pt}-1}{t}\sum_{k=0}^\infty B_k\frac{t^k}{k!}=\frac{e^{pt}-1}{t}\frac{t}{e^t-1}=\sum_{r=0}^{p-1}e^{rt}.
    \end{equation}
    Expanding $e^{rt}$ as a Taylor series in $t$, we then have that the coefficient of $t^n/n!$ on the right-hand side of \eqref{bernoullisum} is
    \begin{equation}\label{rhsbernoulli}
        \sum_{r=1}^{p-1}r^n.
    \end{equation}
    On the other hand, the left-hand side of \eqref{bernoullisum} becomes
    \begin{equation*}
        \frac{e^{pt}-1}{t}\sum_{k=0}^{\infty}B_k\frac{t^k}{k!}=\sum_{k=0}^{\infty}\frac{p^{k+1}t^{k}}{(k+1)!}\sum_{k=0}^\infty\frac{B_k}{k!}t^k=\sum_{n=0}^\infty\sum_{k=0}^n\frac{B_k p^{n-k+1}}{k!(n-k+1)!}t^n
    \end{equation*}
    so that the coefficient of $t_n/n!$ is
    \begin{equation*}
        \sum_{k=0}^n\frac{n!}{k!(n-k+1)!}B_kp^{n-k+1}=pB_n+\sum_{k=0}^{n-1}\binom{n}{k}pB_k\frac{p^{n-k}}{n-k+1}.
    \end{equation*}
    Equating with \eqref{rhsbernoulli} then gives
    \begin{equation*}
        pB_n=-\sum_{k=0}^{n-1}\binom{n}{k}pB_k\frac{p^{n-k}}{n-k+1}+\sum_{r=1}^{p-1}r^n.
    \end{equation*}
    Since $p\geq 3$ we always have $p^{n-k}>n-k+1$ so that $p^{n-k}/(n-k+1)$ is $p$-integral \footnote{A rational number $x=\frac{a}{b}$ is said to be \emph{$p$-integral} if $\nu_p(a)\geq\nu_p(b)$. Here, $\nu_p(a)$ and $\nu_p(b)$ denote the number of times that $p$ divides $a$ and $b$ respectively.} for each $k\in\{0,\dots,n-1\}$. Then, noting that $B_0=1$ it follows by induction that $pB_n$ is $p$-integral for all $n\geq 0$. Namely, this holds for our desired case $n=p-1$ in which
    \begin{equation}\label{bernoulliatp}
        pB_{p-1}=-\sum_{k=0}^{p-2}\binom{p-1}{k}pB_k\frac{p^{p-1-k}}{p-k}+\sum_{r=1}^{p-1}r^{p-1}.
    \end{equation}
    Again noting that $p^{p-1-k}>p-k$ we must have
    \begin{equation*}
        \frac{p^{p-1-k}}{p-k}\equiv 0\pmod{p}.
    \end{equation*}
    So, reducing \eqref{bernoulliatp} mod $p$ gives
    \begin{equation*}
        pB_{p-1}\equiv\sum_{r=1}^{p-1}r^{p-1}\equiv\sum_{r=1}^{p-1}1\equiv -1\pmod{p},
    \end{equation*}
    as required.
\end{proof}
\begin{remark}
    The above proof was adapted from \cite[Chapter X Theorem 2.1]{lang1976introduction}
\end{remark}

%% file: Notation.tex
\chapter*{Notation Index}
\addcontentsline{toc}{chapter}{Notation Index}

\newcommand{\ite}[1]{\item[#1 \hfill]}

\begin{list}{}
   {\setlength{\labelwidth}{3cm} 
    \setlength{\labelsep}{0cm}
    \setlength{\leftmargin}{3.5cm}
    \setlength{\rightmargin}{0cm}
    \setlength{\itemsep}{0cm} 
    \setlength{\parsep}{0.05cm} 
    \setlength{\itemindent}{0cm}
    \setlength{\listparindent}{0cm}}
\ite{$\H$}    {complex upper half plane},
\ite{$\Z_p$}    {the $p$-adic integers},
\ite{$\Q_p$}        {the $p$-adic numbers},
\ite{$j(\gamma,z)$}     {factor of automorphy},
\ite{$|_k\gamma$}  {slash operator of weight $k$; $f|_k\gamma=\det(\gamma)^{k/2}j(\gamma,z)^{-k} f(\gamma\cdot z)$},
\ite{$E_k(z)$}      {weight $k$ Eisenstein series normalised so that $a_0=1$},
\ite{$\Gamma$}      {congruence subgroup},
\ite{$\chi$}   {a Dirichlet character mod $N$},
\ite{$T_m$}     {the $m^{\text{th}}$ Hecke operator},
\ite{$U_m$}     {the linear operator that maps $\sum_{n\geq 0}a_nq^n$ to $\sum_{n\geq 0}a_{mn}q^n$},
\ite{$w_\ell$}      {the Atkin-Lehner operator at $\ell$},
\ite{$W_\ell$}      {the scaled Atkin-Lehner operator with $W_\ell=\ell^{k/2}w_\ell$},
\ite{$i_e$}     {the embedding map that sends $f(z)$ to $f(ez)$},
\ite{$\D_\ell$}     {the operator $\D_\ell=\ell^2U_\ell^2-\ell^k$},
\ite{$\D_\ell^{\chi}$} {the operator $\D_\ell^{\chi}=\ell^2U_\ell^2-\chi'(\ell)\ell^k$ when $\chi$ is induced by $\chi'$ mod $N/\ell$},
\ite{$\Tr_\ell$}    {the trace operator from level $N$ to level $N/\ell$},
\ite{$\langle a\rangle$} {the diamond operator evaluated at $a$},
\ite{$h_N$}     {the level $N$ Fricke involution},
\ite{$H_N$}    {the scaled Fricke involution $H_N=N^{k/2}h_N$},
\ite{$\M_k(\Gamma)$}    {space of weight $k$ modular forms with respect to $\Gamma$},
\ite{$\M(\Gamma)$}    {algebra of modular forms with respect to $\Gamma$},
\ite{$\mcS_k(\Gamma)$}    {space of weight $k$ cusp forms with respect to $\Gamma$},
\ite{$\mcS(\Gamma)$}    {algebra of cusp forms with respect to $\Gamma$},
\ite{$\M_k(N)$}    {space of weight $k$ modular forms for $\Gamma_0(N)$},
\ite{$\mcS_k(N)$}    {space of weight $k$ cusp forms for $\Gamma_0(N)$},
\ite{$\M_k(N,B)$}    {space of weight $k$ modular forms for $\Gamma_0(N)$ with coefficients in $B$},
\ite{$\mcS_k(N,B)$}    {space of weight $k$ cusp forms for $\Gamma_0(N)$ with coefficients in $B$},
\ite{$\M_k^{\chi}(N)$}  {space of weight $k$ level $N$ modular forms with character $\chi$},
\ite{$\mcS_k(N)^{\new}$}    {space of new forms for $\mcS_k(N)$},
\ite{$\mcS_k(N)^{\old}$}    {space of old forms for $\mcS_k(N)$},
\ite{$\mcS_k(N)^{\ell-\new}$}    {space of $\ell$-new forms for $\mcS_k(N)$},
\ite{$\mcS_k(N)^{\ell-\old}$}    {space of $\ell$-old forms for $\mcS_k(N)$},
\ite{$\E_k(N)$}     {the Eisenstein subspace of $\M_k(N)$},

\ite{$\mcS_k(N)^{U_\ell-\new}$}      {elements of $\mcS_k(N)$ in the kernel of $\D_\ell$},
\ite{$\mcS_k(N)^{\Tr_\ell-\new}$}        {elements of $\mcS_k(N)$ in the kernel of $\Tr_\ell$ and $\Tr_\ell W_\ell$},
\ite{$\mcS_k^{\chi}(N)^{U_\ell^{\chi}-\new}$}      {elements of $\mcS_k^{\chi}(N)$ in the kernel of $\D_\ell^{\chi}$},
\ite{$\mcS_k^{\chi}(N)^{\Tr_\ell^{\chi}-\new}$}  {elements of $\mcS_k^{\chi}(N)$ in the kernel of $\Tr_\ell$ and $\Tr_\ell W_\ell$},
\ite{$\mcS_k^{\chi}(N)^{\Tr_\ell^{\chi}-\new'}$}  {elements of $\mcS_k^{\chi}(N)$ in the kernel of $\Tr_\ell$ and $\Tr_\ell H_N$}.
\end{list}

\clearpage

%% file: references.bib
@book{diamond2005first,
  title={A first course in modular forms},
  author={Diamond, Fred and Shurman, Jerry},
  year={2005},
  publisher={Springer}
}

@book{kilford2008modular,
  title={Modular Forms: A classical and computational introduction},
  author={Kilford, Lloyd James Peter},
  year={2008},
  publisher={World Scientific}
}

@article{ao1970hecke,
  title={Hecke operators on $\Gamma_0(m)$},
  author={AO, L Atkin and Lehner, J},
  journal={Math. Ann},
  volume={185},
  pages={134--160},
  year={1970}
}

@article{li1975newforms,
  title={Newforms and functional equations},
  author={Li, Wen-Ch'ing Winnie},
  journal={Mathematische Annalen},
  volume={212},
  number={4},
  pages={285--315},
  year={1975},
  publisher={Springer}
}

@book{gantmakher1959theory,
  title={The theory of matrices},
  author={Gantmacher, Feliks Ruvimovich},
  year={1959},
  publisher={Chelsea Publishing Company}
}

@misc{carlton1999result,
    title={On a Result of Atkin and Lehner},
    author={David Carlton},
    year={1999},
    eprint={math/9903131},
    archivePrefix={arXiv},
    primaryClass={math.NT}
}

@article{ramanujan1916certain,
  title={On certain arithmetical functions},
  author={Ramanujan, Srinivasa},
  journal={Trans. Cambridge Philos. Soc},
  volume={22},
  number={9},
  pages={159--184},
  year={1916}
}

@article{deo2019newforms,
  title={Newforms mod $p$ in squarefree level with applications to Monsky’s Hecke-stable filtration},
  author={Deo, Shaunak and Medvedovsky, Anna},
  journal={Transactions of the American Mathematical Society, Series B},
  volume={6},
  number={8},
  pages={245--273},
  year={2019}
}

@inproceedings{diamond1995modular,
  title={Modular forms and modular curves},
  author={Diamond, Fred and Im, John},
  booktitle={Seminar on Fermat’s Last Theorem},
  publisher={American Math. Soc.},
  pages={39--133},
  year={1995}
}

@book{miyake2006modular,
  title={Modular forms},
  author={Miyake, Toshitsune},
  year={2006},
  publisher={Springer Science \& Business Media}
}

@article{prasanna2009arithmetic,
  title={Arithmetic properties of the Shimura--Shintani--Waldspurger correspondence},
  author={Prasanna, Kartik},
  journal={Inventiones mathematicae},
  volume={176},
  number={3},
  pages={521--600},
  year={2009},
  publisher={Springer}
}

@incollection{serre1973formes,
  title={Formes modulaires et fonctions z{\^e}ta p-adiques},
  author={Serre, Jean-Pierre},
  booktitle={Modular functions of one variable III},
  pages={191--268},
  year={1973},
  publisher={Springer}
}

@article{deligne1974conjecture,
  title={La conjecture de Weil. I},
  author={Deligne, Pierre},
  journal={Publications Math{\'e}matiques de l'Institut des Hautes {\'E}tudes Scientifiques},
  volume={43},
  number={1},
  pages={273--307},
  year={1974},
  publisher={Springer}
}

@book{atiyahintroduction,
  title={Introduction to Commutative Algebra},
  author={Atiyah, M and MacDonald, IG},
  year={1969},
  publisher={Addison-Wesley}
}

@article{mcgraw2003modular,
  title={Modular form congruences and Selmer groups},
  author={McGraw, William J and Ono, Ken},
  journal={Journal of the London Mathematical Society},
  volume={67},
  number={2},
  pages={302--318},
  year={2003},
  publisher={Cambridge University Press}
}

@phdthesis{weisinger1977some,
  title={Some results on classical Eisenstein series and modular forms over function fields},
  author={Weisinger, James Roger},
  year={1977},
  school={Harvard University}
}

@book{stein2007modular,
  title={Modular forms, a computational approach},
  author={Stein, William A},
  year={2007},
  publisher={American Math. Soc.}
}

@book{lang1976introduction,
  title={Introduction to modular forms},
  author={Lang, Serge},
  year={1976},
  publisher={Springer Verlag}
}

@book{fuchs2001modules,
  title={Modules over non-Noetherian domains},
  author={Fuchs, L{\'a}szl{\'o} and Salce, Luigi},
  series={Mathematical Surveys and Monographs},
  volume={84},
  year={2001},
  publisher={American Mathematical Soc.}
}

@book{goren2002lectures,
  title={Lectures on Hilbert modular varieties and modular forms},
  author={Eyal Zvi Goren},
  series={CRM monograph series},
  number={14},
  year={2002},
  publisher={American Mathematical Soc.}
}

@article{diamond1991congruence,
  title={Congruence primes for cusp forms of weight k> 2},
  author={Diamond, Fred},
  journal={Ast{\'e}risque},
  number={196-97},
  pages={205--213},
  year={1991}
}

@inproceedings{ribet1983congruence,
  title={Congruence relations between modular forms},
  author={Ribet, Kenneth A},
  booktitle={Proceedings of the International Congress of Mathematicians},
  volume={1},
  number={2},
  pages={503--514},
  year={1983},
  organization={Warszawa}
}

@article{wiles1995modular,
  title={Modular elliptic curves and Fermat's last theorem},
  author={Wiles, Andrew},
  journal={Annals of mathematics},
  volume={141},
  number={3},
  pages={443--551},
  year={1995}
}

@article{taylor1995ring,
  title={Ring-theoretic properties of certain Hecke algebras},
  author={Taylor, Richard and Wiles, Andrew},
  journal={Annals of Mathematics},
  pages={553--572},
  year={1995},
  volume={141},
  number={3}
}

@article{breuil2001modularity,
  title={On the modularity of elliptic curves over Q: wild 3-adic exercises},
  author={Breuil, Christophe and Conrad, Brian and Diamond, Fred and Taylor, Richard},
  journal={Journal of the American Mathematical Society},
  volume={14},
  number={4},
  pages={843--939},
  year={2001},
}

@article{katz1973p,
  title={p-Adic properties of modular forms and modular schemes},
  author={Katz, N},
  journal={Modular functions of one variable, III (Proc. Internat. Summer School, Univ. Antwerp, Antwerp, 1972), LNM},
  volume={350},
  pages={69--190},
  year={1973}
}

@article{jochnowitz1982congruences,
  title={Congruences between systems of eigenvalues of modular forms},
  author={Jochnowitz, Naomi},
  journal={Transactions of the American Mathematical Society},
  volume={270},
  number={1},
  pages={269--285},
  year={1982}
}

@misc{monsky2015hecke,
    title={A Hecke algebra attached to mod 2 modular forms of level 3},
    author={Monsky, Paul},
    year={2015},
    eprint={1508.07523},
    archivePrefix={arXiv},
    primaryClass={math.NT}
}

@misc{monsky2016hecke,
    title={A Hecke algebra attached to mod 2 modular forms of level 5},
    author={Monsky, Paul},
    year={2016},
    eprint={1610.07058},
    archivePrefix={arXiv},
    primaryClass={math.NT}
}

@online{loeffler2020integrality,
    title={Integrality of Atkin-Lehner operator for $\Gamma_1(N)$},
    author={David Loeffler},
    note={Math Overflow},
    url={http://mathoverflow.net/q/353743},
    urldate = {2020-04-26}
}

@online{sutherland2015number,
    title={Number Theory I},
    author={Andrew Sutherland},
    note={Lecture 17},
    url={https://math.mit.edu/classes/18.785/2015fa/LectureNotes17.pdf},
    year={2015}
}

@book{cohen2017modular,
  title={Modular forms: A Classical Approach},
  author={Cohen, Henri and Str{\"o}mberg, Fredrik},
  volume={179},
  year={2017},
  publisher={American Mathematical Soc.}
}
